\RequirePackage{fix-cm}
\documentclass[smallextended]{svjour3opt} 
\usepackage{etex}
\usepackage{todonotes}
\usepackage{makecell}
\usepackage{palatino,url}
\usepackage{verbatim, amssymb,amsmath}
\usepackage{varwidth}
\usepackage[export]{adjustbox}
\usepackage{multirow}
\usepackage{mathrsfs}
\usepackage{graphics,amsopn,amstext,amsfonts,color}
\usepackage{bm}
\usepackage{setspace}
\usepackage{subfig}
\usepackage{graphicx}
\usepackage{epstopdf}
\usepackage{caption}
\usepackage{multirow}
\usepackage{booktabs}
\usepackage{paralist}
\usepackage{colortbl}
\usepackage{appendix}
\usepackage{soul,caption,subfig}
\usepackage[hidelinks,bookmarksnumbered=true]{hyperref}
\usepackage{graphics,latexsym,amsfonts,verbatim,amsmath}
\usepackage{algorithm}
\usepackage{algorithmicx}
\usepackage{algpseudocode}
\usepackage{tikz}
\usetikzlibrary{calc}
\usetikzlibrary{plotmarks}
\usetikzlibrary{backgrounds}
\usepackage{pgfplots}
\usepackage{bbm}

\usepackage{pgf}
\usepackage{pst-func}

\definecolor{pinegreen}{rgb}{0.0, 0.47, 0.44}
\usepackage{pgfpages}
\usepackage{hhline} 
\usepackage{empheq}
\usepackage{cases}

\usepackage{multirow,todonotes}

\usepackage{epstopdf}

\usepackage{xcolor}

\usepackage{todonotes}
\usepackage{bm}
\usepackage{tikz}
\usepackage{pgfplots}
\usepgfplotslibrary{fillbetween}
\pgfplotsset{compat=newest}
\pgfdeclareplotmark{mystar}{
\node[star,star point ratio=2.25,minimum size=6pt,
inner sep=0pt,draw=black,solid,fill=red] {};
}
\usetikzlibrary{decorations.pathreplacing}

\usepgflibrary{arrows}

\usetikzlibrary{calc}
\usetikzlibrary{plotmarks}
\usepackage{wrapfig}

\usetikzlibrary{shapes,decorations,arrows,calc,arrows.meta,fit,positioning,patterns}
\tikzset{
-Latex,auto,node distance =1 cm and 1 cm,semithick,
state/.style ={ellipse, draw, minimum width = 0.7 cm},
point/.style = {circle, draw, inner sep=0.04cm,fill,node contents={}},
bidirected/.style={Latex-Latex,dashed},
el/.style = {inner sep=2pt, align=left, sloped}
}

\def\E{{\mathbb E}}

\def\Pr{{\mathbb{P}}}

\def\Re{\mathbb{R}}

\def\Ne{\mathbb{N}}
\def\I{\mathbb{I}}
\def\Be{\mathcal{B}}

\def\hat{\widehat}

\def \P{\mathcal{P}}
\def \F{\mathcal{F}}

\def\one{\mathbbm{1}}

\def\F{{\mathcal F}}

\def\P{{\mathcal P}}

\def\X{{\mathcal X}}

\def\Re{{\mathbb R}}
\def\Te{{\mathcal T}}

\newcommand{\rsigma}{\bm{\mathrm{{\Sigma}}}}

\def\CVaR{{\mathrm{CVaR}}}
\def\VaR{{\mathrm{VaR}}}

\def\alsox{{\mathrm{ALSO}{\text-}\mathrm{X}}}
\def\alsoxt{{\mathbf{ALSO}{\text-}\mathbf{X}}}

\def\alsoxs{{\mathrm{ALSO}{\text-}\mathrm{X}\#}}
\def\alsoxus{{\mathrm{ALSO}{\text-}\mathrm{X}\underline{\#}}}

\def\aus{A\underline{\#}}
\def\as{A\#}

\newcommand{\trxi}{\tilde{\bm{\xi}}}
\newcommand{\trzeta}{\tilde{\bm{\zeta}}}
\newcommand{\tzeta}{\tilde{{\zeta}}}
\newcommand{\rzeta}{{\bm{\zeta}}}

\DeclarePairedDelimiter\floor{\lfloor}{\rfloor}
\DeclareMathOperator{\sign}{sign}

\DeclareMathOperator{\supp}{supp}

\DeclareMathOperator{\Diag}{Diag}

\renewcommand{\S}{\mathcal{\bm S}}

\newcommand{\rxi}{\bm{\xi}}
\usepackage{todonotes,bm}

\usepackage{comment}

\usepackage{enumitem}
\usepackage{bigints}

\def\x{\vect{x}}
\usepackage{thmtools}
\usepackage{thm-restate}

\renewcommand{\S}{\mathcal{S}}

\DeclareMathOperator*{\argmin}{argmin}

\newcommand*{\QEDA}{\hfill\ensuremath{\square}}
\newcommand*{\QEDB}{\hfill\ensuremath{\diamond}}

\allowdisplaybreaks

\definecolor{mygreen}{RGB}{28,172,0} 
\definecolor{mypurple}{RGB}{170,55,241}
\definecolor{myorange}{rgb}{0.77,0.38,0.06}

\usepackage[numbers,sort&compress,comma]{natbib}
\allowdisplaybreaks

\title{ $\alsoxt\#$: Better Convex Approximations for Distributionally Robust Chance Constrained Programs}
\date{\today}
\titlerunning{$\alsoxs$}

\author{Nan Jiang \and Weijun Xie}

\institute{First Author: Nan Jiang \at
Affiliation: Georgia Institute of Technology, Atlanta, GA\\
\email{nanjiang@gatech.edu}
\and Corresponding Author: Weijun Xie \at
Affiliation: Georgia Institute of Technology, Atlanta, GA\\
\email{wxie@gatech.edu}
}

\allowdisplaybreaks

\makeatletter
\pgfmathdeclarefunction{erf}{1}{%
\begingroup
\pgfmathparse{#1 > 0 ? 1 : -1}%
\edef\sign{\pgfmathresult}%
\pgfmathparse{abs(#1)}%
\edef\x{\pgfmathresult}%
\pgfmathparse{1/(1+0.3275911*\x)}%
\edef\t{\pgfmathresult}%
\pgfmathparse{%
1 - (((((1.061405429*\t -1.453152027)*\t) + 1.421413741)*\t
-0.284496736)*\t + 0.254829592)*\t*exp(-(\x*\x))}%
\edef\y{\pgfmathresult}%
\pgfmathparse{(\sign)*\y}%
\pgfmath@smuggleone\pgfmathresult%
\endgroup
}
\makeatother

\begin{document}
\maketitle

\begin{abstract} 
This paper studies distributionally robust chance constrained programs (DRCCPs), where the uncertain constraints must be satisfied with at least a probability of a prespecified threshold for all probability distributions from the Wasserstein ambiguity set. As DRCCPs are often nonconvex and challenging to solve optimally, researchers have been developing various convex inner approximations. Recently, ALSO-X has been proven to outperform the conditional value-at-risk (CVaR) approximation of a regular chance constrained program when the deterministic set is convex. In this work, we relax this assumption by introducing a new ALSO-X\# method for solving DRCCPs. Namely, in the bilevel reformulations of ALSO-X and CVaR approximation, we observe that the lower-level ALSO-X is a special case of the lower-level CVaR approximation and the upper-level CVaR approximation is more restricted than the one in ALSO-X. This observation motivates us to propose the ALSO-X\#, which still resembles a bilevel formulation -- in the lower-level problem, we adopt the more general CVaR approximation, and for the upper-level one, we choose the less restricted ALSO-X. We show that ALSO-X\# can always be better than the $\CVaR$ approximation and can outperform ALSO-X under regular chance constrained programs and type $\infty-$Wasserstein ambiguity set.  We also provide new sufficient conditions under which ALSO-X\# outputs an optimal solution to a DRCCP. We apply the proposed ALSO-X\# to a wireless communication problem and numerically demonstrate that the solution quality can be even better than the exact method.
\end{abstract}

\vspace{0.2in}

\newpage
\section{Introduction}
\label{alsox_sha}
In this paper, we consider
a Distributionally Robust Chance Constrained Program (DRCCP) of form:
\begin{align}
{\rm (DRCCP)}\quad v^*=\min_{\bm x\in\X}\left\{\bm c^\top \bm x\colon \inf_{\Pr\in \P} \Pr\left\{\trxi\colon  \bm a_i(\bm x)^\top \trxi \leq  b_i(\bm x), \forall i\in[I] \right\}\geq 1-\varepsilon \right\}. \label{eq_drccp}
\end{align}
In a DRCCP, the objective is to minimize a linear objective function over a deterministic set $\X$ and an uncertain constraint system specified by possibly multiple linear constraints $ \bm a_i(\bm x)^\top \trxi \leq  b_i(\bm x)$ for all $i\in[I]$, where the uncertain constraints are required to be satisfied with probability $1-\varepsilon$ for any probability distribution $\Pr$ from an ambiguity set $\P$. Here, the scalar $\varepsilon\in(0,1)$ denotes a preset risk parameter and set $\P$ is formally defined as a subset of probability distributions $\Pr$ from a measurable space $(\Omega, \F)$ equipped with the sigma algebra $\F$ and induced by the random parameters $\trxi$ 
with support set $\Xi\subseteq \Re^m$. For each uncertain constraint $i\in[I]$, the affine mappings $\bm a_i(\bm x)$ and $b_i(\bm x)$ are defined as $\bm a_i(\bm x)={\bm{A}}_i \bm{x}+{\bm a}_i\in \Re^{m}$ with ${\bm{A}}_i\in \Re^{m\times n},{\bm a}_i\in \Re^m$ and $b_i(\bm x)={\bm{B}}_i^\top \bm{x}+{b}_i\in \Re$ with ${\bm{B}}_i\in \Re^{ n}, {b}_i\in \Re$.  When there is only $I=1$ uncertain constraint, problem \eqref{eq_drccp} is a \textit{single DRCCP} and otherwise, it is a \textit{joint DRCCP}. Notably, when the ambiguity set $\P$ is a singleton (i.e., $\P=\{\Pr\}$), DRCCP \eqref{eq_drccp} reduces to a regular Chance Constrained Program (CCP). 

\subsection{Wasserstein Ambiguity Set}
This paper studies the data-driven $q-$Wasserstein ambiguity set defined as
\begin{equation*}
\mathcal{P}_q =\left\{ \Pr\colon\Pr\left\{ \tilde{\rxi}\in {\Xi}\right\}=1,W_q(\Pr,\Pr_{\tilde\rzeta})\leq \theta \right\},
\end{equation*}
where for any $q\in[1,\infty]$, the $q-$Wasserstein distance is  
\begin{equation*}
W_q(\Pr_1,\Pr_2)=\inf\left\{\left[ \int_{{\Xi}\times{\Xi}}\left\|\bm{\xi}^1-\bm{\xi}^2\right\|^q\mathbb{Q}(d\bm{\xi}^1,d\bm{\xi}^2)\right ]^\frac{1}{q}\colon
\begin{aligned}
& \mathbb{Q} \text{ is a joint distribution of } \trxi^1 \text{ and } \trxi^2\\
& \text{ with marginals }\Pr_1 \text{ and } \Pr_2, \text{ respectively }
\end{aligned}
\right\},
\end{equation*}
 $\theta\geq0$ is the Wasserstein radius, and $\Pr_{\trzeta}$ denotes the reference distribution induced by random parameters $\trzeta$. Recently, there are many exciting works on DRCCP under type $q-$Wasserstein ambiguity set 
\cite{xie2021distributionally,chen2022data,kuccukyavuz2022chance,jiang2022also,ji2021data,ho2022distributionally,shen2022chance,chen2022approximations}. Particularly, according to the equivalent reformulation in proposition 8 of \cite{jiang2022also}, we write DRCCP \eqref{eq_drccp} under type $\infty-$Wasserstein ambiguity set  as 
\begin{align}
v^*=\min_{\bm x\in\X}\left\{\bm c^\top \bm x\colon \Pr\left\{\trzeta\colon  \theta \left\| \bm a_i(\bm x) \right\|_*+ \bm a_i(\bm x)^\top \trzeta \leq  b_i(\bm x), \forall i\in[I] \right\}\geq 1-\varepsilon \right\}. \label{eq_drccp_infty}
\end{align}

Throughout the paper, we make the following assumption:
\begin{enumerate}[label={A\arabic*}]
\setcounter{enumi}{0}
\item \label{A_1} The reference distribution $\Pr_{\trzeta}$ is sub-Gaussian, that is, $\Pr_{\trzeta}\{{\trzeta}:\|\trzeta\|\geq \tau\}\leq C_1\exp(-C_2\tau^2)$ for some positive constants $C_1,C_2$.
\end{enumerate}
It is worth noting that the sub-Gaussian assumption ensures the weak compactness of the Wasserstein ambiguity set and thus enjoys the strong duality of reformulating the worst-case expectation under type $q-$Wasserstein ambiguity set. Particularly, this paper mainly focuses on empirical or elliptical reference distributions, which clearly satisfy Assumption~\ref{A_1}.



\subsection{Relevant Literature}
Distributionally robust chance constrained programs (DRCCPs) have gained much attention recently
when the knowledge about the probability distribution is limited (see, e.g., \cite{hanasusanto2015distributionally,hanasusanto2017ambiguous,xie2018deterministic,zymler2013distributionally,chen2022data,xie2021distributionally,ho2022distributionally,kuccukyavuz2022chance,jiang2022also,ji2021data,chen2022approximations,shen2022chance,ho2022strong,shen2021convex}). As DRCCPs' feasible regions are often nonconvex, some existing research has worked on identifying conditions under which the feasible region in DRCCP \eqref{eq_drccp}  is convex (see, e.g., \cite{prekopa2013stochastic,lagoa2005probabilistically,ghaoui2003worst,calafiore2006distributionally,shapiro2021lectures,shen2021convex,cheng2014distributionally,xie2018deterministic,hanasusanto2015distributionally,li2019ambiguous}). For a single DRCCP \eqref{eq_drccp}, the authors in \cite{shen2021convex} showed that its feasible region is convex if the reference distribution is Gaussian under type $1-$Wasserstein ambiguity set. Similar convexity results apply to a single DRCCP when the ambiguity set $\P$ comprises all probability distributions with known first and second moments (see, e.g., \cite{ghaoui2003worst,calafiore2006distributionally}), known support of $\trxi$ (see, e.g., \cite{cheng2014distributionally}), arbitrary convex mapping of $\trxi$  (see, e.g., \cite{xie2018deterministic}), or the unimodality property of $\Pr$ (see, e.g., \cite{hanasusanto2015distributionally,li2019ambiguous}). 
Researchers have also proposed convex inner approximations of the nonconvex chance constraint  (e.g., \cite{nemirovski2006scenario,calafiore2006scenario,nemirovski2007convex,ahmed2017nonanticipative,chen2022approximations,jiang2022also}).  For example, the well-known conditional value-at-risk ($\CVaR$) approximation is to replace the chance constraint in DRCCP \eqref{eq_drccp} with the more conservative $\CVaR$ constraint (see the details in \cite{nemirovski2006scenario}). Recently, $\alsox$, a convex approximation method proposed in \cite{ahmed2017nonanticipative}, has been  proven to outperform the $\CVaR$ approximation of a regular chance constrained program (see, e.g., theorem 1 in \cite{jiang2022also}). 

Despite the challenges, DRCCPs are effective in decision-making under uncertainty and have been applied to a wide range of problems, including portfolio optimization \cite{pagnoncelli2009computational,pagnoncelli2012risk}, energy management \cite{shi2018distributionally,cai2009community}, supply chain management \cite{gupta2000mid,guillen2010global}, facility location problems \cite{lejeune2016solving}, 
 and wireless communication network \cite{atawia2016joint,wang2014outage,li2019coping,li2022maximizing}.  For example, chance constraints have been used in the design and optimization of wireless communication networks to ensure that the probability of certain operational constraints being violated, such as capacity limits or reliability requirements, is within an acceptable limit (see, e.g., \cite{li2019coping,li2022maximizing}). In portfolio optimization, the objective is to maximize the expected return of the portfolio while ensuring that the probability of portfolio losses does not exceed a specified level (see, e.g., \cite{pagnoncelli2009computational,pagnoncelli2012risk}). We refer interested readers to \cite{ahmed2008solving} for more applications. For a comprehensive review of DRCCPs, interested readers are referred to a recent survey from 
\cite{kuccukyavuz2022chance}.
 

\subsection{Contributions}

In this paper, we study a new method, 
termed ``$\alsoxs$," which advances the recent $\alsox$ in \cite{jiang2022also} in the following three main aspects: (a) for any closed deterministic set $\X$, $\alsoxs$ is always better than $\CVaR$ approximation under any ambiguity set; (b) $\alsoxs$ can be better than $\alsox$; and (c) $\alsoxs$ admits new conditions under which its output solution is also optimal to DRCCPs. More specifically, 
\begin{enumerate}[label=(\roman*)]
\item We prove that $\alsoxs$ is better than $\CVaR$ approximation under a general ambiguity set (beyond Wasserstein ambiguity set) and a closed deterministic set $\X$. This result significantly improves that of theorem 1 in \cite{jiang2022also}), which relies on the convexity of the deterministic set $\X$;
\item We show that under type $\infty-$ Wasserstein ambiguity set, $\alsoxs$ is better than $\alsox$ when the lower-level $\alsox$ admits a unique solution. When the reference distribution is constructed by i.i.d. samples from a continuous nondegerate distribution, or the reference distribution is continuous and nondedegerate, the lower-level $\alsox$ indeed presents a unique solution; 
\item We present new sufficient conditions under which $\alsoxs$ yields an optimal solution to a DRCCP. For example, one sufficient condition is that for a binary DRCCP with an empirical reference distribution; and 
\item  We extend the afromentioned results of $\alsoxs$ to solve a DRCCP \eqref{eq_drccp} under type $q-$Wasserstein ambiguity set with $q\in[1,\infty)$. 
\end{enumerate}

\noindent\textbf{Organization.}
The remainder of the paper is organized as follows. Section~\ref{sec_alsox_sharp_intro} reviews $\alsox$ and $\CVaR$ approximation and introduces the 
$\alsoxs$. Section~\ref{sec_better} shows that $\alsoxs$ is better than $\alsox$ and $\CVaR$ approximation. Section~\ref{sec_unique_alsox}
explores conditions under which $\alsoxs$ is better than $\alsox$.
Section~\ref{sec_exactness} provides the conditions under which $\alsoxs$ outputs an exact optimal solution. 
Section~\ref{sec_extension} extends $\alsoxs$ to solve DRCCPs under type $q-$Wasserstein ambiguity set with $q\in[1,\infty)$. 
Section~\ref{numerical_study} numerically illustrates the proposed methods. Section~\ref{sec_conclusion} concludes the paper.

\noindent\textbf{Notation.} The following notation is used throughout the paper. We use bold letters (e.g., $\bm {x},\bm {A}$) to denote vectors and matrices and use corresponding non-bold letters to denote their components. Given a vector or matrix $\bm{x}$, its zero norm $ \|\bm{x}\|_0$ denotes the number of its nonzero elements. We let $\|\cdot\|_*$ denote the dual norm of a general norm $\|\cdot\|$. 
Given an integer $n$, we let $[n]:=\{1,2,\cdots,n\}$ and use $\Re_+^n:=\{\bm {x}\in \Re^n:x_i\geq0, \forall i\in [n]\}$. Given a real number $\tau$, let $(\tau)_+:=\max\{\tau,0\}$. Given a finite set $I$, let $|I|$ denote its cardinality. We let $\trxi$ denote a random vector and denote its realizations by $\rxi$. Given a vector $\bm{x}\in \Re^n$, let $\mathrm{supp}(\bm{x})$ be its support, i.e., $\mathrm{supp}(\bm{x}):=\{i\in [n]: x_i\neq0\}$. Given a probability distribution $\Pr$ on $\Xi$, we use $\Pr\{A\}$ to denote $\Pr\{\rxi:\text{condition} \ A(\rxi) \ \text{holds}\}$ when $A(\rxi)$ is a condition on $\rxi$, and to denote $\Pr\{\rxi\colon \rxi \in A\}$ when $A \subseteq \Xi$ is $\Pr-$measurable. 
We use $\floor{x}$ to denote the largest integer $y$ satisfying $y\leq x$, for any $x\in \Re$. We use the phrase ``Better Than" to indicate ``at least as good as." Additional notations will be introduced as needed.

\section{$\alsoxs$}
\label{sec_alsox_sharp_intro}
In this section, we first review two convex approximations of DRCCP, $\alsox$ and $\CVaR$ approximation. Then we present $\alsoxs$ for solving DRCCP \eqref{eq_drccp_infty} and show its connections to $\alsox$ and $\CVaR$ approximation. To begin with, we first introduce the notions of $\VaR_{1-\varepsilon}(\cdot)$ and $\CVaR_{1-\varepsilon}(\cdot)$. For a given risk parameter $\varepsilon$ and a given random variable $\tilde{\bm{X}}$ with probability distribution $\Pr$ and cumulative distribution function $F_{\tilde{\bm{X}}}(\cdot)$, $(1-\varepsilon)$ Value-at-Risk ($\VaR$) of $\tilde{\bm{X}}$ is defined as
\begin{align*}
\VaR_{1-\varepsilon}(\tilde{\bm{X}}):=\min_s \left\{ s: F_{\tilde{\bm{X}}}(s)\geq 1-\varepsilon\right\},
\end{align*}
and the corresponding Conditional Value-at-Risk ($\CVaR$) is 
\begin{align*}
\CVaR_{1-\varepsilon}(\tilde{\bm{X}}):=\min_\beta \left\{ \beta+\frac{1}{\varepsilon}\E_{\Pr}[ \tilde{\bm{X}}-\beta ]_+\right\}.
\end{align*}

\subsection{State-of-the-art Convex Approximations}
\label{section_convex_appro}
In general, solving DRCCP \eqref{eq_drccp_infty} is NP-hard (see, e.g., \cite{xie2020bicriteria}). Thus, in this work, instead of solving DRCCP \eqref{eq_drccp_infty} directly, we review two known convex approximations, i.e., the popular $\CVaR$ approximation and the recent $\alsox$. 

The $\alsox$ method with a bilevel structure can be adapted to solve DRCCP \eqref{eq_drccp_infty}.  In the lower-level $\alsox$, we solve the hinge-loss approximation with a given objective upper bound. We then check whether its optimal solution $\bm{x}^*$ satisfies the worst-case chance constraint or not. 
The upper-level $\alsox$ is to search the best upper bound of the objective value.
Formally, $\alsox$ admits the form:
\begin{align*}
v^A =\min _{ {t}}\quad & t,\\
\textup{s.t.}\quad&\bm x^*\in\argmin _{\bm {x}\in \X}\, \sup _{\Pr\in\P_\infty} \left\{ \E_\Pr\left[\max_{i\in [I]}\left(\bm a_i(\bm x)^\top \trxi-b_i(\bm x)\right)_+\right]\colon \bm{c}^\top \bm{x} \leq t \right\},\\
& \inf_{\Pr\in \P_\infty} \Pr\left\{\trxi\colon  \bm a_i(\bm x^*)^\top \trxi \leq  b_i(\bm x^*), \forall i\in[I] \right\}\geq 1-\varepsilon.
\end{align*}
Based on the equivalent reformulation in proposition 9 of \cite{jiang2022also}, we consider the following $\alsox$:
\begin{subequations}
\label{drccp_alsox_formulation}
\begin{align}
v^A =\min _{ {t}}\quad & t,\\
\textup{s.t.}\quad&\bm x^*\in\argmin _{\bm {x}\in \X} \left\{ \E_{\Pr_{\trzeta}}\left[\max_{i\in [I]}\left(\theta \left\| \bm a_i(\bm x) \right\|_*+\bm a_i(\bm x)^\top \trzeta-b_i(\bm x)\right)_+\right]\colon \bm{c}^\top \bm{x} \leq t \right\},\label{drccp_alsox}\\
&  \Pr\left\{\trzeta\colon  \theta \left\| \bm a_i(\bm x^*) \right\|_*+ \bm a_i(\bm x^*)^\top \trzeta \leq  b_i(\bm x^*), \forall i\in[I] \right\}\geq 1-\varepsilon. \label{alsox_drccp_formualtion}
\end{align}
\end{subequations} 
Under type $\infty-$Wasserstein ambiguity set $\P_\infty$, it has been shown that when the deterministic set $\X$ is convex, $\alsox$ is better than $\CVaR$ approximation (see, e.g., theorem 7 in \cite{jiang2022also}). 
However, this result, in general, does not hold for a DRCCP when set $\X$ is nonconvex (see example 2 in \cite{jiang2022also} with Wasserstein radius $\theta=0$). For notational convenience,
let us denote $v^A(t)$ and $\hat {F}(\bm x)$ to be the optimal value and the objective function  of the lower-level $\alsox$ \eqref{drccp_alsox}, respectively. 

The $\CVaR$ approximation has been shown to work quite well for solving DRCCPs (see, e.g., \cite{chen2022data,xie2021distributionally}). For DRCCP \eqref{eq_drccp_infty}, its $\CVaR$ approximation is defined by replacing the worst-case chance constraint by the worst-case $\CVaR$ constraint as below
\begin{equation*}
v^{\CVaR}=\min_{\bm x\in\X}\left\{\bm c^\top \bm x \colon \sup _{{\Pr}\in\P_\infty} \inf_{\beta\leq 0}\left[ \beta+ \frac{1}{\varepsilon}\E_{\Pr}\left[ \left( \max_{i\in[I]} \left(\bm a_i(\bm x)^\top \trxi-b_i(\bm{x}) \right)-\beta \right)_+ \right] \right] \leq 0 \right\}.
\end{equation*}
From the equivalent reformulation in proposition 9 of \cite{jiang2022also}, we consider the following $\CVaR$ approximation of DRCCP \eqref{eq_drccp_infty}:
\begin{equation*}
v^{\CVaR}=\min_{\bm x\in\X}\left\{\bm c^\top \bm x \colon  \inf_{\beta\leq 0}\left[ \beta+ \frac{1}{\varepsilon}\E_{\Pr_{\trzeta}}\left[ \left( \max_{i\in[I]} \left( \theta \left\| \bm a_i(\bm x) \right\|_*+ \bm a_i(\bm x)^\top \trzeta-b_i(\bm{x}) \right)-\beta \right)_+ \right] \right] \leq 0 \right\}.
\end{equation*}
Equivalently, we also recast the $\CVaR$ approximation as a bilevel program, where the upper-level problem is to search the best objective value and the lower-level problem is to minimize the left-hand of $\CVaR$ constraint given that the objective function is upper-bounded
by a given value. That is,
\begin{subequations}
\label{drccp_cvar_formulation}
\begin{align}
v^{\CVaR}=\min _{ {t}}\, & t,\\
\textup{s.t.}\,&\left(\bm x^*,\beta^*\right)\in\argmin _{\begin{subarray}{c}\bm {x}\in \X, \bm{c}^\top \bm{x} \leq t,\\ \beta\leq 0 \end{subarray}}\left\{ \varepsilon\beta+ \E_{\Pr_{\trzeta}}\left[ \left( \max_{i\in[I]} \left( \theta \left\| \bm a_i(\bm x) \right\|_*+ \bm a_i(\bm x)^\top \trzeta-b_i(\bm{x}) \right)-\beta \right)_+ \right] \right\},\label{drccp_cvar}\\
&\varepsilon\beta^*+ \E_{\Pr_{\trzeta}}\left[ \left( \max_{i\in[I]} \left( \theta \left\| \bm a_i(\bm x^*) \right\|_*+ \bm a_i(\bm x^*)^\top \trzeta-b_i(\bm{x}^*) \right)-\beta^* \right)_+ \right]\leq 0.\label{cvar_drccp_formualtion}
\end{align}
\end{subequations} 
 We call the objective function in the lower-level $\CVaR$ approximation \eqref{drccp_cvar} ``$\CVaR$-loss." We observe that if we let variable $\beta=0$ in the lower-level $\CVaR$ approximation \eqref{drccp_cvar}, then we recover the hinge-loss approximation \eqref{drccp_alsox}. In other words,
the lower-level $\alsox$ \eqref{drccp_alsox} and the lower-level $\CVaR$ approximation \eqref{drccp_cvar} coincide when $\beta=0$. This observation inspires us to improve $\alsox$ by replacing its lower-level hinge-loss objective function with the $\CVaR$-loss approximation, which is elaborated in the subsequent subsections. 

It is worth noting that, when the deterministic set $\X$ is discrete, $\CVaR$ approximation can outperform  $\alsox$ when solving DRCCP \eqref{eq_drccp_infty}, as demonstrated by the following example.
\begin{example}
\label{cvar_better_nonconvex}
\rm
Consider a single DRCCP under type $\infty-$Wasserstein ambiguity set  with $\theta=1$. Assume that the empirical distribution has $4$ equiprobable scenarios (i.e., $N=4$, $\Pr\{\trzeta=\rzeta^i\}=1/N$), risk parameter $\varepsilon=1/2$, deterministic set $\X=\{0,1\}$, function $ \bm a_1(x)^\top \rzeta-b_1({x}) ={\zeta}_1 x - {\zeta}_2$, 
$\zeta_1^1=-48$, $\zeta_1^2=\zeta_1^3=\zeta_1^4=100$, $\zeta_2^1=-51$, and $\zeta_2^2=\zeta_2^3=\zeta_2^4=100$. In this example, DRCCP \eqref{eq_drccp_infty} resorts to
\begin{equation*}
v^*=\min_{x\in \{0,1\}}\left\{-x\colon  \I(49x\geq 50)+\I(101x\leq 99)+\I(101x\leq 99)+\I(101x\leq 99)\geq 2 \right\},
\end{equation*}
 where $\CVaR$ approximation returns the optimal solution and $\alsox$ fails to find any feasible solution (see, example 2 in \cite{jiang2022also}).  
\QEDB
\end{example}

\subsection{What is $\alsoxs$?}
\label{section_alsox_sharp_intro}
To overcome the limitations of $\alsox$ and $\CVaR$ approximation, we introduce the new ``$\alsoxs$.''  
As discussed in the previous subsection, 
the lower-level $\alsox$ \eqref{drccp_alsox} can be viewed as a special case of the lower-level $\CVaR$ approximation \eqref{drccp_cvar} by letting $\beta=0$. Thus, one may want to replace the hinge-loss objective function with the $\CVaR$-loss one. On the other hand, one disadvantage of the CVaR approximation \eqref{drccp_cvar_formulation} is that it relies on a more conservative $\CVaR$ constraint \eqref{cvar_drccp_formualtion} for the feasibility check.
To improve the $\CVaR$ approximation \eqref{drccp_cvar_formulation}, we can use chance constraint \eqref{alsox_drccp_formualtion} for the feasibility check, leading to $\alsoxs$, an integration of $\CVaR$ approximation and  $\alsox$.
Formally, $\alsoxs$ admits the form of
\begin{align*}
v^{\as} =\min _{ {t}}\quad & t,\\
\textup{s.t.}\quad &(\bm x^*,\beta^*)\in\argmin _{\bm {x}\in \X,\beta\leq 0}\, \sup _{\Pr\in\P_\infty} \left\{ \varepsilon\beta+ \E_{\Pr}\left[  \max_{i\in[I]} \left(\bm a_i(\bm x)^\top \trxi-b_i(\bm{x}) \right)-\beta \right]_+ \colon \bm{c}^\top \bm{x} \leq t \right\},\\
&  \inf_{\Pr\in \P_\infty} \Pr\left\{\trxi\colon  \bm a_i(\bm x^*)^\top \trxi \leq  b_i(\bm x^*), \forall i\in[I] \right\}\geq 1-\varepsilon.
\end{align*}

According to the reformulations in Section~\ref{section_convex_appro}, the $\alsoxs$ is equivalent to
\begin{subequations}
\label{eq_also_cvar_q_wass}
\begin{align}
v^{\as} =\min _{ {t}}\quad &  t, \label{eq_also_cvar_q_wass_a} \\
\textup{s.t.}\quad&\left(\bm x^*,\beta^*\right)\in\argmin _{\begin{subarray}{c}\bm {x}\in \X, \bm{c}^\top \bm{x} \leq t,\\ \beta\leq 0 \end{subarray}}\left\{ \varepsilon\beta+ \E_{\Pr_{\trzeta}}\left[ \left( \max_{i\in[I]} \left( \theta \left\| \bm a_i(\bm x) \right\|_*+ \bm a_i(\bm x)^\top \trzeta-b_i(\bm{x}) \right)-\beta \right)_+ \right] \right\},\label{eq_also_cvar_q_wass_b}\\
& \Pr\left\{\trzeta\colon  \theta \left\| \bm a_i(\bm x^*) \right\|_*+ \bm a_i(\bm x^*)^\top \trzeta \leq  b_i(\bm x^*), \forall i\in[I] \right\}\geq 1-\varepsilon.\label{eq_also_cvar_q_wass_c}
\end{align}
\end{subequations} 
In the proposed $\alsoxs$ \eqref{eq_also_cvar_q_wass}, given a current objective value $t$, the lower-level $\alsoxs$ \eqref{eq_also_cvar_q_wass_b} is to solve the $\CVaR$ approximation first and then the upper-level $\alsoxs$ is to check whether the lower-level solution satisfies the worst-case chance constraint \eqref{eq_also_cvar_q_wass_c} or not. 
In this way, we introduce the new convex approximation of DRCCP \eqref{eq_drccp_infty}, where we have $v^*\leq v^{\as}$.

Algorithm~\ref{alg_alsox_sharp} summarizes the solution procedure of $\alsoxs$ \eqref{eq_also_cvar_q_wass}, where 
for a given $t$ of the upper-level problem, we solve the lower-level $\CVaR$ approximation \eqref{eq_also_cvar_q_wass_b} with an optimal solution $(\bm{x}^*,\beta^*)$ and check whether $\bm x^*$ is feasible to DRCCP \eqref{eq_drccp_infty} or not, i.e., check if $\bm x^*$ satisfies \eqref{eq_also_cvar_q_wass_c} or not. If the answer is YES, we decrease the value of $t$; otherwise, increase it. In the implementation, we search the optimal $t$ by using the binary search method with a stopping tolerance $\delta_1$.
The implementation details follow similarly to algorithm 1 in \cite{jiang2022also} and the remarks therein.
\begin{algorithm}[htbp]
\caption{The Proposed $\alsoxs$ Algorithm}
\label{alg_alsox_sharp}
\begin{algorithmic}[1]
\State \textbf{Input:} Let $\delta_1$ denote the stopping tolerance parameter, $t_L$ and $t_U$ be the known lower and upper bounds of the optimal value of DRCCP \eqref{eq_drccp_infty}, respectively 
\While {$t_U-t_L>\delta_1$}
\State Let $t=(t_L+t_U)/2$ and $(\bm{x}^*,\beta^*)$ be an optimal solution of the lower-level $\alsoxs$ \eqref{eq_also_cvar_q_wass_b}
\State Let $t_L=t$ if $\bm x^*$ satisfies \eqref{eq_also_cvar_q_wass_c}; otherwise, $t_U=t$
\EndWhile
\State \textbf{Output:} A feasible solution $\bm x^*$ and its objective value $\bar v^{\as}$ to DRCCP \eqref{eq_drccp_infty} 
\end{algorithmic}
\end{algorithm}

Note that we inherit the constraint $\beta\leq 0$ in the lower-level $\alsoxs$ \eqref{eq_also_cvar_q_wass_b} from the $\CVaR$ approximation. One may want to relax this constraint and arrive at the following ``weak" formulation of $\alsoxs$, termed as ``$\alsoxus$:"
\begin{subequations}
\label{eq_also_cvar_q_wass_weak}
\begin{align}
v^{\aus} =\min _{ {t}}\,  & t, \label{eq_also_cvar_q_wass_weak_a} \\
\textup{s.t.}\,&\left(\bm x^*,\beta^*\right)\in\argmin _{\bm {x}\in \X, \bm{c}^\top \bm{x} \leq t,\beta}\left\{ \varepsilon\beta+ \E_{\Pr_{\trzeta}}\left[ \left( \max_{i\in[I]} \left( \theta \left\| \bm a_i(\bm x) \right\|_*+ \bm a_i(\bm x)^\top \trzeta-b_i(\bm{x}) \right)-\beta \right)_+ \right] \right\},\label{eq_also_cvar_q_wass_weak_b}\\
& \Pr\left\{\trzeta\colon  \theta \left\| \bm a_i(\bm x^*) \right\|_*+ \bm a_i(\bm x^*)^\top \trzeta \leq  b_i(\bm x^*), \forall i\in[I] \right\}\geq 1-\varepsilon.\label{eq_also_cvar_q_wass_weak_c}
\end{align}
\end{subequations} 
In our numerical study, we find that $\alsoxs$ consistently outperforms $\alsoxus$. 
The following example shows that $\alsoxs$ can be superior to $\alsoxus$, $\alsox$, $\alsox+$, and $\CVaR$ approximation.
We formally prove this result in the next section.
\begin{example}
\label{sharp_better_nonconvex}
\rm
Consider a single DRCCP under type $\infty-$Wasserstein ambiguity set  with $\theta=1$. Assume that the empirical distribution has $4$ equiprobable scenarios (i.e., $N=4$, $\Pr\{\trzeta=\rzeta^i\}=1/N$), risk parameter $\varepsilon=1/2$, deterministic set $\X=\{0,1\}$, function $ \bm a_1(x)^\top \rzeta-b_1({x}) ={\zeta}_1 x - {\zeta}_2$, 
$\zeta_1^1=-8$, $\zeta_1^2=\zeta_1^3=\zeta_1^4=3$, $\zeta_2^1=-25/2$, and $\zeta_2^2=\zeta_2^3=\zeta_2^4=5/2$. In this example, DRCCP \eqref{eq_drccp_infty} resorts to
\begin{equation*}
v^*=\min_{x\in \{0,1\}}\left\{-x\colon  \I\left(9x\geq \frac{23}{2}\right)+\I\left(4x\leq \frac{3}{2}\right)+\I\left( 4x\leq \frac{3}{2}\right)+\I\left( 4x\leq \frac{3}{2}\right)\geq 2 \right\}.
\end{equation*}
The weak formulation of $\alsoxs$ $v^{\aus} $  \eqref{eq_also_cvar_q_wass_weak} can  be written as
\begin{align*}
&v^{\aus} =\min _{ {t}}\, \Biggl\{t\colon  \sum_{i\in[4]}\I({s^*_i}>0) \leq 2,\\
&(x^*,\bm s^*,\beta^*)\in\argmin_{\begin{subarray}{c}
x\in \{0,1\}, \bm s, \beta\end{subarray}}\biggl\{ \frac{1}{4} \sum_{i\in[4]}s_i - \frac{1}{2}\beta\colon \begin{array}{l}
\displaystyle-9x+\frac{23}{2}\leq s_1, 4x-\frac{3}{2}\leq s_2,4x-\frac{3}{2}\leq s_3, \\
\displaystyle 4x-\frac{3}{2}\leq s_4,-x\leq t, s_i\geq \beta,\forall i\in[4]\end{array}\biggr\}
\Biggr\}.
\end{align*}
Particularly, for any $t\geq-1$, the  $\alsoxus$  returns a solution with $s_1^*=s_2^*=s_3^*=s_4^*=5/2$, $ x^*=1$. Since the support size of $\bm{s}^*$ is greater than $2$, we have to increase the objective bound $t$ to be infinite. However, if we enforce $\beta \leq 0$, i.e., consider the $\alsoxs$, we have $s_1^*=23/2,  s_2^*=s_3^*=s_4^*=-3/2, \beta^*=-3/2$, $x^*=0$.
Therefore, in this example, $\alsoxs$ always returns the optimal solution, but $\alsoxus$ fails to find any feasible solution. Notice that in this example, $\alsox$, $\CVaR$ approximation, and $\alsox+$ (see, e.g., section 4 in \cite{jiang2022also}) all fail to find any feasible solution.   
\QEDB
\end{example}
It is worthy of mentioning that albeit the formulation of $\alsoxus$ tends to be weaker, the unconstrained $\beta$ variable is useful to prove the strength of  $\alsoxs$ under an elliptical reference distribution. 

\section{$\alsoxs$ is Better Than $\alsox$ and $\CVaR$ Approximation}
\label{sec_better}

In this section, we prove that $\alsoxs$ is better than $\alsox$, $\alsoxus$, and $\CVaR$ approximation, respectively. 

\subsection{$\alsoxs$ is Better Than $\alsox$}
\label{sec_alsoxs_better_alsox}
Note that $\alsoxs$ can be viewed as an integration of  $\CVaR$ approximation and $\alsox$. In fact, if the lower-level $\alsox$ provides a feasible solution to DRCCP \eqref{eq_drccp_infty}, then at least one optimal solution of the lower-level $\alsoxs$ is also feasible to DRCCP \eqref{eq_drccp_infty}. Particularly, under type $\infty-$Wasserstein ambiguity set and for a given objective upper bound $t$ such that the optimal value of the lower-level $\alsox$ \eqref{drccp_alsox} is positive (i.e., $v^A(t)>0$), the lower-level $\alsox$ \eqref{drccp_alsox} has a unique optimal solution, which is feasible to DRCCP \eqref{eq_drccp_infty}. 
Then any optimal solution of the lower-level $\alsoxs$ is also feasible to DRCCP \eqref{eq_drccp_infty}.
\begin{theorem}
\label{also_cvar_better_alsox}
Suppose that for any objective upper bound $t$ such that $t\geq \min_{\bm{x}\in \X}\bm c^\top \bm x$ and $v^A(t)>0$, the lower-level $\alsox$ \eqref{drccp_alsox}  admits a unique optimal solution. Then $\alsoxs$ \eqref{eq_also_cvar_q_wass} is better than  $\alsox$ \eqref{drccp_alsox_formulation}, i.e., $v^{\as}\leq v^A$.
\end{theorem}
\begin{proof}
	It is sufficient to show that for any objective upper bound $t$, if the lower-level $\alsox$ \eqref{drccp_alsox} is feasible to DRCCP \eqref{eq_drccp_infty}, then the lower-level $\alsoxs$ \eqref{eq_also_cvar_q_wass_b} is also feasible to DRCCP \eqref{eq_drccp_infty}.
Let $(\hat{\bm x},\hat{\beta})$ denote an optimal solution from the lower-level $\alsoxs$ \eqref{eq_also_cvar_q_wass_b} and let $\bar{\bm x}$ denote an optimal solution from the lower-level $\alsox$ \eqref{drccp_alsox}. We split the proof into two steps by discussing $v^A(t)=0$ or $v^A(t)>0$.

\noindent\textbf{Step I.} When the optimal value of the lower-level $\alsox$ \eqref{drccp_alsox} is $v^A(t)=0$, since $v^A(t)$ is an upper bound of the lower-level $\alsoxs$ \eqref{eq_also_cvar_q_wass_b}. Thus, the optimal value of the lower-level $\alsox$ \eqref{drccp_alsox} is less than or equal to zero. Due to the fact that the objective value of $\alsoxs$ divided by $\varepsilon$ less than or equal to zero implies a conservative approximation of distributionally robust chance constraint \cite{nemirovski2007convex}, we must have that $\hat{\bm x}$ is feasible to DRCCP \eqref{eq_drccp_infty}.

\noindent\textbf{Step II.}  
Next, we consider the case when the optimal value from the lower-level $\alsox$ \eqref{drccp_alsox} is positive, i.e., $v^A(t)>0$.  We discuss two cases on whether $\hat\beta=0$ or not. \par

\noindent\textbf{Case I.}  If $\hat{\beta}=0$, then the lower-level $\alsoxs$ \eqref{eq_also_cvar_q_wass_b} and the lower-level $\alsox$ \eqref{drccp_alsox} coincide. Since the lower-level $\alsox$ \eqref{drccp_alsox} admits a unique optimal solution, we must have $\bar{\bm x}=\hat{\bm x}$. That is, $\alsoxs$ and $\alsox$ are the same. Thus, if $\bar{\bm x}$ is feasible to the DRCCP, i.e., $\bar{\bm x}$ satisfies \eqref{alsox_drccp_formualtion}, $\hat{\bm x}$ is also feasible to DRCCP \eqref{eq_drccp_infty}.

\noindent\textbf{Case II.} Suppose that $\hat{\beta}<0$. From the discussions in section 3 of \cite{rockafellar2000optimization}, we have
\begin{align*}
\VaR_{1-\varepsilon} \left\{ \theta \left\| \bm a_i(\hat{\bm x}) \right\|_*+ \bm a_i(\hat{\bm x})^\top \trzeta -  b_i(\hat{\bm x}), \forall i\in[I] \right\} <0,
\end{align*}
which implies that 
\begin{align*}
 \Pr\left\{\trzeta\colon  \theta \left\| \bm a_i(\hat{\bm x}) \right\|_*+ \bm a_i(\hat{\bm x})^\top \trzeta \leq  b_i(\hat{\bm x}), \forall i\in[I] \right\}\geq 1-\varepsilon. 
\end{align*}
Thus, $\alsoxs$ provides a feasible solution to DRCCP \eqref{eq_drccp_infty}. Since the solution of the lower-level $\alsox$ \eqref{drccp_alsox} may not be feasible to DRCCP \eqref{eq_drccp_infty}, $\alsoxs$ is better than $\alsox$. 
\QEDA
\end{proof}
We make the following remarks about Theorem~\ref{also_cvar_better_alsox}:
\begin{enumerate}[label=(\roman*)]
\item 
 Different from the work \cite{jiang2022also}, our proof does not require the convexity assumption of set $\X$; 
 \item The uniqueness condition is satisfied by many DRCCPs 
 as well as their regular counterparts, as formally proved in the next section; and
\item In Example~\ref{sharp_better_nonconvex}, the lower-level $\alsox$ has a unique optimal solution when $t\geq -1$, which, however, is not feasible to the DRCCP. On the contrary, the proposed $\alsoxs$ can find the optimal solution, which is better than $\alsox$ according to Theorem~\ref{also_cvar_better_alsox}.
\end{enumerate}

The uniqueness assumption in Theorem~\ref{also_cvar_better_alsox} is, in fact, necessary. Below is an example showing that $\alsoxs$ can be worse than $\alsox$. 
\begin{example}\rm
\label{example_worse_convex}
Consider a single DRCCP under type $\infty-$Wasserstein ambiguity set  with $\theta=1/2$ and $\|\cdot\|_*=\|\cdot\|_1$. Assume that the empirical distribution has $3$ equiprobable scenarios (i.e., $N=3$, $\Pr\{\trzeta=\rzeta^i\}=1/N$), risk parameter $\varepsilon=1/2$, deterministic set $\X=[0,10]^2$, function $ \bm a_1(\bm x)^\top \rzeta-b_1({\bm x}) =- \bm x^\top \rzeta +1$, 
$\rzeta^1=(5/2,7/2)^\top$, $\rzeta^2=(5/2,3/2)^\top$, and $\rzeta^3=(3/2,5/2)^\top$. In this example, DRCCP 
 \eqref{eq_drccp_infty} resorts to
\begin{equation*}
v^*=\min_{\bm x\in [0,10]^2}\left\{x_1+x_2\colon  \I(2x_1+3x_2\geq 1)+ \I(2x_1+x_2\geq 1)+ \I(x_1+2x_2\geq 1)\geq 2 \right\},
\end{equation*}
where the optimal value $v^*=1/2$. Its $\alsox$ counterpart admits the following form:
\begin{align*}
&v^{A} =\min _{ {t}}\, \biggl\{t\colon  \sum_{i\in[3]}\I({s^*_i}>0) \leq 2,\\
&(\bm x^*,\bm s^*)\in\argmin_{
\bm x\in [0,10]^2, \bm s\in \Re_+^3}\biggl\{ \frac{1}{3} \sum_{i\in[3]}s_i \colon \begin{array}{l}
\displaystyle 2x_1+3x_2\geq 1-s_1, 2x_1+x_2\geq 1-s_2,\\
\displaystyle x_1+2x_2\geq 1-s_3,x_1+x_2\leq t \end{array}\biggr\}
\biggr\}.
\end{align*}
When $t=1/2$, one optimal solution of the lower-level $\alsox$ is $x_1^*=0,x_2^*=1/2,s_1^*=0, s_2^* = 1/2, s_3^*=0$, which is feasible to the DRCCP. In this case, $\alsox$ can find an optimal solution of the DRCCP with $v^A=v^*=1/2$. However, when $t=1/2$, another optimal solution of the lower-level $\alsox$ is $x_1^*=1/4, x_2^*=1/4, s_1^*=0, s_2^* = 1/4, s_3^*=1/4$, which is infeasible to the DRCCP. Hence, it violates the uniqueness assumption. 

Now let us consider corresponding $\alsoxs$:
\begin{align*}
&v^{\as} =\min _{ {t}}\, \biggl\{t\colon  \sum_{i\in[3]}\I({s^*_i}>0) \leq 2,\\
&(\bm x^*,\bm s^*,\beta^*)\in\argmin_{
\bm x\in [0,10]^2, \bm s,\beta\leq 0}\biggl\{ \frac{1}{3} \sum_{i\in[3]}s_i -\frac{1}{2}\beta \colon \begin{array}{l}
\displaystyle 2x_1+3x_2\geq 1-s_1, 2x_1+x_2\geq 1-s_2,\\
\displaystyle x_1+2x_2\geq 1-s_3,x_1+x_2\leq t, s_i\geq \beta,\forall i\in[3] \end{array}\biggr\}
\biggr\}.
\end{align*}
When $t=1/2$, one of its optimal solution is $x_1^*=1/4, x_2^*=1/4, s_1^*=0, s_2^* = 1/4, s_3^*=1/4, \beta^*=0$, which is infeasible to the DRCCP. Thus, $\alsoxs$ may not be able to find an optimal solution of the DRCCP. That is, we can have $v^{\as}> v^{A}$ when the uniqueness assumption is violated. \QEDB
\end{example}

\subsection{$\alsoxs$ is Better Than $\alsoxus$}
In the lower-level $\alsoxs$ \eqref{eq_also_cvar_q_wass_b}, we impose the constraint $\beta\leq 0$ based on the $\CVaR$ approximation \eqref{drccp_cvar_formulation}. By relaxing this constraint, we obtain a weaker $\alsoxus$ \eqref{eq_also_cvar_q_wass_weak}. We show that when $\alsoxus$ provides a feasible solution to DRCCP \eqref{eq_drccp_infty}, the lower-level $\alsoxs$ is equivalent to that of $\alsoxus$. Particularly, under type $\infty-$Wasserstein ambiguity set and for a given objective upper bound $t$, when there exists an optimal solution from the lower-level $\alsoxus$ \eqref{eq_also_cvar_q_wass_weak_b} that is feasible to DRCCP \eqref{eq_drccp_infty}, any optimal solution of the lower-level $\alsoxs$ is also feasible to DRCCP \eqref{eq_drccp_infty}.

\begin{theorem}
\label{also_sharp_better_weak_sharp}
Suppose that for any objective upper bound $t$ such that $t\geq \min_{\bm{x}\in \X}\bm c^\top \bm x$, $\alsoxs$ \eqref{eq_also_cvar_q_wass} is better than  $\alsoxus$ \eqref{eq_also_cvar_q_wass_weak}, i.e., $v^{\as}\leq v^{\aus}$.
\end{theorem}
\begin{proof}
It is sufficient to show that for a given objective upper bound $t$, if the lower-level $\alsoxus$ \eqref{eq_also_cvar_q_wass_weak_b}  obtains a feasible solution to DRCCP \eqref{eq_drccp_infty}, then the lower-level $\alsoxs$ \eqref{eq_also_cvar_q_wass_b} is also feasible. 

Let $(\hat{\bm x}, \hat{\beta})$ denote an optimal solution from the lower-level  $\alsoxs$ \eqref{eq_also_cvar_q_wass_b} and let $(\bar{\bm x}, \bar{\beta})$ denote an optimal solution from the lower-level $\alsoxus$ \eqref{eq_also_cvar_q_wass_weak_b}. Suppose that $\bar{\bm x}$ is feasible to DRCCP \eqref{eq_drccp_infty}, i.e., $\bar{\bm x}$ satisfies \eqref{alsox_drccp_formualtion}. Now let
\begin{align*}
 \bar{\beta}^*: =\VaR_{1-\varepsilon} \left\{ \theta \left\| \bm a_i(\bar{\bm x}) \right\|_*+ \bm a_i(\bar{\bm x})^\top \trzeta - b_i(\bar{\bm x}), \forall i\in[I] \right\}\leq 0.   
\end{align*}
According to theorem 1 in \cite{rockafellar2000optimization} (see, e.g., equation (7) in \cite{rockafellar2000optimization}), then we have that $(\bar{\bm x}, \bar{\beta}^*)$ is another optimal solution to the lower-level $\alsoxus$ \eqref{eq_also_cvar_q_wass_weak_b}. Since the only difference between  the lower-level $\alsoxs$ \eqref{eq_also_cvar_q_wass_b} and the lower-level $\alsoxus$ \eqref{eq_also_cvar_q_wass_weak_b} is the constraint $\beta\leq 0$, the solution $(\bar{\bm x}, \bar{\beta}^*)$ is also optimal to the lower-level $\alsoxs$ \eqref{eq_also_cvar_q_wass_b}. That is, for a given objective upper bound $t$, both lower-level problems have the same optimal value.

In this case, for any optimal solution from  the lower-level $\alsoxs$ \eqref{eq_also_cvar_q_wass_b}, it should also be optimal to the lower-level $\alsoxus$ \eqref{eq_also_cvar_q_wass_weak_b}. Hence, $(\hat{\bm x}, \hat{\beta})$ is also optimal to the lower-level $\alsoxus$ \eqref{eq_also_cvar_q_wass_weak_b}.
 Based on theorem 1 in \cite{rockafellar2000optimization} (see, e.g., equation (7) in \cite{rockafellar2000optimization}), we have 
\begin{align*}
\hat\beta\geq \VaR_{1-\varepsilon} \left\{ \theta \left\| \bm a_i(\hat{\bm x}) \right\|_*+ \bm a_i(\hat{\bm x})^\top \trzeta -  b_i(\hat{\bm x}), \forall i\in[I] \right\}.
\end{align*}
Combining with the condition that $\hat\beta\leq 0$, we have
\begin{align*}
\VaR_{1-\varepsilon} \left\{ \theta \left\| \bm a_i(\hat{\bm x}) \right\|_*+ \bm a_i(\hat{\bm x})^\top \trzeta -  b_i(\hat{\bm x}), \forall i\in[I] \right\} \leq 0,
\end{align*}
which implies that $\hat{\bm x}$ satisfies \eqref{alsox_drccp_formualtion}, i.e., $\hat{\bm x}$ is also feasible to DRCCP \eqref{eq_drccp_infty}.
This completes the proof.
 \QEDA
\end{proof}
%

We make the following remarks about Theorem~\ref{also_sharp_better_weak_sharp}:
\begin{enumerate}[label=(\roman*)]
\item Note that in Theorem~\ref{also_sharp_better_weak_sharp}, we show that $\alsoxs$ is better than $\alsoxus$;
and
\item One may expect that $\alsox$ and $\alsoxus$ are comparable. In fact, they are not. In Example~\ref{cvar_better_nonconvex}, we can show that  $\alsoxus$ returns a better solution (i.e., $v^{\aus}=0$ and $\alsox$ fails to find any feasible solution), while in Example~\ref{example_worse_convex}, $\alsox$ can have a better solution (i.e., $v^A=1/2 <v^{\aus}=2/3$).
\end{enumerate}

\subsection{$\alsoxus$ is Better Than $\CVaR$ Approximation}
Recall that the differences between $\alsoxus$ and $\CVaR$ approximation lie in the corresponding upper-level and lower-level problems, where the checking condition in the upper-level $\CVaR$ approximation is more restricted than the one in the upper-level $\alsoxus$, and the lower-level $\alsoxus$ is a relaxation of the lower-level $\CVaR$ approximation. As a result, we show that when $\CVaR$ approximation provides a feasible solution to DRCCP \eqref{eq_drccp_infty}, any optimal solution of the lower-level $\alsoxus$ must be feasible to DRCCP \eqref{eq_drccp_infty}.
\begin{theorem}
\label{also_sharp_better_cvar}
$\alsoxus$ \eqref{eq_also_cvar_q_wass_weak} is better than $\CVaR$ approximation \eqref{drccp_cvar_formulation}, i.e., $v^{A\underline\#}\leq v^{\CVaR}$.
\end{theorem}
\begin{proof}
Notice that for a given objective upper bound $t$, the lower-level $\alsoxus$ is a relaxation of the lower-level $\CVaR$ approximation. Thus, if the optimal value of the lower-level $\CVaR$ approximation is non-positive (i.e., constraint \eqref{cvar_drccp_formualtion} holds), then the optimal value of the lower-level $\alsoxus$ must be non-positive, which ensures that any optimal solution is feasible to the DRCCP \eqref{eq_drccp_infty} according to \cite{nemirovski2007convex}.
Therefore, for a given $t$, $\alsoxus$ must find a feasible solution to DRCCP \eqref{eq_drccp_infty} if  $\CVaR$ approximation finds one. This completes the proof.
\QEDA
\end{proof}
We remark that the result in Theorem~\ref{also_sharp_better_cvar}  can be extended to any general ambiguity set, that is, $\alsoxus$ is better than  $\CVaR$ approximation under a general ambiguity set. However, this does not hold for Theorem~\ref{also_cvar_better_alsox}, since the worst-case distributions in the lower-level $\alsox$ \eqref{drccp_alsox} and the upper-level  $\alsox$ \eqref{alsox_drccp_formualtion} may not be the same. Hence, $\alsoxs$ and $\alsox$ are not comparable under the general ambiguity set.
Below is an example to illustrate Theorem~\ref{also_sharp_better_cvar}.
\begin{example}
\label{weak_sharp_better_cvar}
\rm
Consider a single DRCCP under type $\infty-$Wasserstein ambiguity set  with $\theta=1/2$. Assume that the empirical distribution has $3$ equiprobable scenarios (i.e., $N=3$, $\Pr\{\tilde\zeta=\zeta^i\}=1/N$), risk parameter $\varepsilon=1/2$, deterministic set $\X=\Re_+$, function $ a_1(x)^\top \zeta-b_1({x}) = x - {\zeta}$, 
$\zeta^1=5/2$, $\zeta^2=3/2$, and $\zeta^3=1/2$. In this example, DRCCP \eqref{eq_drccp_infty} resorts to
\begin{equation*}
v^*=\min_{x\geq 0}\left\{x\colon  \I\left(x\geq 3\right)+\I\left(x\geq 2\right)+\I\left( x\geq 1\right)\geq 2 \right\},
\end{equation*}
Its $\CVaR$ approximation is
\begin{align*}
v^\CVaR= \min_{x\geq 0,\bm s,\beta\leq 0} \left\{ x\colon x\geq 3-s_1, x\geq 2-s_2, x\geq 1-s_3, \frac{1}{3}\sum_{i\in[3]}s_i-\frac{\beta}{2}\leq 0, s_i\geq \beta, \forall i\in[3]\right\},
\end{align*}
and the weak formulation of $\alsoxs$   \eqref{eq_also_cvar_q_wass_weak} can be written as
\begin{align*}
&v^{\aus} =\min _{ {t}}\, \Biggl\{t\colon  \sum_{i\in[3]}\I({s^*_i}>0) \leq 2,\\
&(x^*,\bm s^*,\beta^*)\in\argmin_{\begin{subarray}{c}
x\geq 0, \bm s, \beta\end{subarray}}\biggl\{ \frac{1}{3} \sum_{i\in[3]}s_i - \frac{1}{2}\beta\colon \begin{array}{l}
\displaystyle x \geq 3-s_1, x\geq 2-s_2,x\geq 1-s_3, \\
\displaystyle x\leq t, s_i\geq \beta,\forall i\in[3]\end{array}\biggr\}
\Biggr\}.
\end{align*}
By the straightforward calculation, we have  $v^* =2$, $v^\CVaR =8/3$, $v^{\aus}=2$. Therefore, in this example, $\alsoxus$ returns the optimal solution, but $\CVaR$ approximation cannot.   
\QEDB
\end{example}

\subsection{Summary of Comparisons}
Finally, we conclude this section by providing theoretical comparisons among the output objective values of  $\alsoxs$, $\alsoxus$, $\alsox$, and $\CVaR$ approximation, which are shown in Figure~\ref{summary_figure}.

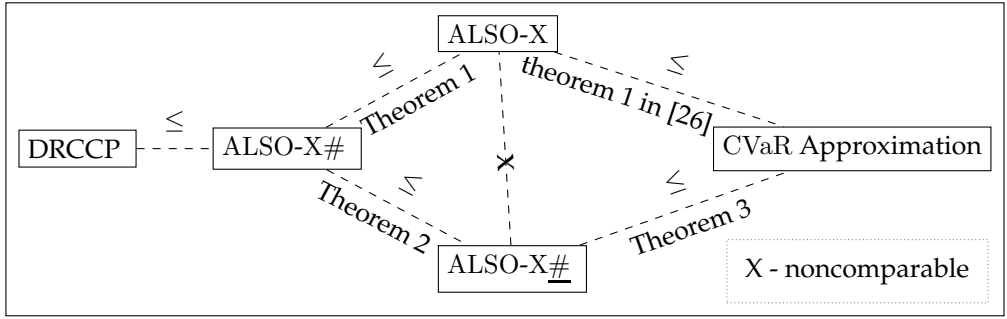
\begin{figure}[htbp]\centering
	\begin{tikzpicture}[framed,
		every path/.style={>=latex},
		every node/.style={draw}
		]
		\node[rectangle, align=left] (sharp) { $\alsoxs$ };
		\node[rectangle, left = of sharp, align=left] (drccp) { DRCCP };
			\draw [dashed,-,circle] (drccp) -- (sharp); 
		\draw [dashed,-,circle,draw opacity=0] (drccp) -- (sharp) node[above,midway,sloped]{$\leq$}; 
		
		\node[rectangle, below right= of sharp, align=left] (wsharp) {$\alsoxus$ };
		\node[rectangle, above right= of sharp, align=left] (alsox) { $\alsox$};
		
		\node[rectangle, above right= of wsharp, align=left] at (4.6,-1.3) (cvar) {$\CVaR$ Approximation};  
		
		\draw [dashed,-,circle] (sharp) -- (wsharp); 
		\draw [dashed,-,circle,draw opacity=0] (sharp) -- (wsharp) node[above,midway,sloped]{$\leq$}; 
		\draw [dashed,-,rectangle,draw opacity=0] (sharp) -- (wsharp) 
		node[below,sloped,pos=0.4]{ Theorem~\ref{also_sharp_better_weak_sharp}};
		
		\draw [dashed,-,circle] (sharp) -- (alsox); 
		\draw [dashed,-,circle,draw opacity=0] (sharp) -- (alsox) node[above,midway,sloped]{$\leq$}; 
		\draw [dashed,-,rectangle,draw opacity=0] (sharp) -- (alsox) 
		node[below,sloped,pos=0.6]{Theorem~\ref{also_cvar_better_alsox}};
		
		\draw [dashed,-,circle] (alsox) -- (cvar); 
		\draw [dashed,-,circle,draw opacity=0] (alsox) -- (cvar) node[above,midway,sloped]{$\leq$}; 
		\draw [dashed,-,rectangle,draw opacity=0] (alsox) -- (cvar)
		node[below,sloped,pos=0.3]{ {theorem 1 in \cite{jiang2022also}}};
		
		\draw [dashed,-,circle] (wsharp) -- (cvar); 
		\draw [dashed,-,circle,draw opacity=0] (wsharp) -- (cvar) node[above,midway,sloped]{$\leq$}; 
		\draw [dashed,-,rectangle,draw opacity=0] (wsharp) -- (cvar)
		node[below,sloped,pos=0.5]{Theorem~\ref{also_sharp_better_cvar} };
		
		\draw [dashed,-,circle] (wsharp) -- (alsox); 
		\draw [dashed,-,circle,draw opacity=0] (wsharp) -- (alsox) 
  node[left,pos=0.3,sloped]{\LARGE x};
  

  	\matrix [densely dotted,draw opacity=0.5,below] at (7.5,-1.2) {
	\node[font=\fontsize{1pt}{3pt},draw opacity=0] { X - noncomparable}; \\
		};

		
	\end{tikzpicture}
	\caption{Summary of Comparisons}
	\label{summary_figure}
\end{figure} 

\section{The Optimal Solution of the Lower-level $\alsox$ is Unique}
\label{sec_unique_alsox}
This section investigates conditions under which the lower-level $\alsox$ can provide a unique optimal solution, a sufficient condition guaranteeing that $\alsoxs$ is better than $\alsox$ according to Theorem \ref{also_cvar_better_alsox}.
Notably, we prove that the lower-level $\alsox$ \eqref{drccp_alsox} admits a unique optimal solution if one of the following conditions hold: (i) empirical data are sampled from continuous nondegenerate distributions and set $\X$ is arbitrary; (ii) a single DRCCP with continuous reference distribution and set $\X$ is convex;  (iii) a joint DRCCP with right-hand uncertainty, a continuous reference distribution, and a convex set $\X$; or (iv) a joint DRCCP with left-hand uncertainty and set $\X$ is convex.

\subsection{Uniqueness: DRCCPs with an i.i.d. Empirical Reference Distribution Sampling from a Continuous Distribution}
In this subsection, we consider the case under which lower-level $\alsox$ \eqref{drccp_alsox} can provide a unique solution when the reference distribution is of finite support and is constructed by  i.i.d. samples from a continuous nondegenerate distribution. 
For the given i.i.d. samples of the random parameters $\trzeta$, we consider the following $\alsox$:
\begin{subequations}
\label{eq_also_x_discrete_formulation}
\begin{align}
v^A =\min _{ {t}}\quad & t,\label{eq_also_x_discrete_a}\\
\textup{s.t.}\quad&\bm x^*\in\argmin_{\bm x}\left\{\frac{1}{N} \sum_{j\in[N]}\max_{i\in [I]} \left[\theta\left\|\bm{a}_i(\bm{x}) \right\|_*+\bm{a}_i(\bm{x})^\top{\rzeta^j}-b_i(\bm{x})\right]_+\colon \bm{c}^\top \bm{x} \leq t,\bm {x}\in \X\right\},\label{eq_also_x_discrete}\\
& \sum_{j\in[N]}\I\left\{\max_{i\in [I]}\left[\theta\left\|\bm{a}_i(\bm{x}^*) \right\|+\bm{a}_i(\bm{x}^*)^\top {\rzeta^j}-b_i(\bm{x}^*)\right]_+=0\right\} \geq N-\floor{N\varepsilon}.\label{eq_also_x_discrete_checking}
\end{align}
\end{subequations}
In fact, we show that with  probability $1$, for any objective upper bound $t$ such that $t\geq \min_{\bm{x}\in \X}\bm c^\top \bm x$ and the optimal value of the lower-level $\alsox$ \eqref{drccp_alsox} is positive (i.e., $v^A(t)>0$), the lower-level $\alsox$ \eqref{drccp_alsox} admits a unique optimal solution. Note that if $v^A(t)=0$, then any optimal solution of the lower-level $\alsox$ \eqref{drccp_alsox} is feasible to DRCCP \eqref{eq_drccp_infty}. Thus, we focus on the non-trivial case when $v^A(t)>0$.

\begin{theorem}
\label{alsox_unique_sampling_linear}
Suppose (i) for any $\bm x_1, \bm x_2\in\X$ with $\bm x_1\not= \bm x_2$ and any pair $(i_1,i_2)\in[I]\times [I]$, $\bm{a}_{i_1}(\bm{x}_1)\not=\bm{a}_{i_2}(\bm{x}_2)$; (ii) the true distribution $\Pr^*$ of the random parameters $\trzeta$ is continuous and nondegenerate; and (iii)  $\rzeta^1,\rzeta^2,\cdots, \rzeta^N$ are i.i.d. samples of the random parameters $\trzeta$.  Then the lower-level $\alsox$ \eqref{drccp_alsox}  admits a unique optimal solution when $t\geq \min_{\bm x\in \X}\bm c^\top \bm x$ and $v^A(t)>0$.
\end{theorem}
\begin{proof}
We first write the lower-level $\alsox$ \eqref{eq_also_x_discrete} as
\begin{align}
F(\bm x,\S_{\bm x})\colon= \frac{1}{N}\sum_{j\in[\S_{\bm x}]} \max_{i\in [I]} \left[\theta\left\|\bm{a}_i(\bm{x}) \right\|_*+\bm{a}_i(\bm{x})^\top {\rzeta^j}-b_i(\bm{x})\right],
\label{alsox_discrete_set_eq}
\end{align}
where set $\S_{\bm x}$ is defined as  $\S_{\bm x}=\{ j\in[N]:\theta\left\|\bm{a}_i(\bm{x}) \right\|_*+ \bm{a}_i(\bm{x})^\top {\rzeta^j}-b_i(\bm{x})\geq 0, \forall i\in[I]  \}$.
Suppose there exist two different solutions $\bm x_1\not=\bm x_2\in\X$ in \eqref{alsox_discrete_set_eq}. In this proof, we  suppress the notations as $F(\bm x_1,\S_1)=F(\bm x_1,\S_{\bm x_1})$ and $F(\bm x_2,\S_2)=F(\bm x_2,\S_{\bm x_2})$. 
We consider $F(\bm x_1,\S_{1})=F(\bm x_2,\S_{2})>0$ since $v^A(t)>0$. We split the remaining proof into three steps based on the deterministic set $\X$. 

\noindent\textbf{Step I.} Suppose that set $\X$ is compact and discrete. Let $i^j_1$ and $i^j_2$ denote the maximum pieces of the sample $j\in[N]$ in the objective function \eqref{alsox_discrete_set_eq} corresponding to $\bm x_1, \bm x_2$, respectively. There are two cases to discuss. 

\textbf{Case 1).} When $\S_1 = \S_2=\Te\subseteq[N]$ and $\Te\not= \emptyset$, by the definition of $F(\bm x_1,\S_{1})=F(\bm x_2,\S_{2})$, we have
\begin{align*}
&\Pr^*\left\{\trzeta\colon F(\bm x_1,\S_1)=F(\bm x_2,\S_2)>0\middle| \Te\not=\emptyset,i^j_1,i^j_2, \forall j\in[N] \right\}\\ 
\leq &\Pr^*\biggl\{\trzeta\colon \sum_{j\in\Te}\left[ \left[\bm{a}_{i^j_1}(\bm{x}_1) - \bm{a}_{i^j_2}(\bm{x}_2)  \right]^\top{\rzeta^j} \right] \\
 &\quad\quad\quad = \sum_{j\in\Te} \left[ b_{i^j_1}(\bm{x}_1)-b_{i^j_2}(\bm{x}_2)+\theta\left[\left\|\bm{a}_{i^j_2}(\bm{x}) \right\|_* - \left\|\bm{a}_{i^j_1}(\bm{x}) \right\|_* \right]\right] \bigg| \Te\not=\emptyset, i^j_1,i^j_2, \forall j\in[N]\biggr\}.
\end{align*}
According to the presumption (i), for any $\bm x_1\not=\bm x_2$ and any pair $(i_1,i_2)\in[I]\times [I]$,  we have $\bm{a}_{i^j_1}(\bm{x}_1)\not=\bm{a}_{i^j_2}(\bm{x}_2)$ for all $j\in[N]$. Since the random parameters $\trzeta$ is continuous and nondegenerate (see, e.g., definition 24.16 in \cite{ken1999levy}), then
\begin{align*}
&\Pr^*\biggl\{\trzeta\colon \sum_{j\in\Te}\left[ \left[\bm{a}_{i^j_1}(\bm{x}_1) - \bm{a}_{i^j_2}(\bm{x}_2)  \right]^\top{\rzeta^j} \right] \\
 &\quad\quad\quad = \sum_{j\in\Te} \left[ b_{i^j_1}(\bm{x}_1)-b_{i^j_2}(\bm{x}_2)+\theta\left[\left\|\bm{a}_{i^j_2}(\bm{x}) \right\|_* - \left\|\bm{a}_{i^j_1}(\bm{x}) \right\|_* \right]\right]\bigg| \Te\not=\emptyset, i^j_1,i^j_2, \forall j\in[N]\biggr\} = 0.
\end{align*}
%

\textbf{Case 2).} When $\S_1 \not= \S_2$, by the definition of $F(\bm x_1,\S_{1})=F(\bm x_2,\S_{2})$,  we have
\begin{align*}
&\Pr^*\left\{\trzeta\colon F(\bm x_1,\S_1)=F(\bm x_2,\S_2)>0 \middle| \S_1 \not= \S_2, i^j_1,i^j_2,\forall j\in[N] \right\}  \\
\leq  &\Pr^*\biggl\{\trzeta\colon \sum_{j\in\S_1\setminus \S_2}\left[  \bm{a}_{i^j_1}(\bm{x}_1)^\top {\rzeta^j}\right] -\sum_{j\in\S_2\setminus \S_1}\left[   \bm{a}_{i^j_2}(\bm{x}_2) ^\top {\rzeta^j} \right] 
+\sum_{j\in\S_1\cap \S_2}\left[  \left[\bm{a}_{i^j_1}(\bm{x}_1)- \bm{a}_{i^j_2}(\bm{x}_2)  \right]^\top {\rzeta^j}  \right] \\
& \quad =\sum_{j\in\Te} \left[ b_{i^j_1}(\bm{x}_1)-b_{i^j_2}(\bm{x}_2)+\theta\left[\left\|\bm{a}_{i^j_2}(\bm{x}) \right\|_* - \left\|\bm{a}_{i^j_1}(\bm{x}) \right\|_* \right]\right] \bigg| \S_1 \not= \S_2, i^j_1,i^j_2, \forall j\in[N]  \biggr\}.
\end{align*}
Since at least one of the sets $ \S_1\setminus \S_2$, $\S_2\setminus \S_1$, $ \S_1\cap \S_2$ is nonempty, together with the fact that the distribution of random parameters $\trzeta$ is continuous and nondegenerate, we have 
\begin{align*}
&\Pr^*\biggl\{\trzeta\colon \sum_{j\in\S_1\setminus \S_2}\left[  \bm{a}_{i^j_1}(\bm{x}_1)^\top {\rzeta^j}\right] -\sum_{j\in\S_2\setminus \S_1}\left[   \bm{a}_{i^j_2}(\bm{x}_2) ^\top {\rzeta^j} \right] 
+\sum_{j\in\S_1\cap \S_2}\left[  \left[\bm{a}_{i^j_1}(\bm{x}_1)- \bm{a}_{i^j_2}(\bm{x}_2)  \right]^\top {\rzeta^j}  \right] \\
& \quad =\sum_{j\in\Te} \left[ b_{i^j_1}(\bm{x}_1)-b_{i^j_2}(\bm{x}_2)+\theta\left[\left\|\bm{a}_{i^j_2}(\bm{x}) \right\|_* - \left\|\bm{a}_{i^j_1}(\bm{x}) \right\|_* \right]\right] \bigg| \S_1 \not= \S_2, i^j_1,i^j_2, \forall j\in[N]  \biggr\} = 0.
\end{align*}
Combining these two cases,  for a given $\bm x_1\not=\bm x_2$,  we have
\begin{align*}
& \Pr^*\{\trzeta\colon F(\bm x_1,\S_1)=F(\bm x_2,\S_2)>0 \} \\
= & \sum_{\Te\subseteq[N],i^j_1\in[I],i^j_2\in[I],\forall j\in \Te} \Pr^*\left\{ \Te\not=\emptyset, i^j_1,i^j_2, \forall j\in \Te \right\}\Pr^*\left\{\trzeta\colon F(\bm x_1,\S_1)=F(\bm x_2,\S_2)>0\middle| \Te\not=\emptyset,i^j_1,i^j_2, \forall j\in[N]  \right\} \\
& + \sum_{\S_1,\S_2\subseteq[N], \S_1\not=\S_2,i^j_1\in ,i^j_2\in\S_2,\forall j\in \S_1\cup\S_2} \Pr^*\left\{ \S_1 \not= \S_2,i^j_1,i^j_2,\forall j\in \S_1\cup\S_2 \right\}\\
&\Pr^*\left\{\trzeta\colon F(\bm x_1,\S_1)=F(\bm x_2,\S_2)>0\middle| \S_1 \not= \S_2,i^j_1,i^j_2, \forall j\in[N]  \right\}  \\
=& 0.
\end{align*}
Hence, for any $\bm x_1\not=\bm x_2\in\X$, we further have
\begin{align*}
\Pr^*\left\{\trzeta\colon \bigcup_{ \begin{subarray}{c}
\bm x_1\in\X, \bm x_2\in\X,\\
\bm x_1\not = \bm x_2
\end{subarray}
}  F(\bm x_1,\S_1)=F(\bm x_2,\S_2)>0 \right\} =0.
\end{align*}
Therefore, we show that when set $\X$ is discrete and compact, there exists a unique optimal solution in the lower-level \eqref{alsox_discrete_set_eq} with probability $1$.


\noindent\textbf{Step II.} Suppose that set $\X$ is compact but may not be discrete.  Suppose set $\X\subseteq [-M,M]^n$. Then for some small $\nu>0$, by discretization, we have for any $\bm x\in \X$, there exists $\bm y\in\X^\nu$, such that $\|\bm x-\bm y\|_\infty\leq \nu$ and $|\X^\nu|\leq |2M/\nu|^n$. Instead of optimizing over $\X$, we consider optimizing over $\X^\nu$ in \eqref{alsox_discrete_set_eq}. Here, we choose $\nu$ as $\nu(\tau)=2M/2^\tau,\tau\in\Ne$. Following the similar procedures in Step I, we assume that there exist two different solutions $\bm x_1(\nu(\tau))\not=\bm x_2(\nu(\tau))$ such that $ \bm x_1(\nu(\tau)), \bm x_2(\nu(\tau))\in\X^\nu(\tau) , F(\bm x_1(\nu(\tau)),\S_1)=F(\bm x_2(\nu(\tau)),\S_2)>0$.  For any $\tau\in\Ne$, we have
\begin{align*}
\Pr^*\left\{\trzeta\colon\bigcup_{ \begin{subarray}{c}
\bm x_1(\nu(\tau))\in\X^\nu(\tau),\bm x_2(\nu(\tau))\in\X^\nu(\tau),\\
\bm x_1(\nu(\tau))\not=\bm x_2(\nu(\tau))
\end{subarray}
} F(\bm x_1(\nu(\tau)),\S_1)=F(\bm x_2(\nu(\tau)),\S_2)>0 \right\}=0.
\end{align*}
When $\tau$ increases, the number of feasible solutions increases. That is, the measurable sequence 
\begin{align*}
\bigcup_{ \begin{subarray}{c}
\bm x_1(\nu(\tau))\in\X^\nu(\tau),\bm x_2(\nu(\tau))\in\X^\nu(\tau),\\
\bm x_1(\nu(\tau))\not=\bm x_2(\nu(\tau))
\end{subarray}
} \left\{ F(\bm x_1(\nu(\tau)),\S_1)=F(\bm x_2(\nu(\tau)),\S_2)>0 \right\}
\end{align*}
is monotone nondecreasing as $\tau$ increases. Following the Monotone Convergence Theorem for sequences of measurable sets (see, e.g., theorem 1.26 in \cite{yeh2006real}),  when $\tau\to \infty$, the limit of this measurable sequence  exists. Thus, we have
\begin{align*}
&\Pr^*\left\{\trzeta\colon \lim_{\tau\to \infty} \bigcup_{ \begin{subarray}{c}
\bm x_1(\nu(\tau))\in\X^\nu(\tau),\bm x_2(\nu(\tau))\in\X^\nu(\tau),\\
\bm x_1(\nu(\tau))\not=\bm x_2(\nu(\tau))
\end{subarray}
} F(\bm x_1(\nu(\tau)),\S_1)=F(\bm x_2(\nu(\tau)),\S_2)>0 \right\}\\
=&\lim_{\tau\to \infty}\Pr^*\left\{\trzeta\colon  \bigcup_{ \begin{subarray}{c}
\bm x_1(\nu(\tau))\in\X^\nu(\tau),\bm x_2(\nu(\tau))\in\X^\nu(\tau),\\
\bm x_1(\nu(\tau))\not=\bm x_2(\nu(\tau))
\end{subarray}
} F(\bm x_1(\nu(\tau)),\S_1)=F(\bm x_2(\nu(\tau)),\S_2)>0 \right\}=0.
\end{align*}
Thus, when set $\X$ is compact but not discrete, there exists a unique solution of \eqref{alsox_discrete_set_eq} with probability $1$.

\noindent\textbf{Step III.} The result holds when set $\X$ is not compact. Let $\hat{\X}_r=\X\cap \Be(0,r)$ with $r>0$. By definition, $\hat{\X}_r$ is compact. Then following the similar procedures in Step I and Step II, there exist two different solutions $\bm x_1\not=\bm x_2$ such that $\bm x_1, \bm x_2\in\hat{\X}_r, F(\bm x_1,\S_1)=F(\bm x_2,\S_2)>0$. For any $r>0$, since set $\hat{\X}_r$ is bounded, according to Step II, we have 
\begin{align*}
&\Pr^*\left\{\trzeta\colon  \bigcup_{ \begin{subarray}{c}
\bm x_1\in\hat{\X}_r,\bm x_2\in\hat{\X}_r,\\
\bm x_1\not=\bm x_2
\end{subarray}
} F(\bm x_1,\S_1)=F(\bm x_2,\S_2)>0 \right\}=0.
\end{align*}
Due to Monotone Convergence Theorem for sequences of measurable sets (see, e.g., theorem 1.26 in \cite{yeh2006real}), we have
\begin{align*}
&\Pr^*\left\{\trzeta\colon  \lim_{r\to \infty}\bigcup_{ \begin{subarray}{c}
\bm x_1\in\hat{\X}_r,\bm x_2\in\hat{\X}_r,\\
\bm x_1\not=\bm x_2
\end{subarray}
} F(\bm x_1,\S_1)=F(\bm x_2,\S_2)>0 \right\}
\\=&
\lim_{r\to \infty}\Pr^*\left\{\trzeta\colon  \bigcup_{ \begin{subarray}{c}
\bm x_1\in\hat{\X}_r,\bm x_2\in\hat{\X}_r,\\
\bm x_1\not=\bm x_2
\end{subarray}
} F(\bm x_1,\S_1)=F(\bm x_2,\S_2)>0 \right\}=0.
\end{align*}
Therefore, when set $\X$ is not compact, there exists a unique solution of \eqref{alsox_discrete_set_eq} with probability $1$. This completes the proof.
\QEDA
\end{proof}
We make the following remarks on Theorem~\ref{alsox_unique_sampling_linear}:
\begin{enumerate}[label=(\roman*)]
\item
The proof shows that any objective value of the lower-level $\alsox$ is unique; and
\item  The key of the proof is to exploit the properties of the linear uncertain constraints and continuous nondegenerate distribution. Relaxing any of them, the result in Theorem~\ref{alsox_unique_sampling_linear} may not hold. 
\end{enumerate}

Sampling from a continuous distribution helps us find a unique solution of the lower-level $\alsox$ \eqref{drccp_alsox}. However, 
if the conditions in Theorem~\ref{alsox_unique_sampling_linear} were not met, there might not be a unique optimal solution.
For example,  when the true distribution of random parameters $\trxi$ is of discrete support, the lower-level $\alsox$ may not have a unique solution, i.e., in Example~\ref{example_worse_convex}, when $t=1/2$, the lower-level $\alsox$ provides two optimal solutions, of which one is feasible to DRCCP \eqref{eq_drccp_infty} and another one is not. 

We use the following example to show that the lower-level $\alsox$ \eqref{drccp_alsox} does not admit a unique solution either when the reference distribution is continuous.
\begin{example}\rm
 Consider a single DRCCP under type $\infty-$Wasserstein ambiguity set with a Gaussian distribution $\trzeta\thicksim\mathcal{N}(\bm{\Bar{\mu}},\bar{\rsigma})$, and $\theta=1/2$) with $n=3,\bm{\Bar{\mu}}= [1,1,1]^\top$, $ \bm{\Bar{\mathrm{\Sigma}}}=\left[ \begin{smallmatrix} 1 & 0 & 0\\0 & 1 & 0 \\ 0 & 0 & 1 \end{smallmatrix} \right]$, risk parameter $\varepsilon=0.40$, set $\X=\{0,1\}^3$ and function $\bm a_1(\bm x)^\top \rzeta-b_1({\bm x})=-2+\bm x^\top \rzeta$.  In this example, 
the lower-level $\alsox$ \eqref{drccp_alsox} can be written as 
 \begin{align*}
\bm x^*\in\argmin_{\bm x}\left\{ \E_{\Pr_{\trzeta}}\left[\theta\left\| \bm x \right\|_*+\tilde{\zeta}_1 x_1 +\tilde{\zeta}_2 x_2+\tilde{\zeta}_3 x_3 -2\right]_+\colon -x_1-x_2-x_3 \leq t,\bm {x}\in \{0,1\}^3\right\},
\end{align*}
where the dual norm is $\|\cdot\|_2$.
Let $t=-2$. A simple calculation shows that there are three optimal solutions from the lower-level $\alsox$ \eqref{drccp_alsox} with the same positive objective value, i.e., $(\bm x^1)^*=(1,1,0)^\top$, $(\bm x^2)^*=(1,0,1)^\top$, $(\bm x^3)^*=(0,1,1)^\top$. Therefore, in this case, there is no unique solution from the lower-level $\alsox$.
\QEDB
\end{example}
This motivates us to restrict set $\X$ to be convex when the reference distribution is continuous in the next subsections.

\subsection{Uniqueness: Single DRCCPs with Continuous Reference Distributions}
\label{sec_unique_single_drccp_continous}
We consider the case when the reference distribution is continuous.  For a single DRCCP \eqref{eq_drccp_infty}, i.e., $I=1$, when the affine mappings are $\bm{a}_1(\bm{x})=\bm x$, $b_1(\bm x)=b_1$, the random parameters $\trzeta$ is continuous and the deterministic set $\X$ is convex, the lower-level $\alsox$ \eqref{drccp_alsox} is equivalent to
\begin{align*}
\bm x^* \in \argmin_{\bm x}\left\{\E_{\Pr_{\trzeta}}\left[\theta\left\|\bm x \right\|_*+\bm x^\top{\trzeta}-b_1\right]_+\colon \bm{c}^\top \bm{x} \leq t,\bm {x}\in \X\right\}.
\end{align*}
Recall that $\hat {F}(\bm x)$ denotes the objective function in the lower-level $\alsox$ \eqref{drccp_alsox}. In this case, we have
\begin{align*}
	\hat {F}(\bm x)=\E_{\Pr_{\trzeta}}\left[\theta\left\|\bm x \right\|_*+\bm x^\top{\trzeta}-b_1\right]_+.
\end{align*}

\begin{restatable}{theorem}{theoremdrccpuniquecontinuous}\label{theorem_drccp_unique_continuous} 
Suppose that in a single DRCCP \eqref{eq_drccp_infty}, the deterministic set $\X$ is convex, the reference distribution ${\Pr_{\trzeta}}$ is continuous and nondegenerate with support $\Re^n$, $\|\cdot\|_*=\|\cdot\|_p$ with $p\in (1,\infty)$ and affine mappings $\bm{a}_1(\bm{x})=\bm x$ and $b_1(\bm x)=b_1$. Then the lower-level $\alsox$ \eqref{drccp_alsox} admits a unique optimal solution when $t\geq \min_{\bm x\in \X}\bm c^\top \bm x$ and $v^A(t)>0$.
\end{restatable}

\proof
	See Appendix~\ref{proof_theorem_drccp_unique_continuous}.
\QEDA
 

Note that our analysis in Theorem~\ref{theorem_drccp_unique_continuous} shows that the uniqueness of the lower-level $\alsox$ applies to any general continuous nondegenerate distribution. Then the following corollary shows that the lower-level $\alsox$ \eqref{drccp_alsox} returns a unique optimal solution when we know the upper and lower bounds of the support $\Xi$.

\begin{corollary}
\label{corollary_drccp_unique_continuous_support}
Suppose that in a single DRCCP \eqref{eq_drccp_infty}, the deterministic set $\X$ is convex, $\|\cdot\|_*=\|\cdot\|_p$ with $p\in (1,\infty)$ and affine mappings $\bm{a}_1(\bm{x})=\bm x$ and $b_1(\bm x)=b_1$, and the reference distribution ${\Pr_{\trzeta}}$ is continuous and nondegenerate with a closed convex support $\Xi$ such that
\begin{align*}
    & \exists \rzeta\in \Xi \colon \min_{\bm x\in \X, \bm c^\top \bm x\leq t} \left\{\theta\|\bm x\|_*+\bm x^\top \rzeta \right\}>b_1, \textup{ and } \exists \rzeta\in \Xi \colon \max_{\bm x\in \X, \bm c^\top \bm x\leq t} \left\{\theta\|\bm x\|_*+\bm x^\top \rzeta \right\}<b_1.
\end{align*}
 Then the lower-level $\alsox$ \eqref{drccp_alsox} admits a unique optimal solution when $t\geq \min_{\bm x\in \X}\bm c^\top \bm x$ and $v^A(t)>0$.
\end{corollary}
\begin{proof}
Using the fact that in Part (i) of the proof of Theorem \ref{theorem_drccp_unique_continuous}, we also have $\int_{\partial \mu(\bm x^*)}  ( \bm y^\top{\partial (\theta\|\bm x\|_*)}/{\partial \bm x}+\bm y^\top\rzeta)^2\Pr(d\rzeta)>0$ according to the presumption, the proof is almost identical to that of Theorem \ref{theorem_drccp_unique_continuous}  and is thus  omitted.
    \QEDA
\end{proof}

We remark that the result in Theorem~\ref{theorem_drccp_unique_continuous} can also be generalized to  $\|\cdot\|_*=\|\cdot\|_p$ with $p\in \{1,\infty\}$. Due to the page limit, we refer interested readers to Appendix~\ref{proof_corollary_drccp_unique_continuous} for the proof.

 \subsection{Uniqueness: Joint DRCCPs with a Continuous Reference Distribution}
\label{sec_unique_joint_drccp_independent}
In this subsection, we consider a joint DRCCP with right-hand uncertainty and a continuous reference distribution. In particular, we assume that $I=n$, the uncertainty constraint is $\bm a_i(\bm x)^\top \rzeta -  b_i(\bm x) = {\zeta}_i - x_i$, and the random parameter $\tilde{\zeta}_i$  is continuous for each $i\in[n]$. That is, we consider the following DRCCP:
\begin{align*}
v^* = \min_{\bm x\in \X}\left\{\bm c^\top \bm x\colon	\inf_{\Pr\in\P_\infty}\Pr\left\{ \trxi\colon\tilde{\xi}_i\leq x_i, \forall i\in[n] \right\}\geq 1-\varepsilon\right\}, 
\end{align*}
that is,
\begin{align}
v^* = \min_{\bm x\in \X}\left\{\bm c^\top \bm x\colon	\Pr\left\{\trzeta\colon \tilde{\zeta}_i+\theta\leq x_i  , \forall i\in[n] \right\}\geq 1-\varepsilon\right\}. \label{joint_drccp_eq_multi}
\end{align}
In this case, the lower-level $\alsox$ \eqref{drccp_alsox} is 
\begin{align*}
v^A(t)=	\min_{\bm x\in\X}\left\{\E_{\Pr_{\trzeta}} \left[ \max_{i\in[n]} \left\{  \tilde{\zeta}_i + \theta - x_i\right\}_+ \right]\colon \bm c^\top \bm x\leq t \right\}.
\end{align*}
Our proof idea is to show the positive definiteness of the Hessian of the objective function in the lower-level $\alsox$ \eqref{drccp_alsox} to prove the uniqueness. 

\begin{theorem}
\label{theorem_drccp_unique_joint_continuous}
Suppose that in a joint DRCCP \eqref{joint_drccp_eq_multi}, the deterministic set $\X$ is convex, and the reference distribution ${\Pr_{\trzeta}}$ is continuous with support $\Re^n$. Then the lower-level $\alsox$ \eqref{drccp_alsox} admits a unique optimal solution when $t\geq \min_{\bm x\in \X}\bm c^\top \bm x$ 
and $v^A(t)>0$.
\end{theorem}
\begin{proof}
Suppose that there exists an optimal solution $\bm{x}^*$ to the lower-level $\alsox$ \eqref{drccp_alsox}. We split the proof into two steps to show that $\bm{x}^*$ is the unique solution.
 
\noindent{\textbf{Step I.}}	We first provide the closed-form expression of the lower-level $\alsox$ \eqref{drccp_alsox} and its first-order and second-order derivatives, where the objective function in the lower-level $\alsox$ \eqref{drccp_alsox} is $\hat F(\bm x^*) = \max_{i\in[n]} \left\{  \zeta_i +\theta- x_i^*\right\}_+$. According to the continuity of function $f(\tau) = \max\{\tau, 0\}$ and theorem 1 in \cite{rockafellar1982interchange}, we can interchange the
subdifferential operator and expectation, the first-order derivative for each $i\in [n]$ of $\hat F(\bm x^*)$ is
	 \begin{align*}
	 	\frac{\partial{\hat {F}(\bm x^*)}}{\partial x_i}= -\int_{\mu_i(\bm x^*)}  \Pr(d\rzeta),
	 \end{align*}
where $\mu_i(\bm x)= \{\rzeta: \zeta_i +\theta- x_i \geq 0, \zeta_i +\theta- x_i\geq \zeta_j +\theta- x_j, \forall j\in [n]\setminus \{i\} \}$. Let us take its second derivative, i.e.,
for each $i\in [n]$, the diagonal entry of the Hessian matrix $H_{\hat F}(\bm x^*)$ is
	 \begin{align*}
	 	\frac{\partial^2{\hat {F}(\bm x^*)}}{\partial x_i^2}= \int_{\partial\mu_{i0}(\bm x^*)}  \Pr(d\rzeta)+\sum_{\tau \in [n]\setminus \{i\}} \int_{\partial\mu_{i\tau}(\bm x^*)} \Pr(d\rzeta),
	 \end{align*}
	where 
 \begin{align*}
  \partial\mu_{i0}(\bm x)= \left\{\rzeta: \zeta_i +\theta- x_i = 0\geq \zeta_j +\theta- x_j, \forall j\in [n]\setminus \{i\} \right\},   
 \end{align*}
 and 
 \begin{align*}
\partial \mu_{i\tau}(\bm x)= \left\{\rzeta: \zeta_i +\theta- x_i \geq 0, \zeta_i +\theta- x_i=\zeta_\tau +\theta- x_\tau\geq  \zeta_j +\theta- x_j, \forall j\in [n]\setminus \{i\} \right\}.
 \end{align*}
For each $\tau\in [n]\setminus \{i\}$, the off-diagonal entry of the Hessian matrix $H_{\hat F}(\bm x^*)$ is
	 \begin{align*}
	 	\frac{\partial^2{\hat {F}(\bm x^*)}}{\partial x_i\partial x_\tau}= - \int_{\partial\mu_{i\tau}(\bm x^*)} \Pr(d\rzeta).
	 \end{align*}

\noindent{\textbf{Step II.}}	Next, we show that the Hessian matrix $H_{\hat F}(\bm x^*)$ is strictly diagonally dominant. We split the following proof into two parts: (i) $\int_{\partial\mu_{i0}(\bm x^*)}  \Pr(d\rzeta)>0$ for each $i\in[n]$; and (ii) $\int_{\partial\mu_{i\tau}(\bm x^*)} \Pr(d\rzeta)>0$ for each $i\in[n]$ and $\tau\in [n]\setminus \{i\}$. 

\noindent{\textit{Part (i).}} The fact that $\int_{\partial\mu_{i0}(\bm x^*)}  \Pr(d\rzeta)>0$ for each $i\in[n]$ is because set $\partial\mu_{i0}(\bm x^*)$ has a dimension of $n-1$ and set $\partial\mu_{i0}(\bm{x}^*)$ has a nonempty relative interior. 

\noindent{\textit{Part (ii).}} The fact that $\int_{\partial\mu_{i\tau}(\bm x^*)} \Pr(d\rzeta)>0$ for each $i\in[n]$ and $\tau\in [n]\setminus \{i\}$ is because set $\partial\mu_{i\tau}(\bm x^*)$ has a dimension of $n-1$ and set $\partial\mu_{i\tau}(\bm x^*)$ has a nonempty relative interior. 

Thus, the Hessian matrix $H_{\hat F}(\bm x^*)$ is strictly diagonally dominant, i.e., the following two conditions satisfied: (i) for each $i\in[n]$ and $\tau\in [n]\setminus \{i\}$, $\partial^2{\hat {F}(\bm x^*)}/\partial x_i\partial x_\tau<0$; and (ii) for each $i\in[n]$, $\partial^2{\hat {F}(\bm x^*)}/\partial x_i^2 +\sum_{\tau\in[n]\setminus\{j\}} \partial^2{\hat {F}(\bm x^*)}/\partial x_i\partial x_\tau>0$. According to Gershgorin circle theorem (see, e.g., theorem 6.1.10 in \cite{horn2012matrix}), the Hessian matrix $H_{\hat F}(\bm x^*)$ is positive definite. Therefore, the optimal solution $\bm x^*$ is unique.    \QEDA

\end{proof}

The following corollary shows that the lower-level $\alsox$ \eqref{drccp_alsox} returns a unique optimal solution when  $\Xi$ is closed and convex with mild conditions.
\begin{corollary}
\label{corollary_drccp_unique_joint_continuous}
Suppose support $\Xi$ of $\trzeta$ is closed and convex. When $t\geq \min_{\bm x\in \X}\bm c^\top \bm x$
and $v^A(t)>0$, for any $\bm x\in \X\cap\{\bm c^\top \bm x\leq t\}$ with a positive lower-level $\alsox$ objective value, the set
\begin{align*}
  \partial\mu_{i0}(\bm x)= \left\{\rzeta \in \Xi: \zeta_i +\theta- x_i = 0\geq \zeta_j +\theta- x_j, \forall j\in [n]\setminus \{i\} \right\}  
 \end{align*}
 has a dimension of $n-1$. In a joint DRCCP \eqref{joint_drccp_eq_multi}, when the deterministic set $\X$ is convex and the reference distribution ${\Pr_{\trzeta}}$ is continuous with support $\Xi$. Then the lower-level $\alsox$ \eqref{drccp_alsox} admits a unique optimal solution when $t\geq \min_{\bm x\in \X}\bm c^\top \bm x$ and $v^A(t)>0$.
\end{corollary}
\begin{proof}
Using the fact that similar to part (ii) in the proof of Theorem \ref{theorem_drccp_unique_joint_continuous}, we also have $\int_{\partial\mu_{i0}(\bm x^*)}  \Pr(d\rzeta)>0$ for each $i\in[n]$ according to the presumption, the proof is almost identical to that of Theorem \ref{theorem_drccp_unique_joint_continuous}  and is thus omitted.
    \QEDA
\end{proof}



 \subsection{Uniqueness: Joint DRCCPs with Left-hand Side Uncertainty}
\label{sec_unique_joint_drccp_knapsack}
We consider the case when the reference distribution is continuous. For a joint DRCCP \eqref{eq_drccp_infty} with left-hand side uncertainty and knapsack constraints, we assume that the affine mappings are $\bm{a}_i(\bm{x})=\bm x$, $b_i(\bm x)=b_i$ for each $i\in[I]$, the random parameters $\trzeta$ is continuous with $\rzeta:=[\rzeta_1,\cdots,\rzeta_I]$ such that $\rzeta_i$ and $\rzeta_j$ do not overlap for each $i\neq j$,
and the deterministic set $\X$ is convex. That is, we consider the following DRCCP:
\begin{align*}
v^* = \min_{\bm x\in \X}\left\{\bm c^\top \bm x\colon	\inf_{\Pr\in\P_\infty}\Pr\left\{ \trxi\colon \bm x^\top \trxi_i\leq b_i, \forall i\in[I] \right\}\geq 1-\varepsilon\right\}, 
\end{align*}
that is,
\begin{align}
v^* = \min_{\bm x\in \X}\left\{\bm c^\top \bm x\colon	\Pr\left\{\trzeta\colon \theta\left\|\bm x\right\|_*+\bm x^\top \trzeta_i\leq b_i  , \forall i\in[I] \right\}\geq 1-\varepsilon\right\}. \label{joint_drccp_eq}
\end{align}
In this case, the lower-level $\alsox$ \eqref{drccp_alsox} is 
\begin{align*}
v^A(t)=	\min_{\bm x\in\X}\left\{\E_{\Pr_{\trzeta}} \left[ \max_{i\in[I]} \left\{ \theta\left\|\bm x\right\|_*+ \bm x^\top \trzeta_i- b_i\right\}_+ \right]\colon \bm c^\top \bm x\leq t \right\}.
\end{align*}
Recall that $\hat {F}(\bm x)$ denotes the objective function in the lower-level $\alsox$ \eqref{drccp_alsox}. Under this circumstance, we have
\begin{align*}
	\hat {F}(\bm x)=\E_{\Pr_{\trzeta}} \left[ \max_{i\in[I]} \left\{ \theta\left\|\bm x\right\|_*+ \bm x^\top \trzeta_i- b_i\right\}_+ \right].
\end{align*}

\begin{theorem}
\label{theorem_drccp_unique_continuous_joint}
Suppose that in a joint DRCCP \eqref{joint_drccp_eq}, the deterministic set $\X$ is convex, the reference distribution ${\Pr_{\trzeta}}$ is continuous and nondegenerate with support $\Re^n$, $\|\cdot\|_*=\|\cdot\|_p$ with $p\in (1,\infty)$, affine mappings $\bm{a}_i(\bm{x})=\bm x$ and $b_i(\bm x)=b_i$ for each $i\in[I]$, and the random parameters $\trzeta$ is continuous with $\rzeta:=[\rzeta_1,\cdots,\rzeta_I]$ such that $\rzeta_i$ and $\rzeta_j$ do not overlap for each $i\neq j$. Then the lower-level $\alsox$ \eqref{drccp_alsox} admits a unique optimal solution when $t\geq \min_{\bm x\in \X}\bm c^\top \bm x$ and $v^A(t)>0$.
\end{theorem}
\begin{proof}
The first-order derivative of $\hat {F}(\bm x)$ is
\begin{align*}
\hat {F}'(\bm x)=\frac{\partial{\hat {F}(\bm x)}}{\partial \bm x}=\sum_{i\in[I]} \int_{ \mu_i(\bm x)} \frac{\partial}{\partial \bm x} \left[ \theta\left\|\bm x \right\|_*+\bm x^\top{\rzeta}_i \right] \Pr(d\rzeta),
\end{align*}
where for each $i\in[I]$, $\mu_i(\bm x)= \{\rzeta: \theta\|\bm x \|_*+\bm x^\top{\rzeta}_i\geq b_i, \bm x^\top{\rzeta}_i- b_i \geq \bm x^\top{\rzeta}_j- b_j, \forall j\in[I]\setminus \{i\}\}$; and the Hessian of $\hat {F}(\bm x)$ is
\begin{align*}
H_{\hat F}(\bm x)
=& \sum_{i\in[I]}\biggl[    \frac{1}{\left\|\bm x\right\|_2}\int_{ \mu_i(\bm x)\cap \{\theta\left\|\bm x \right\|_*+\bm x^\top{\rzeta}_i= b_i \}} 
 \left( \frac{\partial \theta\left\|\bm x\right\|_*}{\partial \bm x}+\rzeta_i\right)\left( \frac{\partial \theta\left\|\bm x\right\|_*}{\partial \bm x}+\rzeta_i\right)^\top \Pr(d\rzeta)\\
 & \quad + \frac{1}{\left\|\bm x\right\|_2}\sum_{\tau\in[I]\setminus\{i\}}\int_{ \mu_i(\bm x)\cap \{\bm x^\top{\rzeta}_i- b_i =\bm x^\top{\rzeta}_\tau- b_\tau \}} 
 \left( \rzeta_i-\rzeta_\tau\right)\left( \rzeta_i-\rzeta_\tau\right)^\top \Pr(d\rzeta) 
 + 
  \theta\int_{ \mu_i(\bm x)}\frac{\partial^2 \|\bm x\|_*}{\partial \bm x^2}\Pr(d\rzeta) \biggr].
\end{align*}
The remaining proof is similar to that of Theorem~\ref{theorem_drccp_unique_continuous} and is thus omitted for brevity.
\QEDA
\end{proof}

Similarly, when the upper and lower bounds of the support $\Xi$ are accessible, the following corollary shows that the lower-level $\alsox$ \eqref{drccp_alsox} returns a unique optimal solution.
\begin{corollary}
    Suppose support $\Xi$ of $\trzeta$ is closed and convex. When $t\geq \min_{\bm x\in \X}\bm c^\top \bm x$
and $v^A(t)>0$, for any $\bm x\in \X\cap\{\bm c^\top \bm x\leq t\}$ with a positive lower-level $\alsox$ objective value, the set
\begin{align*}
  \partial\mu_{i}(\bm x)= \left\{\rzeta \in \Xi: \theta\left\|\bm x \right\|_*+\bm x^\top{\rzeta}_i-b_i = 0\geq \theta\left\|\bm x \right\|_*+\bm x^\top{\rzeta}_j-b_j, \forall j\in [n]\setminus \{i\} \right\}  
 \end{align*}
 has a dimension of $n-1$. In a joint DRCCP \eqref{joint_drccp_eq}, when the deterministic set $\X$ is convex and the reference distribution ${\Pr_{\trzeta}}$ is continuous support $\Xi$. Then the lower-level $\alsox$ \eqref{drccp_alsox} admits a unique optimal solution when $t\geq \min_{\bm x\in \X}\bm c^\top \bm x$ and $v^A(t)>0$.
\end{corollary}

\section{Exactness: $\alsoxs$ Provides an Optimal Solution to a DRCCP }
\label{sec_exactness}

In this section, we provide sufficient conditions under which $\alsoxs$ provides an optimal solution to a DRCCP under type $\infty-$Wasserstein ambiguity set. 
According to Theorem~\ref{also_sharp_better_weak_sharp}, $v^{\as}\leq v^{\aus}$. Thus, for ease of analysis, we focus on $\alsoxus$ \eqref{eq_also_cvar_q_wass_weak}. 
To begin with, we recast DRCCP \eqref{eq_drccp_infty} and $\alsoxus$ \eqref{eq_also_cvar_q_wass_weak} in the following forms:
\begin{align}
v^* = \min_{\bm x\in\X} \left\{ \bm c^\top \bm x \colon  \hat{G}_{\theta}\left(\bm x^\top \bm h\right)\leq 0\right\}\label{drccp_exact_eq};
\end{align}
and
\begin{subequations}
    \label{alsoxus_exact_eq}
\begin{align}
	v^{\aus} =  \min_{t} \quad & t,     \label{alsoxus_exact_eq_a}\\
	\textup{s.t.} \quad & \bm x^*\in \argmin_{\bm x\in\X} \left\{\overline{F}_{\theta}\left(\bm x^\top \bm h\right) \colon  \bm c^\top \bm x\leq t \right\},\label{alsoxus_exact_eq_b}\\
	&  \hat{G}_{\theta}\left((\bm x^*)^\top\bm h\right)\leq 0.\label{alsoxus_exact_eq_c}
\end{align}
\end{subequations}
We see that if both functions $\hat{G}_{\theta}(\cdot)$ and $\overline F_{\theta}(\cdot)$ are monotone nondecreasing, then $\alsoxus$ can find an optimal solution to DRCCP. So is $\alsoxs$ according to Theorem \ref{also_sharp_better_weak_sharp}.
\begin{theorem}
\label{exactness_alsoxus_monotone}
 Suppose that in DRCCP \eqref{drccp_exact_eq} and $\alsoxus$ \eqref{alsoxus_exact_eq}, both functions $\hat{G}_{\theta}(\cdot)$ and $\overline F_{\theta}(\cdot)$ are monotone nondecreasing. Then $\alsoxs$ is exact.
\end{theorem}
\begin{proof}
Let $v_1,v_2$ be the optimal values of DRCCP \eqref{drccp_exact_eq} and $\alsoxus$ \eqref{alsoxus_exact_eq}, respectively. Since $\alsoxus$ \eqref{alsoxus_exact_eq} is a conservative approximation, we must have $v_1\leq v_2$. Then it remains to show that $v_2\leq v_1$. 

Let $\bm x^*$ be an optimal solution of DRCCP \eqref{drccp_exact_eq} and $t^* = \bm c^\top \bm x^*$. Plug $t^*$ into the lower-level $\alsoxus$ and let $\hat{\bm x}$ be its optimal solution, that is,
  \begin{align*}
      \hat{\bm x}\in \argmin_{\bm x\in\X} \left\{\overline{F}_{\theta}\left(\bm x^\top \bm h\right) \colon  \bm c^\top \bm x\leq t^* \right\}.
  \end{align*}
  Since $\bm x^*$ is feasible to the lower-level $\alsoxus$ with $t^* = \bm c^\top \bm x^*$, we must have $\overline{F}_{\theta}(\hat{\bm x}^\top\bm h) \leq \overline{F}_{\theta}((\bm x^*)^\top \bm h)$. According to the monotonicity assumption of the function $\hat{G}_{\theta}(\bm x^\top \bm h)$, we further have 
  $\hat{G}_{\theta}(\hat{\bm x}^\top\bm h  ) \leq \hat{G}_{\theta}((\bm x^*)^\top\bm h)$, which implies that $\hat{G}_{\theta}( \hat{\bm x}^\top\bm h)\leq 0$. Hence, we have $v_2\leq t^*=v_1$. That is, $\alsoxus$ is exact.
  
 According to Theorem \ref{also_sharp_better_weak_sharp}, $\alsoxs$ is less conservative than $\alsoxus$. Thus, $\alsoxs$  is also exact. This completes the proof.
\QEDA
\end{proof}

Next, we identify three special families of DRCCPs satisfying the conditions in Theorem \ref{exactness_alsoxus_monotone}, namely, single DRCCPs with elliptical, multinomial, and finite-support reference distributions, respectively.

\subsection{Special Case I: Single DRCCPs with Elliptical Reference Distributions}
\label{sec_exactness_elliptical}
We consider a single DRCCP when the reference distribution is elliptical. 
Note that an elliptical distribution $\Pr_{\mathrm{E}}(\bm{\mu},\rsigma,\hat g)$ is described by three parameters, a location parameter $\bm{\mu}$, a positive semi-definite matrix $\rsigma$, and a generating function $\hat g$, 
and its probability density function $\hat{f}$ has the following form:
\begin{equation*}
\hat{f}(\bm{x})=\bar{k}\cdot \hat g\left(\frac{1}{2}(\bm{x}-\bm{\mu})^\top\rsigma^{-1}(\bm{x}-\bm{\mu})\right) 
\end{equation*}
with a positive normalization scalar $\bar{k}$. The probability density function of the standard univariate elliptical distribution $\Pr_{\mathrm{E}}(0,1,\hat g)$ is $\varphi(z)=\bar{k}\hat g(z^2/2)$, and the corresponding cumulative distribution function is $\mathrm{\Phi}(\tau)=\int_{-\infty}^\tau\bar{k}\hat g(z^2/2)dz$. For the single DRCCP \eqref{eq_drccp_infty}, i.e., $I=1$, when the affine mappings are $\bm{a}_1(\bm{x})=\bm x$, $b_1(\bm x)=b_1$, the random parameters $\trzeta$ follow a joint elliptical distribution with $\trzeta\thicksim\Pr_{\mathrm{E}}(\bm{\mu},\rsigma,\hat g)$, and the norm defining the Wasserstein distance is the generalized Mahalanobis norm associated with the matrix $\rsigma$, i.e., $\|\bm{y}\|=\sqrt{\bm{y}^\top\rsigma^{\dagger}\bm{y}}$, for some $\bm y\in \Re^n$, where $\rsigma^{\dagger}$ is the pseudo-inverse. According to the reformulations in proposition 10 of \cite{jiang2022also}, DRCCP \eqref{eq_drccp_infty} resorts to 
\begin{align*}
v^*=\min_{\bm x\in\X}\left\{ \bm c^\top \bm  x\colon \bm \mu^\top \bm x+ \left( \mathrm{\Phi}^{-1}(1-\varepsilon)+\theta\right) \sqrt{\bm{x}^\top\rsigma\bm{x}} -b_1\leq 0\right\},
\end{align*}
the lower-level $\alsoxus$ \eqref{eq_also_cvar_q_wass_weak_b} is equivalent to
\begin{align*}
 \bm x^*\in\argmin_{\bm{x}\in \X,\bm{c}^\top \bm{x} \leq t} \left\{ \bm \mu^\top \bm x +\left[\overline{G}\left(\left(\mathrm{\Phi}^{-1}(1-\varepsilon)\right)^2/2\right)/\varepsilon+\theta\right]\sqrt{\bm x^\top\rsigma\bm x} -b_1 \right\},
\end{align*}
where $\overline{G}(\tau) = G(\infty)-G(\tau)$ and $ G(\tau) = \bar{k}\int_{0}^{\tau} \hat g(z)dz$.

Then we study the exactness of $\alsoxus$ for the following two conditions.

\noindent\textbf{Condition I.} For a single DRCCP under an elliptical reference distribution, suppose that $\rsigma=\bm \mu\bm \mu^\top$ and $\bm \mu^\top \bm x\geq 0$ for any $\bm{x}\in \X$. In this case, we can simplify DRCCP \eqref{eq_drccp_infty} and the lower-level $\alsoxus$ \eqref{eq_also_cvar_q_wass_weak_b} as
\begin{subequations}
	\begin{align}
		&v^*=\min_{\bm x\in \X} \left\{ \bm c^\top \bm x\colon \hat{G}_{\theta}(\bm \mu ^\top \bm x)=\left(1+\mathrm{\Phi}^{-1}(1-\varepsilon) +\theta\right)\bm \mu ^\top \bm x   - b_1 \leq 0\right\},\label{drccp_exact_eq_c1}\\
  &\bm x^*\in\argmin_{\bm{x}\in \X} \left\{   \overline{F}_{\theta}(\bm \mu ^\top \bm x)= \left(1+\overline{G}\left(\left(\mathrm{\Phi}^{-1}(1-\varepsilon)\right)^2/2\right)/\varepsilon+\theta\right)\bm \mu ^\top \bm x -b_1\colon  \bm{c}^\top \bm{x} \leq t\right\},\label{alsoxus_exact_eq_b_c1}
	\end{align}
 \end{subequations}
	respectively. The exactness result readily follows from Theorem~\ref{exactness_alsoxus_monotone}, which is summarized below.
\begin{corollary}
\label{cor_drccp_exact_eq_c1}
	Suppose that in a single DRCCP \eqref{eq_drccp_infty},  the reference distribution ${\Pr_{\trzeta}}$ is elliptical with affine mappings $\bm{a}_1(\bm{x})=\bm x$, $b_1(\bm x)=b_1$, $\rsigma=\bm \mu\bm \mu^\top$, $\bm \mu^\top \bm x\geq 0$ for any $\bm{x}\in \X$, and $1+\mathrm{\Phi}^{-1}(1-\varepsilon)+\theta\geq 0$. Then $\alsoxs$ is exact.
\end{corollary}
\begin{proof} 
According to the reformulations \eqref{drccp_exact_eq_c1} and \eqref{alsoxus_exact_eq_b_c1} and the assumptions that  $\bm \mu^\top \bm x\geq 0$ for any $\bm{x}\in \X$ and $1+\mathrm{\Phi}^{-1}(1-\varepsilon)+\theta\geq 0$, 
both functions $\hat{G}_{\theta}(\cdot)$ and $\overline F_{\theta}(\cdot)$ are monotone nondecreasing.
Hence, conditions in Theorem~\ref{exactness_alsoxus_monotone} are satisfied, and we have that $\alsoxs$ is exact.
    \QEDA
\end{proof}

\noindent\textbf{Condition II.} 
For a single DRCCP under an elliptical reference distribution, suppose that $\X\subseteq\{0,1\}^n$, $\bm{\mu}\geq \bm 0$, and $\rsigma= \Diag (\bm \mu)$. In this case, DRCCP \eqref{eq_drccp_infty} and the lower-level $\alsoxus$ \eqref{eq_also_cvar_q_wass_weak_b} can be simplified as
\begin{subequations}
	\begin{align}
	& 	v^*=\min_{\bm x\in \{0,1\}^n} \left\{ \bm c^\top \bm x\colon \hat{G}_{\theta}(\bm \mu ^\top \bm x)=\bm \mu ^\top \bm x+\left(\mathrm{\Phi}^{-1}(1-\varepsilon)+\theta \right)\sqrt{\bm \mu^\top \bm x} - b_1 \leq 0 \right\},\label{drccp_exact_eq_c2} \\
 & \bm x^*\in\argmin_{\bm{x}\in \{0,1\}^n} \left\{   \overline{F}_{\theta}(\bm \mu ^\top \bm x)=  \bm \mu ^\top \bm x +
\left[\overline{G}\left(\left(\mathrm{\Phi}^{-1}(1-\varepsilon)\right)^2/2\right)/\varepsilon+\theta\right]\sqrt{\bm \mu^\top \bm x}  -b_1 \colon \bm{c}^\top \bm{x} \leq t \right\},\label{alsoxus_exact_eq_b_c2}
	\end{align}
 \end{subequations}
respectively. According to Theorem~\ref{exactness_alsoxus_monotone}, we have the following exactness result.
\begin{corollary}
\label{cor_drccp_exact_eq_c2}
	Suppose that in a single DRCCP \eqref{eq_drccp_infty},  the reference distribution ${\Pr_{\trzeta}}$ is elliptical with affine mappings $\bm{a}_1(\bm{x})=\bm x$, $b_1(\bm x)=b_1$, $\X\subseteq\{0,1\}^n$, $\bm{\mu}\geq \bm 0$,  $\rsigma= \Diag (\bm \mu)$, and $\mathrm{\Phi}^{-1}(1-\varepsilon)+\theta\geq 0$. Then $\alsoxs$ is exact.
\end{corollary}
\begin{proof}
According to the reformulations \eqref{drccp_exact_eq_c2} and \eqref{alsoxus_exact_eq_b_c2} and the assumptions that $\bm \mu\geq \bm 0$ and $\mathrm{\Phi}^{-1}(1-\varepsilon)+\theta\geq 0$, both functions $\hat{G}_{\theta}(\cdot)$ and $\overline F_{\theta}(\cdot)$ are monotone nondecreasing. Hence, according to Theorem~\ref{exactness_alsoxus_monotone}, we have that $\alsoxs$ is exact.
    \QEDA
\end{proof}
We remark that the results in Corollary~\ref{cor_drccp_exact_eq_c1} and Corollary~\ref{cor_drccp_exact_eq_c2} hold for general type $q-$Wasserstein ambiguity set, as shown in Section~\ref{extension_exact}.

\subsection{Special Case II: Single DRCCPs with i.i.d. Random Parameters and Binary Decision} 

In this subsection, we study the exactness of a single DRCCP \eqref{eq_drccp_infty} with the binary decision variables, where we can provide the closed-form expression of the lower-level $\alsoxus$. 

Let us first consider a single packing DRCCP \eqref{eq_drccp_infty}, where the deterministic set $\X\subseteq \{0,1\}^n$ is binary, the affine mappings are $\bm{a}_1(\bm{x})=\bm x$, $b_1(\bm x)= b_1 \geq 0$, and  support is nonnegative $\Xi\subseteq \Re_+^n$ with i.i.d. random parameters $\trzeta$. That is, we  consider the following DRCCP \eqref{eq_drccp_infty}:
\begin{align*}
v^*=	\min_{\bm x\in\X}\left\{\bm c^\top\bm x\colon \inf_{\Pr\in\P_\infty}	\Pr\left\{ \trxi\colon  \bm x^\top{\trxi}\leq b_1\right\} \geq 1-\varepsilon  \right\}. 
\end{align*}
In this case, DRCCP \eqref{eq_drccp_infty} is equivalent to
\begin{subequations}
\begin{align}
v^*=		\min_{\bm x\in\X}\left\{\bm c^\top\bm x\colon	\hat{G}_{\theta}(\bm e^\top  \bm x)=1-\varepsilon -\Pr\left\{ \trzeta\colon \max_{\rxi\in\Re_+^n}\left\{ \sum_{i\in[n]}\xi_i x_i\colon \| \rxi-\trzeta\|_p \leq \theta  \right\}\leq b_1\right\} \leq 0 \right\}. \label{drccp_packing_infty}
\end{align}
And the lower-level $\alsoxus$ \eqref{eq_also_cvar_q_wass_weak_b} is equivalent to
\begin{align}
 \bm x^*\in\argmin_{\bm{x}\in \X,\bm{c}^\top \bm{x} \leq t} \left\{  \overline{F}_{\theta}(\bm e^\top \bm x )=\CVaR_{1-\varepsilon} \left\{-b_1+\max_{\rxi\in\Re_+^n}\left[ \sum_{i\in[n]}\xi_i x_i\colon \| \rxi-\trzeta\|_p \leq \theta  \right]\right\}  \right\}.\label{alsoxus_packing_infty}
\end{align}
\end{subequations}
In this case, we can show that $\alsoxs$ is exact.
\begin{corollary}
	\label{packing_exact_cor}
	Consider a single  DRCCP 
 \eqref{eq_drccp_infty} with affine mappings $\bm{a}_1(\bm{x})=\bm x$, $b_1(\bm x)=b_1 \geq 0$, the deterministic set $\X\subseteq \{0,1\}^n$, and the random parameters $\trzeta$ are i.i.d. and nonnegative. Then $\alsoxs$ is exact.
\end{corollary}
\begin{proof}
It is sufficient to show that functions $\hat{G}_{\theta}(\cdot)$ and $\overline{F}_{\theta}(\cdot)$ indeed exist and share the same monotonicity. We first notice that
\begin{align*}
\max_{\rxi\in\Re_+^n}\left\{ \sum_{i\in[n]}\xi_i x_i\colon \| \rxi-\trzeta\|_p \leq \theta  \right\}
=\sum_{i\in[n]}\tzeta_i x_i+\max_{\rxi\in\Re_+^n}\left\{ \sum_{i\in[n]}\xi_i x_i\colon \| \rxi\|_p \leq \theta  \right\}.
\end{align*}
Since $\{\tzeta_i\}_{i\in [n]}$ are i.i.d. nonnegative random parameters, for any $\bm{x}\in \X$ such that $\bm e^\top \bm x=\ell$, we have
\begin{align*}
\max_{\rxi\in\Re_+^n}\left\{ \sum_{i\in[n]}\xi_i x_i\colon \| \rxi-\trzeta\|_p \leq \theta  \right\}
\overset{\Pr_{\trzeta}}{\sim} \sum_{i\in[\ell]}\tzeta_i+\max_{\rxi\in\Re_+^\ell}\left\{ \sum_{i\in[\ell]}\xi_i\colon \| \rxi\|_p \leq \theta  \right\}.
\end{align*}
Hence, let us define
\begin{align*}
\hat{G}_{\theta}(\bm e^\top  \bm x)=\hat{G}_{\theta}(\ell)=1-\varepsilon -\Pr\left\{ \trzeta\colon \max_{\rxi\in\Re_+^n}\left\{ \sum_{i\in[n]}\xi_i x_i\colon \| \rxi-\trzeta\|_p \leq \theta  \right\}\leq b_1\right\}.
\end{align*}
Similarly, we also define
\begin{align*}
\overline{F}_{\theta}(\bm e^\top  \bm x)=\overline{F}_{\theta}(\ell)=\CVaR_{1-\varepsilon} \left\{ -b_1+\max_{\rxi\in\Re_+^n}\left[ \sum_{i\in[n]}\xi_i x_i\colon \| \rxi-\trzeta\|_p \leq \theta  \right]\right\}.
\end{align*}
Since all the random parameters $\{\tzeta_i\}_{i\in [n]}$ are nonnegative, we have
\begin{align*}
\sum_{i\in[\ell+1]}\tzeta_i+\max_{\rxi\in\Re_+^{\ell+1}}\left\{ \sum_{i\in[\ell+1]}\xi_i\colon \| \rxi\|_p \leq \theta  \right\}\geq \sum_{i\in[\ell]}\tzeta_i+\max_{\rxi\in\Re_+^\ell}\left\{ \sum_{i\in[\ell]}\xi_i\colon \| \rxi\|_p \leq \theta  \right\}
\end{align*}
almost surely. Thus, we have
\begin{align*}
\hat{G}_{\theta}(\ell)\leq \hat{G}_{\theta}(\ell+1),\overline{F}_{\theta}(\ell) \leq \overline{F}_{\theta}(\ell+1). 
\end{align*}
According to Theorem~\ref{exactness_alsoxus_monotone},  $\alsoxs$ is exact.
\QEDA
\end{proof}
We remark that the result in Corollary~\ref{packing_exact_cor} also holds for a covering DRCCP. That is, let us consider the following 
\begin{align*}
v^*=	\min_{\bm x\in\X}\left\{\bm c^\top\bm x\colon \inf_{\Pr\in\P_\infty}	\Pr\left\{ \trxi\colon \bm x^\top {\trxi} \geq b_1\right\} \geq 1-\varepsilon  \right\},
\end{align*}
where the random parameters $\{\rzeta_i\}_{i\in [n]}$ are i.i.d. and nonnegative and cost vector $\bm c$ is nonnegative. Let us denote $1-y_i=x_i$ for all $i\in [n]$. Then chance constrained covering problem is equivalent to
\begin{align*}
v^*=	\min_{(\bm e -\bm y)\in\X}\left\{\bm c^\top(\bm e -\bm y)\colon \inf_{\Pr\in\P_\infty}	\Pr\left\{ \trxi\colon  (\bm e -\bm y)^\top{\trxi}\geq b_1\right\} \geq 1-\varepsilon  \right\}. 
\end{align*}
As a result, the proof in Corollary~\ref{packing_exact_cor} simply follows, which is summarized below.
\begin{corollary}
	\label{covering_exact_cor}
	Consider a single  DRCCP with affine mappings $\bm{a}_1(\bm{x})=-\bm x$, $b_1(\bm x)=-b_1\leq 0$, the deterministic set $\X\subseteq \{0,1\}^n$, and the random parameters $\trzeta$ being i.i.d. and nonnegative. Then $\alsoxs$ is exact.
\end{corollary}

\subsection{Special Case III: Single DRCCPs with Empirical Reference Distribution}
In this subsection, we study the exactness of a single DRCCP with empirical reference distribution, 
where the affine mappings are $\bm{a}_1(\bm{x})=\bm x$, $b_1(\bm x)=b_1 \geq 0$, and the support is discrete. That is, we consider the following DRCCP:
\begin{align}
	v^*=\min_{\bm x\in\X}\left\{\bm c^\top\bm x\colon 	\Pr\left\{ \trzeta\colon \theta\left\|\bm x \right\|_*+ \bm x^\top {\trzeta}\leq b_1\right\} \geq 1-\varepsilon  \right\}. \label{drccp_discrete_infty}
\end{align}
\begin{corollary}
	\label{packing_exact}
	Consider a single  DRCCP \eqref{eq_drccp_infty} with affine mappings $\bm{a}_1(\bm{x})=\bm x$, $b_1(\bm x)=b_1 \geq 0$,  
 and the norm $\|\cdot\|$ is the generalized Mahalanobis norm associated with the matrix $\rsigma=\bm\mu\bm\mu^\top$ and $\bm \mu^\top \bm x\geq 0$ for any $\bm{x}\in \X$. Suppose the support is discrete with
$\Pr\{\trzeta = \zeta_i\bm\mu\}=p_i\geq 0$ for all $ i\in[N]$, where   $\sum_{i\in[N]}p_i=1$ and $0\leq \zeta_1<\zeta_2<\cdots<\zeta_N$ are scalars.
 Then $\alsoxs$ is exact.
\end{corollary}
\begin{proof}
We define $K\in[N]$ as $\sum_{i\in[K-1]}p_i<1-\varepsilon, \sum_{i\in[K]}p_i\geq 1-\varepsilon$. 
By definition, we have $\zeta_K>0$ and $\|\bm x\|_*=|\bm\mu^\top \bm x|$. According to the assumption that $\Pr\{\trzeta = \zeta_i\bm\mu\}=p_i\geq 0$ for all $ i\in[N]$, DRCCP \eqref{drccp_discrete_infty} can be simplified as
\begin{align*}
   v^*=\min_{\bm x\in\X}\left\{\bm c^\top\bm x\colon  \hat{G}_{\theta}(\bm \mu ^\top \bm x)=1-\varepsilon-\sum_{i\in [N]}p_i\I\left(\left(\zeta_i+ \theta\right)\bm \mu^\top \bm x   \leq b_1\right) \leq 0 \right\},
\end{align*}
and the corresponding lower-level $\alsoxus$ \eqref{eq_also_cvar_q_wass_weak_b} is equivalent to
\begin{align*}
	\bm x^*\in\argmin_{\bm{x}\in \X,\bm{c}^\top \bm{x} \leq t} \left\{ \overline{F}_{\theta}(\bm \mu ^\top \bm x)= \frac{1}{\varepsilon} \left[ \left(\sum_{i\in[K]}p_i - (1-\varepsilon)\right) \zeta_i + \sum_{j\in[K+1,N]}p_j \zeta_j  \right] \bm{\mu}^\top \bm{x} + \theta\|\bm\mu\|_2^{-1} \bm \mu^\top\bm x-b_1\right\}.
\end{align*}
Both functions $\hat{G}_{\theta}(\cdot)$ and $\overline F_{\theta}(\cdot)$ are monotone nondecreasing.
Therefore, the conditions in Theorem~\ref{exactness_alsoxus_monotone} are satisfied, and we conclude that $\alsoxs$ is exact.
\QEDA
\end{proof}

\section{Extensions:  DRCCPs under Type $q-$Wasserstein Ambiguity Set}
\label{sec_extension}
In this section, we extend our discussions to type $q-$Wasserstein ambiguity set with $q\in[1,\infty)$. We first provide equivalent reformulations. Then we show that under type $q-$Wasserstein ambiguity set, $\alsoxs$ can provide an optimal solution to a DRCCP. The results in this section rely on the equivalent reformulations of  DRCCP, $\alsox$, $\CVaR$ approximation, and $\alsoxs$ under type $q-$Wasserstein ambiguity set with $q\in[1,\infty)$, which are displayed in Appendix~\ref{sec_extension_eq_refor}.

\subsection{Comparisons of $\alsoxs$, $\alsoxus$, $\alsox$, and $\CVaR$ Approximation}
\label{extension_formulations}
As mentioned in Section~\ref{sec_better},  the main result of this paper in Theorem~\ref{also_cvar_better_alsox} cannot be extended to type $q-$Wasserstein ambiguity set with $q\in[1,\infty)$. That is, $\alsox$ and $\alsoxs$ are not comparable under type $q-$Wasserstein ambiguity set when $q\in [1,\infty)$. Below is an example.
\begin{example}\rm
\label{example_general_worse_convex}
	Consider a single DRCCP under type $1-$Wasserstein ambiguity set  with $\theta=1$ and $\|\cdot\|_*=\|\cdot\|_2$. Assume that the empirical distribution has $4$ equiprobable scenarios (i.e., $N=4$, $\Pr\{\trzeta=\rzeta^i\}=1/N$), risk parameter $\varepsilon=1/2$, the deterministic set $\X=\Re_+^3$, $\bm c=(-4, -2, -3)^\top$, function $\bm a_1(\bm x)^\top \rzeta-b_1({\bm x})=\bm x^\top\rzeta -3$, 
	$\rzeta^1=(4,6,3)^\top$, $\rzeta^2=(5,0,3)^\top$, $\rzeta^3=(2,1,4)^\top$, and $\rzeta^4=(0,2,5)^\top$. In this example, numerically, we can solve $\alsox$ $v_1^A$, $\alsoxs$ $v_1^{\as}$, and $\CVaR$ approximation $v_1^\CVaR$, where the approximated objective values are $v_1^A=-2.4929$, $v_1^{\as}=-2.4369$, and $v_1^\CVaR=-2.033$ with error bound $[-10^{-4}, 10^{-4}]$.
\QEDB
\end{example}
Albeit $\alsox$ and $\alsoxs$ are not comparable, following the similar proofs as those of Theorem~\ref{also_sharp_better_weak_sharp} and Theorem~\ref{also_sharp_better_cvar} in Section~\ref{sec_better}, we can prove that $\alsoxs$ is better than  $\CVaR$ approximation and $\alsoxs$ is better than  $\alsoxus$ under type $q-$Wasserstein ambiguity set with $q\in[1,\infty)$. Interested readers are referred to Appendix~\ref{sec_extension_eq_refor} for proofs.

\begin{restatable}{proposition}{alsosharpbettercvarqwass}\label{also_sharp_better_cvar_q_wass} 
Under type $q-$Wasserstein ambiguity set with $q\in[1,\infty)$, 
 suppose that for any objective upper bound t such that $t\geq \min_{\bm x\in\X}\bm c^\top \bm x$, $\alsoxs$ is better than  $\CVaR$ approximation.
\end{restatable}

\begin{restatable}{proposition}{alsosharpbetterweaksharpqwass}\label{also_sharp_better_weak_sharp_q_wass} 
Under type $q-$Wasserstein ambiguity set with $q\in[1,\infty)$, 
 suppose that for any objective upper bound t such that $t\geq \min_{\bm x\in\X}\bm c^\top \bm x$ and the lower-level $\alsoxus$ admits a unique optimal $\bm x$-solution, $\alsoxs$ is better than  $\alsoxus$.
\end{restatable}

\subsection{Exactness of $\alsoxs$}
\label{extension_exact}
In this subsection, we extend the discussion in Section~\ref{sec_exactness_elliptical} for the single DRCCP \eqref{eq_drccp_q_general} under elliptical distribution with affine mappings  $\bm{a}_1(\bm{x})=\bm x$ and $b_1(\bm x)=b_1$. Due to the page limit, we refer interested readers to Appendix~\ref{sec_extension_eq_refor} for detailed derivations.
Under type $q-$Wasserstein ambiguity set with elliptical reference distribution, DRCCP resorts to  
\begin{align}
v^*=\min_{\bm x\in\X}\left\{ \bm c^\top \bm  x\colon \bm \mu^\top \bm x+ \eta_q^*\sqrt{\bm{x}^\top\rsigma\bm{x}} -b_1\leq 0\right\}, \label{eq_drccp_q_general_formulation}
\end{align}
with
\begin{align*}
  \eta_q^* = \min_{\eta}\left\{ \eta\colon \int^\eta_{\mathrm{\Phi}^{-1}(1-\varepsilon)} \left(\eta-t\right)^q\bar{k}\hat g(t^2/2)dt\geq \theta^q, \eta\geq \mathrm{\Phi}^{-1}(1-\varepsilon)\right\}. 
\end{align*}
In this case, the lower-level $\alsoxus$ is equivalent to
\begin{align}
 \bm x^*\in\argmin_{\bm{x}\in \X,\bm{c}^\top \bm{x} \leq t} \left\{  \bm \mu^\top \bm x +
\left[\overline{G}\left(\left(\mathrm{\Phi}^{-1}(1-\varepsilon)\right)^2/2\right)/\varepsilon+\theta\varepsilon^{-\frac{1}{q}}\right]\sqrt{\bm x^\top\rsigma\bm x} -b_1 \right\}.\label{eq_alsoxs_q_b_formulation}
\end{align}
Let us make the same assumption as that in  Condition I of Section~\ref{sec_exactness_elliptical}, i.e., for a single DRCCP under an elliptical reference distribution, suppose that $\rsigma=\bm \mu\bm \mu^\top$ and $\bm \mu^\top \bm x\geq 0$ for any $\bm{x}\in \X$. Then DRCCP \eqref{eq_drccp_q_general_formulation} and the lower-level $\alsoxus$ \eqref{eq_alsoxs_q_b_formulation} can be simplified as
\begin{subequations}
	\begin{align*}
		&v^*=\min_{\bm x\in \X} \left\{ \bm c^\top \bm x\colon \hat{G}_{\theta}(\bm \mu ^\top \bm x)=\left(1+ \eta_q^* \right)\bm \mu ^\top \bm x   - b_1 \leq 0\right\},\\
  &\bm x^*\in\argmin_{\bm{x}\in \X} \left\{   \overline{F}_{\theta}(\bm \mu ^\top \bm x)= \left(1+\overline{G}\left(\left(\mathrm{\Phi}^{-1}(1-\varepsilon)\right)^2/2\right)/\varepsilon+\theta\varepsilon^{-\frac{1}{q}}\right)\bm \mu ^\top \bm x -b_1\colon  \bm{c}^\top \bm{x} \leq t\right\},
	\end{align*}
 \end{subequations}
respectively. The assumptions that  $\bm \mu^\top \bm x\geq 0$ for any $\bm{x}\in \X$ and $1+\eta_q^*\geq 0$ ensure that both functions $\hat{G}_{\theta}(\cdot)$ and $\overline F_{\theta}(\cdot)$ are monotone nondecreasing. Then the exactness result directly follows from Theorem~\ref{exactness_alsoxus_monotone} and Corollary~\ref{cor_drccp_exact_eq_c1}, which is summarized below.

\begin{corollary}
   Suppose that in a single DRCCP \eqref{eq_drccp_q_general_formulation},  the reference distribution ${\Pr_{\trzeta}}$ is elliptical with affine mappings $\bm{a}_1(\bm{x})=\bm x$, $b_1(\bm x)=b_1$, $\rsigma=\bm \mu\bm \mu^\top$, $\bm \mu^\top \bm x\geq 0$ for any $\bm{x}\in \X$, and $\eta_q^*\geq -1$. Then $\alsoxs$ is exact.
\end{corollary}
Similarly, let us make the same assumption as that in Condition II of Section~\ref{sec_exactness_elliptical}, i.e., for a single DRCCP under an elliptical reference distribution, suppose that 
 $\X\subseteq\{0,1\}^n$, $\bm{\mu}\geq \bm 0$, and $\rsigma= \Diag (\bm \mu)$. In this case, DRCCP \eqref{eq_drccp_q_general_formulation} and the lower-level $\alsoxus$ \eqref{eq_alsoxs_q_b_formulation} can be simplified as
 \begin{subequations}
	\begin{align*}
	& 	v^*=\min_{\bm x\in \{0,1\}^n} \left\{ \bm c^\top \bm x\colon \hat{G}_{\theta}(\bm \mu ^\top \bm x)=\bm \mu ^\top \bm x+\eta_q^*\sqrt{\bm \mu^\top \bm x} - b_1 \leq 0 \right\},\\
 & \bm x^*\in\argmin_{\bm{x}\in \{0,1\}^n} \left\{   \overline{F}_{\theta}(\bm \mu ^\top \bm x)=  \bm \mu ^\top \bm x +
\left[\overline{G}\left(\left(\mathrm{\Phi}^{-1}(1-\varepsilon)\right)^2/2\right)/\varepsilon+\theta\varepsilon^{-\frac{1}{q}}\right]\sqrt{\bm \mu^\top \bm x}  -b_1 \colon \bm{c}^\top \bm{x} \leq t \right\}.
	\end{align*}
 \end{subequations}
The assumptions  $\bm \mu\geq \bm 0$ and $\eta_q^*\geq 0$ guarantee that
both functions $\hat{G}_{\theta}(\cdot)$ and $\overline F_{\theta}(\cdot)$ are monotone nondecreasing. Then we have the following exactness result.
\begin{corollary}
  Suppose that in a single DRCCP \eqref{eq_drccp_q_general_formulation},  the reference distribution ${\Pr_{\trzeta}}$ is elliptical with affine mappings $\bm{a}_1(\bm{x})=\bm x$, $b_1(\bm x)=b_1$, $\X\subseteq\{0,1\}^n$, $\bm{\mu}\geq \bm 0$,  $\rsigma= \Diag (\bm \mu)$, and $\eta_q^*\geq 0$. Then $\alsoxs$ is exact.
\end{corollary}

\section{Numerical Study}
\label{numerical_study}
In this section, we  numerically  demonstrate the efficacy of the proposed methods. All the instances in this section are executed in Python 3.9 with calls to solver Gurobi (version 9.1.1 with default settings) on a personal PC with an Apple M1 Pro processor and 16G of memory. 


\subsection{Synthetic Cases}

We evaluate the differences among $\CVaR$ approximation, $\alsox$, and $\alsoxs$ using ``\textrm{Improvement from $\CVaR$ approximation}'' to denote the percentage of differences between the value of a proposed algorithm and $\CVaR$ approximation, i.e.,
\begin{align*}
	&\textrm{Improvement from $\CVaR$ approximation } (\%) \\
	&= \frac{ \CVaR \textrm{ approximation value}-\textrm{output value of a proposed algorithm} }{|\CVaR \textrm{ approximation value}|}\times 100.
\end{align*}
Typically, $\CVaR$ approximation is quite conservative. As a better alternative, we also use  ``\textrm{Improvement from $\alsox$}'' to denote the percentage of differences between the value of a proposed algorithm and $\alsox$ approximation, i.e.,
\begin{align*}
	\textrm{Improvement from $\alsox$ } (\%) = \frac{ \alsox \textrm{ value} -\textrm{output value of a proposed algorithm} }{|\alsox \textrm{ value}|}\times 100.
\end{align*}
We compare the performances of $\CVaR$ approximation, $\alsox$, and $\alsoxs$ of solving a single DRCCP 
with different sizes of data points $N=400,600,1000$, varying risk level $\varepsilon=0.10, 0.20$, fixed Wasserstein radius $\theta= 0.05$, and also different dimensions of decision variables $n=20,40,100$. In $\alsox$ and $\alsoxs$ algorithm, we use the optimal value from $\CVaR$ approximation as an initial upper bound $t_U$ and the quantile bound from \cite{ahmed2017nonanticipative} as an initial lower bound $t_L$. For each parametric setting, we generate $5$ random instances and report their average performance. 

We separate our discussions into type $\infty-$Wasserstein ambiguity set and type $2-$Wasserstein ambiguity set, respectively.

\noindent\textbf{Case I. Testing a DRCCP with type $\infty-$Wasserstein ambiguity set.} We split the discussions into the continuous case and the binary case.

\noindent\textbf{1.1 Continuous Case.}
Let us first consider the following DRCCP:
\begin{align*}
	v^*=\min _{\bm{x}} &\left\{ \bm{c}^\top \bm{x}\colon \bm{x}\in [0,1]^n, \frac{1}{N}\sum_{i\in[N]} \one \left[ \theta\left\| \bm x\right\|_2+\sum_{j\in[n]} \zeta^i_j x_j\leq b^i \right] \geq 1-\varepsilon\right\}.
\end{align*}
Above, we generate the samples $\{\rzeta^i\}_{i\in [N]}$ by assuming that the random parameters $\trzeta$ are discrete and i.i.d. uniformly distributed between $1$ and $80$. We set $\delta_1=10^{-2}$  in $\alsoxs$ Algorithm~\ref{alg_alsox_sharp}. For each random instance, we assume the cost vector $\bm{c}$ to be random integer with each entry uniformly distributed between $-30$ and $-1$. And we assume the random parameter $\tilde{b}$ is discrete and i.i.d. uniformly distributed between $1$ and $20$.
The numerical results are displayed in Table~\ref{tab_single_two_norm}. We see that although the computation time of $\alsoxs$ is comparable to that of $\alsox$, the solution quality of $\alsoxs$ is around 4\%-10\% better than that of $\alsox$. This demonstrates the effectiveness of our proposed $\alsoxs$. 

\begin{table}[htbp]
		\centering
\scriptsize
	\caption{Numerical Results of a DRCCP under Type $\infty-$Wasserstein Ambiguity Set with $\theta=0.05$}
	\renewcommand{\arraystretch}{1} 
	\label{tab_single_two_norm}
	\begin{tabular}{|c|c|r|r|rr|rr|}
		\hline
		\multirow{2}{*}{$\varepsilon$} & \multicolumn{1}{c|}{\multirow{2}{*}{$N$}} & \multicolumn{1}{c|}{\multirow{2}{*}{$n$}} & $\CVaR$  & \multicolumn{2}{c|}{$\alsox$}               & \multicolumn{2}{c|}{$\alsoxs$}            \\ \cline{4-8} 
		& \multicolumn{1}{c|}{}                   & \multicolumn{1}{c|}{}                   & Time (s) & \multicolumn{1}{r|}{Time (s)} & \makecell{Improvement from \\ $\CVaR$ approximation (\%)}  & \multicolumn{1}{r|}{Time (s)} &\makecell{Improvement from \\ $\alsox$ (\%)} \\ \hline
		\multirow{9}{*}{0.10}     & \multirow{3}{*}{400}                    & 20                                      & 0.15  & \multicolumn{1}{r|}{1.53}  & 13.28       & \multicolumn{1}{r|}{1.56}  & 8.73        \\ \cline{3-8} 
		&                                         & 40                                      & 0.29  & \multicolumn{1}{r|}{3.26}  & 11.29       & \multicolumn{1}{r|}{2.73}  & 5.86        \\ \cline{3-8} 
		&                                         & 100                                     & 0.73  & \multicolumn{1}{r|}{9.12}  & 14.59       & \multicolumn{1}{r|}{13.62} & 6.77        \\ \cline{2-8} 
		& \multirow{3}{*}{600}                    & 20                                      & 0.39  & \multicolumn{1}{r|}{2.27}  & 8.48        & \multicolumn{1}{r|}{2.78}  & 6.06        \\ \cline{3-8} 
		&                                         & 40                                      & 0.35  & \multicolumn{1}{r|}{4.64}  & 9.22        & \multicolumn{1}{r|}{4.41}   & 6.38        \\ \cline{3-8} 
		&                                         & 100                                     & 0.78  & \multicolumn{1}{r|}{11.82} & 9.79        & \multicolumn{1}{r|}{12.73} & 5.69        \\ \cline{2-8} 
		& \multirow{3}{*}{1000}                   & 20                                      & 0.25  & \multicolumn{1}{r|}{4.77}  & 10.76       & \multicolumn{1}{r|}{4.81}  & 9.16        \\ \cline{3-8} 
		&                                         & 40                                      & 0.72  & \multicolumn{1}{r|}{7.20}   & 9.16        & \multicolumn{1}{r|}{7.35}  & 5.10         \\ \cline{3-8} 
		&                                         & 100                                     & 1.62  & \multicolumn{1}{r|}{20.15} & 6.31        & \multicolumn{1}{r|}{21.98} & 3.38        \\ \hline \hline
		\multirow{9}{*}{0.20}     & \multirow{3}{*}{400}                    & 20                                      & 0.22  & \multicolumn{1}{r|}{2.75}  & 12.94       & \multicolumn{1}{r|}{2.43}  & 9.20         \\ \cline{3-8} 
		&                                 & 40                                      & 0.41  & \multicolumn{1}{r|}{2.74}  & 13.53       & \multicolumn{1}{r|}{2.34}  & 6.28        \\ \cline{3-8} 
		&                                         & 100                                     & 0.66  & \multicolumn{1}{r|}{4.79}  & 5.39        & \multicolumn{1}{r|}{5.51}  & 6.02        \\ \cline{2-8} 
		& \multirow{3}{*}{600}                    & 20                                      & 0.49  & \multicolumn{1}{r|}{3.49}  & 6.49        & \multicolumn{1}{r|}{3.24}  & 7.91        \\ \cline{3-8} 
		&                                         & 40                                      & 0.54  & \multicolumn{1}{r|}{6.18}  & 5.33        & \multicolumn{1}{r|}{4.07}  & 5.97        \\ \cline{3-8} 
		&                                         & 100                                     & 1.43  & \multicolumn{1}{r|}{12.87} & 3.82        & \multicolumn{1}{r|}{10.66} & 7.76        \\ \cline{2-8} 
		& \multirow{3}{*}{1000}                   & 20                                      & 0.52  & \multicolumn{1}{r|}{6.50}   & 8.59        & \multicolumn{1}{r|}{5.26}  & 5.76        \\ \cline{3-8} 
		&                                         & 40                                      & 0.84  & \multicolumn{1}{r|}{7.68}  & 4.21        & \multicolumn{1}{r|}{7.53}  & 5.09        \\ \cline{3-8} 
		&                                         & 100                                     & 2.38  & \multicolumn{1}{r|}{21.92} & 4.76        & \multicolumn{1}{r|}{18.47} & 3.73        \\ \hline
	\end{tabular}
\end{table}
\noindent\textbf{1.2 Binary Case.}
Let us first consider the following DRCCP:
\begin{align*}
	v^*=\min _{\bm{x}} &\left\{ \bm{c}^\top \bm{x}\colon \bm{x}\in \{0,1\}^n, \frac{1}{N}\sum_{i\in[N]} \one \left[ \theta\left\| \bm x\right\|_1+\sum_{j\in[n]} \zeta^i_j x_j\leq b^i \right] \geq 1-\varepsilon\right\}.
\end{align*}
Above, we generate the samples $\{\rzeta^i\}_{i\in [N]}$ by assuming that the random parameters $\trzeta$ are discrete and i.i.d. uniformly distributed between $-10$ and $20$. We set $\delta_1=0.5$  in $\alsoxs$ Algorithm~\ref{alg_alsox_sharp}. For each random instance, we assume the cost vector $\bm{c}$ to be random integer, with each entry uniformly distributed between $-10$ and $-1$. And we assume the random parameter $\tilde{b}$ is discrete and i.i.d. uniformly distributed between $1$ and $200$. The numerical results for this case are displayed in Table~\ref{tab_single_bianry_dual_one_norm}. Similar to the results in Table~\ref{tab_single_two_norm}, we conclude that $\alsoxs$ enhances the solution quality from $\alsox$ by around 3\%-7\% improvement with a comparable computation time. 

\begin{table}[htbp]
	\centering
   \scriptsize
\caption{Numerical Results of a Binary DRCCP under Type $\infty-$Wasserstein Ambiguity Set with $\theta=0.05$}
\renewcommand{\arraystretch}{1} 
\label{tab_single_bianry_dual_one_norm}
	\begin{tabular}{|c|c|r|r|rr|rr|}
		\hline
		\multirow{2}{*}{$\varepsilon$} & \multicolumn{1}{c|}{\multirow{2}{*}{$N$}} & \multicolumn{1}{c|}{\multirow{2}{*}{$n$}} & $\CVaR$   & \multicolumn{2}{c|}{$\alsox$}               & \multicolumn{2}{c|}{$\alsoxs$}            \\ \cline{4-8} 
		& \multicolumn{1}{c|}{}                   & \multicolumn{1}{c|}{}                   & Time (s) & \multicolumn{1}{r|}{Time (s)} & \makecell{Improvement from \\ $\CVaR$ approximation (\%)}  & \multicolumn{1}{r|}{Time (s)} &\makecell{Improvement from \\ $\alsox$ (\%)} \\ \hline
		\multirow{9}{*}{0.10}     & \multirow{3}{*}{400}                    & 20              
		& 0.30 & \multicolumn{1}{r|}{1.97} &  5.16  & \multicolumn{1}{r|}{2.54} & 6.67 
		                         \\ \cline{3-8} 
		&                                         & 40       & 0.31 & \multicolumn{1}{r|}{5.65} & 8.03 & \multicolumn{1}{r|}{8.34} & 7.10                                     \\ \cline{3-8} 
		&                                         & 100      & 0.24 & \multicolumn{1}{r|}{7.09} & 3.09 & \multicolumn{1}{r|}{7.97} & 4.36                                      \\ \cline{2-8} 
		& \multirow{3}{*}{600}                    & 20        & 0.31 & \multicolumn{1}{r|}{5.28} & 7.49 & \multicolumn{1}{r|}{4.75} & 5.33                                       \\ \cline{3-8} 
		&                                         & 40          & 1.29  & \multicolumn{1}{r|}{12.18}  & 5.11 &  \multicolumn{1}{r|}{15.50} & 7.14                                   \\ \cline{3-8} 
		&                                         & 100       & 1.35 & \multicolumn{1}{r|}{14.35} & 5.19 & \multicolumn{1}{r|}{16.23} & 5.56                                     \\ \cline{2-8} 
		& \multirow{3}{*}{1000}                   & 20     &  0.95 &  \multicolumn{1}{r|}{5.82} & 3.51  &  \multicolumn{1}{r|}{6.29} & 2.78                                        \\ \cline{3-8} 
		&                                         & 40          &  0.52 &  \multicolumn{1}{r|}{16.13}  & 3.63  &  \multicolumn{1}{r|}{17.48}  & 2.83                                  \\ \cline{3-8} 
		&                                         & 100         & 0.61 &  \multicolumn{1}{r|}{24.69} &  4.82 &  \multicolumn{1}{r|}{26.11} & 3.19                                 \\ \hline \hline
		\multirow{9}{*}{0.20}     & \multirow{3}{*}{400}                    & 20              &  0.40 &  \multicolumn{1}{r|}{5.75} & 4.73 &  \multicolumn{1}{r|}{5.59} & 7.35                                \\ \cline{3-8} 
		&                                         & 40     & 0.57 &  \multicolumn{1}{r|}{8.14} & 4.27 &  \multicolumn{1}{r|}{9.43} & 7.32                                       \\ \cline{3-8} 
		&                                         & 100     & 0.64 &  \multicolumn{1}{r|}{9.34} & 4.04 &  \multicolumn{1}{r|}{11.48} & 5.41                                     \\ \cline{2-8} 
		& \multirow{3}{*}{600}                    & 20      &  0.42 &  \multicolumn{1}{r|}{5.58} & 4.79 &  \multicolumn{1}{r|}{5.24} & 5.08                                    \\ \cline{3-8} 
		&                                         & 40         & 0.80 &  \multicolumn{1}{r|}{16.92} & 3.37  &  \multicolumn{1}{r|}{14.95} & 6.92                               \\ \cline{3-8} 
		&                                         & 100        &  1.45 &  \multicolumn{1}{r|}{21.12} & 4.58 &  \multicolumn{1}{r|}{21.06} & 6.36                                     \\ \cline{2-8} 
		& \multirow{3}{*}{1000}                   & 20     & 1.04 &  \multicolumn{1}{r|}{7.51} & 3.49 &  \multicolumn{1}{r|}{7.45} & 6.25                                     \\ \cline{3-8} 
		&                                         & 40      & 1.26 &  \multicolumn{1}{r|}{24.39} & 4.12 &  \multicolumn{1}{r|}{22.99} & 6.24                                     \\ \cline{3-8} 
		&                                         & 100       & 2.70 &  \multicolumn{1}{r|}{32.93}  & 2.87 & \multicolumn{1}{r|}{33.65} & 5.56                                 \\ \hline
	\end{tabular}
\end{table}

\noindent\textbf{Case II. Testing a DRCCP with type $2-$Wasserstein ambiguity set.} We split the discussions into the continuous case and the binary case.

\noindent\textbf{2.1 Continuous Case.}
Let us first consider the following DRCCP:
\begin{align*}
	v^*=\min _{\bm{x}} &\left\{ \bm{c}^\top \bm{x}\colon \bm{x}\in [0,1]^n, \inf_{\Pr\in\P_2}\Pr\left\{\trxi\colon \bm x^\top \trxi \leq b  \right\} \geq 1-\varepsilon\right\}.
\end{align*}
Note that this DRCCP may not have an MIP reformulation (see, e.g., \cite{jiang2022dfo}).
Above, we generate the samples $\{\rzeta^i\}_{i\in [N]}$ by assuming that the random parameters $\trzeta$ are discrete and i.i.d. uniformly distributed between $1$ and $80$. We set $\delta_1=10^{-2}$  in $\alsoxs$ Algorithm~\ref{alg_alsox_sharp}. For each random instance, we assume $b=10$ and assume the cost vector $\bm{c}$ to be random integer with each entry uniformly distributed between $-20$ and $-1$.  The numerical results are displayed in Table~\ref{tab_single_two_norm_type_2}. We show that $\alsoxs$ yields around 5\%-11\% improvement. We also notice that $\alsox$ can improve the solution of $\CVaR$ approximation. 

\begin{table}[htbp]
	\centering
	\scriptsize
	\caption{Numerical Results of a DRCCP under Type $2-$Wasserstein Ambiguity Set with $\theta=0.50$}
	\renewcommand{\arraystretch}{1} 
	\label{tab_single_two_norm_type_2}
	\begin{tabular}{|c|c|r|r|rr|rr|}
		\hline
		\multirow{2}{*}{$\varepsilon$} & \multicolumn{1}{c|}{\multirow{2}{*}{$N$}} & \multicolumn{1}{c|}{\multirow{2}{*}{$n$}} & $\CVaR$  & \multicolumn{2}{c|}{$\alsox$}               & \multicolumn{2}{c|}{$\alsoxs$}            \\ \cline{4-8} 
		& \multicolumn{1}{c|}{}                   & \multicolumn{1}{c|}{}                   & Time (s) & \multicolumn{1}{r|}{Time (s)} & \makecell{Improvement from \\ $\CVaR$ approximation (\%)}  & \multicolumn{1}{r|}{Time (s)} &\makecell{Improvement from \\ $\CVaR$ approximation (\%)} \\ \hline
		\multirow{9}{*}{0.10}     & \multirow{3}{*}{400}                    & 20                                      &  0.09 & \multicolumn{1}{r|}{0.90 }  &7.21      & \multicolumn{1}{r|}{ 0.81}  & 7.22       \\ \cline{3-8} 
		&                                         & 40                                      & 0.25  & \multicolumn{1}{r|}{ 2.42}  & 6.43       & \multicolumn{1}{r|}{2.06  }  & 6.54      \\ \cline{3-8} 
		&                                         & 100                                     & 0.55  & \multicolumn{1}{r|}{5.76 }  & 4.73       & \multicolumn{1}{r|}{  3.97} & 4.73        \\ \cline{2-8} 
		& \multirow{3}{*}{600}                    & 20                                      & 0.20  & \multicolumn{1}{r|}{ 1.40}  & 8.60        & \multicolumn{1}{r|}{ 1.12}  & 8.60        \\ \cline{3-8} 
		&                                         & 40                                      & 0.38 & \multicolumn{1}{r|}{3.83 }  & 6.92        & \multicolumn{1}{r|}{2.72 }   & 6.92        \\ \cline{3-8} 
		&                                         & 100                                     & 1.16  & \multicolumn{1}{r|}{12.46 } & 5.27        & \multicolumn{1}{r|}{ 8.96} & 5.27       \\ \cline{2-8} 
		& \multirow{3}{*}{1000}                   & 20                                      & 0.23  & \multicolumn{1}{r|}{2.24 }  & 8.78      & \multicolumn{1}{r|}{1.88 }  & 8.82       \\ \cline{3-8} 
		&                                         & 40                                      & 0.76  & \multicolumn{1}{r|}{ 5.62}   & 7.82        & \multicolumn{1}{r|}{ 4.39}  & 7.94         \\ \cline{3-8} 
		&                                         & 100                                     & 1.48  & \multicolumn{1}{r|}{18.82 } & 5.82        & \multicolumn{1}{r|}{ 13.36} & 5.87       \\ \hline \hline
		\multirow{9}{*}{0.20}     & \multirow{3}{*}{400}                    & 20                                      & 0.25  & \multicolumn{1}{r|}{ 0.93}  &     10.45   & \multicolumn{1}{r|}{ 0.82}  &   10.61      \\ \cline{3-8} 
		&                                         & 40                                      & 0.45  & \multicolumn{1}{r|}{2.89}  &10.17          & \multicolumn{1}{r|}{ 2.16 }  &   10.24     \\ \cline{3-8} 
		&                                         & 100                                     & 0.66  & \multicolumn{1}{r|}{6.17    }  &  6.14    & \multicolumn{1}{r|}{  4.03} &     6.32    \\ \cline{2-8} 
		& \multirow{3}{*}{600}                    & 20                                      &  0.13  & \multicolumn{1}{r|}{1.41}  &   10.10  & \multicolumn{1}{r|}{ 1.19}  &    10.10     \\ \cline{3-8} 
		&                                         & 40                                      & 0.32  & \multicolumn{1}{r|}{3.93  }  & 8.46        & \multicolumn{1}{r|}{ 2.83 }  &   8.62      \\ \cline{3-8} 
		&                                         & 100                                     & 1.14  & \multicolumn{1}{r|}{ 13.38} & 6.77         & \multicolumn{1}{r|}{ 9.30} &    6.76    \\ \cline{2-8} 
		& \multirow{3}{*}{1000}                   & 20                                      & 0.28  & \multicolumn{1}{r|}{2.35 }   &    10.70     & \multicolumn{1}{r|}{ 2.05}  &     10.70   \\ \cline{3-8} 
		&                                         & 40                                      & 0.56  & \multicolumn{1}{r|}{5.67 }   & 10.33     & \multicolumn{1}{r|}{ 4.29}  &   10.33     \\ \cline{3-8} 
		&                                         & 100                                     & 1.83  & \multicolumn{1}{r|}{ 20.01  } & 6.82      & \multicolumn{1}{r|}{14.30} &   6.92    \\ \hline
	\end{tabular}
\end{table}

\noindent\textbf{2.2 Binary Case.}
Let us first consider the following DRCCP:
\begin{align*}
		v^*=\min _{\bm{x}} &\left\{ \bm{c}^\top \bm{x}\colon \bm{x}\in \{0,1\}^n, \inf_{\Pr\in\P_2}\Pr\left\{ \trxi\colon\bm x^\top \trxi \leq b  \right\} \geq 1-\varepsilon\right\}.
\end{align*}
Above, we generate the samples $\{\rzeta^i\}_{i\in [N]}$ by assuming that the random parameters $\trzeta$ are discrete and i.i.d. uniformly distributed between $-20$ and $50$. We set $\delta_1=0.5$  in $\alsoxs$ Algorithm~\ref{alg_alsox_sharp}. For each random instance, we assume $b=400$ and assume the cost vector $\bm{c}$ to be random integer, with each entry uniformly distributed between $-20$ and $-10$.  In this numerical experiment, we suppose the dimensions of decision variables $n=20,40$. The numerical results for this case are displayed in Table~\ref{tab_single_bianry_dual_type_2_one_norm}. We conclude that $\alsoxs$ enhances the solution quality from $\CVaR$ approximation by around 5-10\% improvement.

\begin{table}[htbp]
	\centering
	\scriptsize
	\caption{Numerical Results of a  Binary DRCCP under Type $2-$Wasserstein Ambiguity Set with $\theta=0.20$}
	\renewcommand{\arraystretch}{1} 
	\label{tab_single_bianry_dual_type_2_one_norm}
	\begin{tabular}{|c|c|r|r|rr|rr|}
		\hline
		\multirow{2}{*}{$\varepsilon$} & \multicolumn{1}{c|}{\multirow{2}{*}{$N$}} & \multicolumn{1}{c|}{\multirow{2}{*}{$n$}} & $\CVaR$  & \multicolumn{2}{c|}{$\alsox$}               & \multicolumn{2}{c|}{$\alsoxs$}            \\ \cline{4-8} 
		& \multicolumn{1}{c|}{}                   & \multicolumn{1}{c|}{}                   & Time (s) & \multicolumn{1}{r|}{Time (s)} & \makecell{Improvement from \\ $\CVaR$ approximation (\%)}  & \multicolumn{1}{r|}{Time (s)} &\makecell{Improvement from \\ $\CVaR$ approximation (\%)} \\ \hline
		\multirow{9}{*}{0.10}     & \multirow{3}{*}{400}                    & 20              
		& 0.23 & \multicolumn{1}{r|}{1.55} &  5.80  & \multicolumn{1}{r|}{1.43} &6.02
		\\ \cline{3-8} 
		&                                         & 40       & 2.81 & \multicolumn{1}{r|}{34.11} & 8.23 & \multicolumn{1}{r|}{25.23} & 7.10                                     \\ \cline{2-8} 
		& \multirow{3}{*}{600}                    & 20        & 0.36 & \multicolumn{1}{r|}{2.60 } & 5.19 & \multicolumn{1}{r|}{ 2.69} & 5.19                                      \\ \cline{3-8} 
		&                                         & 40          & 5.77 & \multicolumn{1}{r|}{57.60 }  & 9.81 &  \multicolumn{1}{r|}{55.39 } & 10.19                                   \\ \cline{2-8} 
		& \multirow{3}{*}{1000}                   & 20     &  1.44 &  \multicolumn{1}{r|}{5.29 } & 7.92  &  \multicolumn{1}{r|}{4.78 } & 8.19                                        \\ \cline{3-8} 
		&                                         & 40          &  12.45 &  \multicolumn{1}{r|}{ 86.86}  & 9.67  &  \multicolumn{1}{r|}{85.98 }  & 9.67                                \\ \cline{2-8} 
		\hline \hline
		\multirow{9}{*}{0.20}     & \multirow{3}{*}{400}                    & 20              &  0.24 &  \multicolumn{1}{r|}{ 1.80} & 7.67 &  \multicolumn{1}{r|}{ 1.65} & 7.67                              \\ \cline{3-8} 
		&                                         & 40     & 3.49 &  \multicolumn{1}{r|}{ 35.70} & 9.46 &  \multicolumn{1}{r|}{ 27.71} & 9.46                                       \\ \cline{2-8} 
		& \multirow{3}{*}{600}                    & 20      &  0.57 &  \multicolumn{1}{r|}{ 3.14} & 7.73 &  \multicolumn{1}{r|}{ 3.08} & 7.82                                  \\ \cline{3-8} 
		&                                         & 40         & 13.85 &  \multicolumn{1}{r|}{ 60.79} & 9.82  &  \multicolumn{1}{r|}{ 74.74} & 11.50                              \\ \cline{2-8} 
		& \multirow{3}{*}{1000}                   & 20     & 1.01 &  \multicolumn{1}{r|}{5.61 } & 7.26 &  \multicolumn{1}{r|}{4.97 } & 7.66                                     \\ \cline{3-8} 
		&                                         & 40      & 16.48 &  \multicolumn{1}{r|}{82.65 } & 8.73 &  \multicolumn{1}{r|}{ 87.72} & 9.87                                     \\ \hline
	\end{tabular}
\end{table}

\subsection{Application: Resource Allocation in Wireless Communication}

 We compare $\alsoxs$ and the exact method in the wireless communication network problem, where we can use $\alsoxs$ to minimize the energy consumed. Specifically,
we consider a predictive resource allocation problem for energy-efficient video streaming (see, e.g., \cite{atawia2015chance,atawia2016joint}), where chance constraints are employed to ensure a high quality of service for each user. The objective of this problem is to minimize the energy consumption in transmitting the video content to the users while satisfying the chance constraints. The problem is formally formulated as
\begin{align}
\min_{\bm x\in[0,1]^{n\times T}} \left\{ \sum_{t\in[T]}\sum_{i\in[n]} x_{i,t}\colon\right.
&\inf_{\Pr\in\P_\infty}\Pr\left\{ \trxi\colon \sum_{t'\in[t]} \tilde{\xi}_{i,t'}x_{i,t'} \geq D_{i,t},\forall t\in[T] \right\}\geq 1-\varepsilon, \forall i\in[n], \notag\\
&\left.
\sum_{i\in[n]} x_{i,t}\leq 1, \forall t\in[T] \right\}, \label{allocation_obj}
\end{align}
where $x_{i,t}$ denotes resource allocation decision at time slot $t$ to user $i$, $n$ denotes the number of all users, $T$ denotes the number of time slots, and $D_{i,t}$ denotes the demand for each user $i\in[n]$ up to time $t\in [T]$. The random parameter ${\xi}_{i,t}$ denotes the random amount of available rate for user $i\in[n]$ at time slot $t$.

Above, we generate the samples $\{\rzeta^i\}_{i\in [N]}$ by assuming that the random parameters $\trzeta$ are discrete and i.i.d. uniformly distributed between $20$ and $40$. We set $\delta_1=10^{-1}$ in $\alsoxs$ Algorithm~\ref{alg_alsox_sharp}. And we assume for each user $i\in[n]$ up to time $t\in [T]$, the demand $D_{i,t}$ is $D_{i,t}=tD$ with $D=1.0, 1.5$.  We consider the number of data samples $N=56,72$, the risk level $\varepsilon=0.20, 0.30$, the Wasserstein radius $\theta= 0.50$, the number of users $n=8$, and the number of time slots $T=60$. We generate $5$ random instances to find their average performance. 

The proposed $\alsoxs$ can effectively identify better feasible solutions than the exact Big-M model with a much shorter solution time, which is typically required in many wireless communication applications. Since we consider the number of time slots $T=60$s, for a fair comparison, we set the time limit of each instance to {$60$s (i.e., 1 minute)}, and we use $``\textrm{UB}" $ and $``\textrm{LB}" $ to denote the best upper bound and the best lower bound found by the $\text{Big-M}$ model within the one-minute time limit. Since we may not be able to solve the $\text{Big-M}$ model to optimality in one minute, we use GAP to denote its optimality gap as $\textrm{GAP } (\%)=(| \textrm{UB} -\textrm{LB} |)/(|\textrm{LB}|)\times 100$, and we use the term ``Improvement from Big-M model'' to denote the solution quality improvement of $\alsoxs$, i.e., $\textrm{Improvement from Big-M model } (\%)= (\textrm{UB} - \alsoxs \textrm{ value})/(|\textrm{UB} |)\times 100$. 
 The numerical results for this case are displayed in Table~\ref{tab_resource_dual_norm_1_norm}. We find that $\alsoxs$ can provide better solutions than the $\text{Big-M}$ model in a much shorter time, which validates the efficacy of our proposed methods. We remark that  we can design real-time algorithms by using $\alsoxs$ as a future study (see, e.g., \cite{atawia2018utilization,atawia2017robust,atawia2016joint}).

  \begin{table}[htbp]
	\centering
	\scriptsize
	\caption{{Comparisons of $\alsoxs$ and $\text{Big-M}$ Model in Resource Allocation Problem \eqref{allocation_obj}}}
	\renewcommand{\arraystretch}{1} 
	\label{tab_resource_dual_norm_1_norm}
	\begin{tabular}{|c|c|r|rr|rr|}
		\hline
		\multirow{2}{*}{$D$}   & \multirow{2}{*}{$\varepsilon$} & \multirow{2}{*}{$N$}  & \multicolumn{2}{c|}{Big-M Model}   
	& \multicolumn{2}{c|}{$\alsoxs$}                         \\ \cline{4-7} 
		&                                & \multicolumn{1}{c|}{}                   & \multicolumn{1}{l|}{Gap (\%)} & \multicolumn{1}{l|}{Time (s)} & \multicolumn{1}{c|}{ \makecell{Improvement from \\ Big-M model (\%)}}  & \multicolumn{1}{l|}{Time (s)} \\ \hline
		\multirow{4}{*}{1.0} & \multirow{2}{*}{0.20}          & 56                                       & \multicolumn{1}{r|}{ 19.05}    &         60                   & \multicolumn{1}{r|}{0.10}        &       21.41                     \\ \cline{3-7} 
		&                                & 72                                      & \multicolumn{1}{r|}{ 19.88 }    &      60                      & \multicolumn{1}{r|}{ 0.12}        &   36.48                         \\ \cline{2-7} 
		& \multirow{2}{*}{0.30}          & 56                                       & \multicolumn{1}{r|}{25.53}    &         60                   & \multicolumn{1}{r|}{0.10}        &    25.35                        \\ \cline{3-7} 
		&                                & 72                                      & \multicolumn{1}{r|}{28.13}    &          60                  & \multicolumn{1}{r|}{0.21}        &         37.07                   \\ \hline
		\multirow{4}{*}{1.5} & \multirow{2}{*}{0.20}          & 56                                       & \multicolumn{1}{r|}{19.60}    &     60                       & \multicolumn{1}{r|}{0.14}        &  34.19                          \\ \cline{3-7} 
		&                                & 72                                      & \multicolumn{1}{r|}{20.68}    &    60                        & \multicolumn{1}{r|}{0.19}        &     41.42                     \\ \cline{2-7} 
		& \multirow{2}{*}{0.30}          & 56                                       & \multicolumn{1}{r|}{27.50 }    &            60                & \multicolumn{1}{r|}{0.23 }        &       36.93                     \\ \cline{3-7} 
		&                                & 72                                      & \multicolumn{1}{r|}{28.84 }    &     60                       & \multicolumn{1}{r|}{0.39}        &    42.06                        \\ \hline
	\end{tabular}	
 \end{table}

Subsequently, we demonstrate how to use the $\alsoxs$ result to save energy-consuming. For illustration, we consider the following parametric setting $D=1.0$, $N=56$, $\varepsilon=0.20$, and $n=8$. We first numerically choose a proper Wasserstein radius. We use generated data to solve the DRCCP with $\alsoxs$, the DRCCP with $\text{Big-M}$ model, and the regular CCP counterpart  (i.e., $\theta=0$) with $\alsoxs$, respectively. Then we generate new samples with the same sample size and obtain  95\% confidence intervals by plugging the solutions in the regular CCP to calculate the corresponding probability that the constraints are satisfied. We repeat the same procedure for a list of $\theta$ values and select the smallest $\theta$ that the confidence interval of the violation probability in the DRCCP is beyond that of the regular CCP. Specifically, to select the smallest $\theta$, we take the following steps: (i) for each $\theta\in\{0.1, 0.2, \cdots, 0.9,1.0\}$, we generate $N=56$ scenarios and solve the DRCCP and its regular CCP counterpart; (ii) generate $N=56$ new scenarios with the same parameters; (iii) plug the solution from part (i) into the newly generated scenarios and calculate the probability that the constraints are satisfied; (iv) repeat previous procedures $50$ times and output the  95\% confidence intervals; and (v) choose the smallest $\theta$ such that the confidence interval of the DRCCP is entirely above that of the regular CCP counterpart.  The result is shown in Figure~\ref{Fig_wireless_allocation_radius}. The DRCCP using $\alsoxs$ with Wasserstein radius $\theta=0.7$ can guarantee that the chance constraints are satisfied with a high probability. In contrast, the best Wasserstein radius is $\theta=1.0$ when applying the result from the $\text{Big-M}$ model. In this way, we conclude that compared with the $\text{Big-M}$ model, the average energy saving using the $\alsoxs$ is $1.93\%$. 


With the best-tuned Wasserstein radius, we compare the performances of the solutions from the $\text{Big-M}$ model and $\alsoxs$ by generating $100$ new scenarios to evaluate the probability that the constraints are satisfied and output the corresponding  95\% confidence intervals. The testing setting is the same as that of the training setting above, except that we assume that random parameters $\trzeta$ are discrete and i.i.d. uniformly distributed between $20\times (1-1.8\rho)$ and $40\times(1+0.4\rho)$ with $\rho\in[0.02,0.20]$, where the value $\rho$ represents the noise level in the training distribution. The result is displayed in Figure~\ref{Fig_wireless_allocation_evaluation}. It is seen that compared with the $\text{Big-M}$ model, $\alsoxs$ yields ideal lower constraint violations. Specifically, when the noise level is small, i.e., $0.02\leq \rho\leq 0.16$ in Figure~\ref{Fig_wireless_allocation_evaluation}, $\alsoxs$ often guarantees a lower violation of constraints. However, when the noise level is large, i.e., $\rho=0.20$ in Figure~\ref{Fig_wireless_allocation_evaluation}, both $\alsoxs$ and $\text{Big-M}$ model cannot provide the chance constraint guarantee. This result further demonstrates that using $\alsoxs$ with the best-tune Wasserstein radius can be better or at least achieve the similar constraint violation probability as the $\text{Big-M}$ model. As mentioned in the previous paragraph, $\alsoxs$ provides a better solution with lower energy consumption. This demonstrates the better solution quality of $\alsoxs$ compared to the $\text{Big-M}$ model.


	\begin{figure}[htbp]
	\subfloat[Comparisons of Tuning Wasserstein Radii]{
		\includegraphics[width=.5\textwidth]{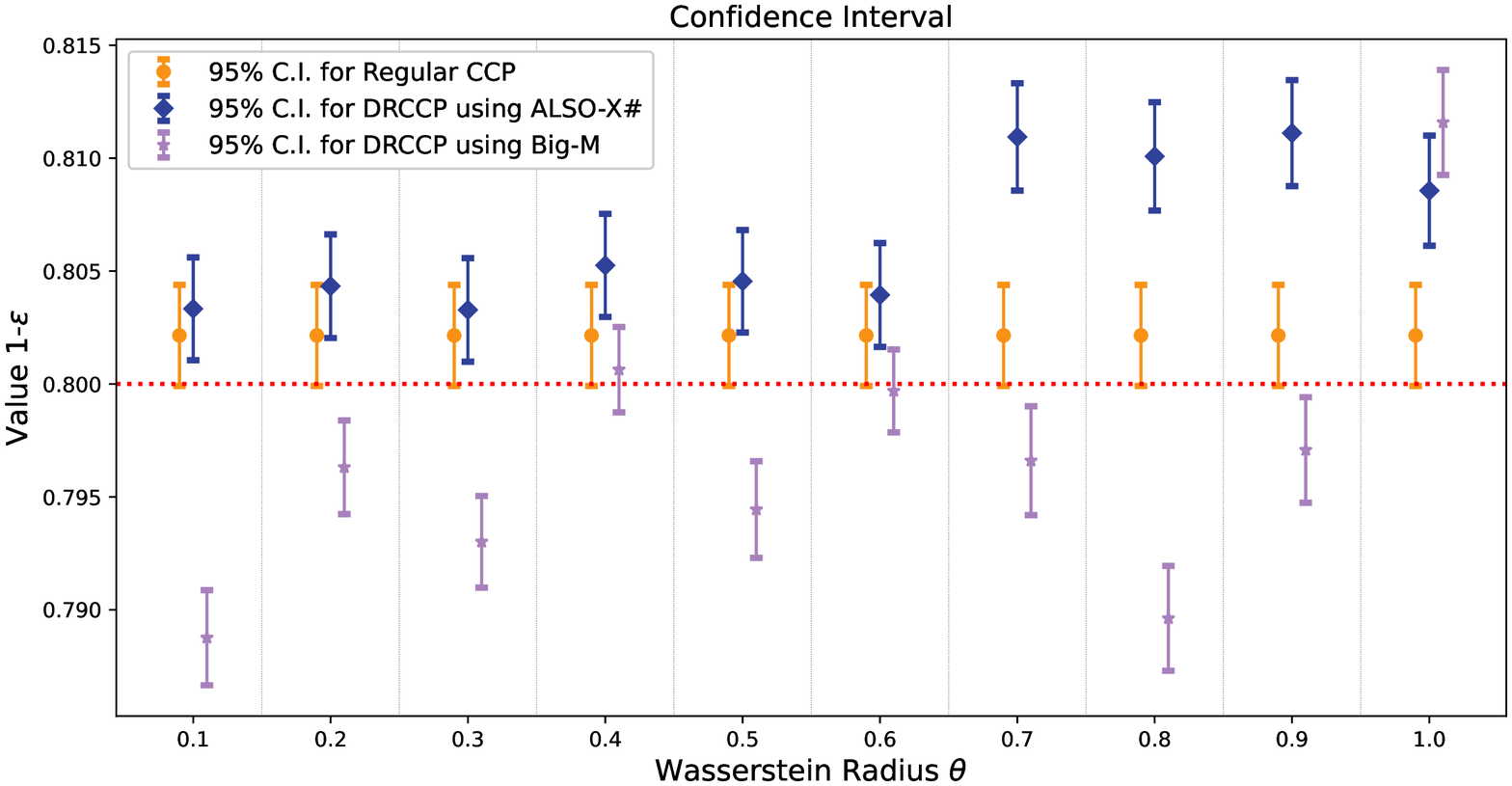}\label{Fig_wireless_allocation_radius}}
	~
	\subfloat[Evaluations of $\alsoxs$ and $\text{Big-M}$ Model with the Best-tuned Wasserstein Radii]{
		\includegraphics[width=.5\textwidth]{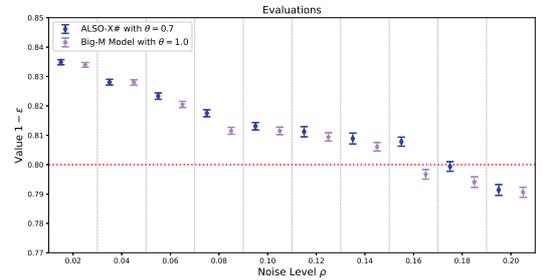}\label{Fig_wireless_allocation_evaluation}}
	\caption{Tuning Wasserstein Radius and Comparing the $\alsoxs$ and $\text{Big-M}$ Model Solutions in Resource Allocation \eqref{allocation_obj}
	}	\label{Fig_wireless_allocation}
\end{figure}




\section{Conclusion}
\label{sec_conclusion}
In this work, we proposed $\alsoxs$ for solving distributionally robust chance constrained programs (DRCCPs) by integrating $\alsox$ and $\CVaR$ approximation. We proved that $\alsoxs$ is better than $\CVaR$ approximation even when this deterministic set is nonconvex.  We provided sufficient conditions that $\alsoxs$ always outperforms  $\alsox$, i.e., when $\alsox$ admits a unique optimal solution. We showed that $\alsoxs$ can deliver an optimal solution to a DRCCP and also extended the discussions to the general Wasserstein ambiguity set. Our numerical studies demonstrated the effectiveness of $\alsoxs$. For a future study, it will be interesting to implement the $\alsoxs$  algorithm in real-time for wireless communication problems. Another interesting direction is to study the approximation guarantees of $\alsoxs$ when solving a DRCCP.

\section*{Acknowledgment}
This research has been supported in part by the National Science Foundation grants 2246414 and 2246417.

\section*{Declarations}

\noindent \textbf{Funding and/or Conflicts of interests/Competing interests:} This research has been supported by the National Science Foundation. Authors have no other competing interests to report.

\bibliography{alsocvar.bib}
\newpage 
\begin{appendices}

\section{Proofs}
\label{appendix_proof_sec}

\subsection{Proof of Theorem~\ref{theorem_drccp_unique_continuous}}
\label{proof_theorem_drccp_unique_continuous}
\theoremdrccpuniquecontinuous*
\begin{proof} 
We split the proof into two parts by discussing whether there exists an optimal solution $\bm{x}^*$ to the lower-level $\alsox$ \eqref{drccp_alsox} such that $\supp(\bm{x}^*)=n$.
\begin{enumerate}[label=(\roman*)]
    \item Suppose that there exists an optimal solution $\bm{x}^*$ to the lower-level $\alsox$ such that $\supp(\bm{x}^*)=n$.

First, for any $\bm{x}$ with $\supp(\bm{x})=n$, according to the continuity of function $f(\tau) = \max\{\tau, 0\}$ and theorem 1 in \cite{rockafellar1982interchange}, we can interchange the
subdifferential operator and expectation, the first-order derivative $\hat {F}'(\bm x)$ is
\begin{align*}
\hat {F}'(\bm x)=\frac{\partial{\hat {F}(\bm x)}}{\partial \bm x}= \int_{ \mu(\bm x)} \frac{\partial}{\partial \bm x} \left[ \theta\left\|\bm x \right\|_*+\bm x^\top{\rzeta} \right] \Pr(d\rzeta),
\end{align*}
where $\mu(\bm x)= \{\rzeta: \theta\|\bm x \|_*+\bm x^\top{\rzeta}\geq b_1 \}$ with its boundary $\partial \mu(\bm x)= \{\rzeta: \theta\|\bm x \|_*+\bm x^\top{\rzeta}=b_1 \}$.

According to equation (4) in \cite{uryasev2000introduction}, the Hessian of $\hat {F}(\bm x)$ is
\begin{align*}
   H_{\hat F}(\bm x) =  \frac{1}{\left\|\bm x\right\|_2} \int_{\partial \mu(\bm x)}    \left( \frac{\partial \theta\left\|\bm x\right\|_*}{\partial \bm x}+\rzeta\right)  \left( \frac{\partial \theta\left\|\bm x\right\|_*}{\partial \bm x}+\rzeta\right)^\top \Pr(d\rzeta)+\theta \int_{ \mu(\bm x)}\frac{\partial^2 \|\bm x\|_*}{\partial \bm x^2}\Pr(d\rzeta). 
\end{align*}
Recall that $\bm x^*$ is optimal to the lower-level $\alsox$ \eqref{drccp_alsox} and $\bm x^*\neq \bm 0$. To show that $\bm x^*$ is the unique optimal solution, it suffices to show that its corresponding Hessian $H_{\hat F}(\bm x)$ is positive definite (PD). 

 For any $\bm y\in \Re^n$ and $\bm y\neq \bm 0$, we have 
 \begin{align*}
    \bm y^\top  H_{\hat F}(\bm x) \bm y = & \frac{1}{\left\|\bm x\right\|_2}\int_{\partial \mu(\bm x)} \left( \bm y^\top\frac{\partial \theta\left\|\bm x\right\|_*}{\partial \bm x}+\bm y^\top\rzeta\right)^2\Pr(d\rzeta)+ \theta \int_{ \mu(\bm x)}\bm y^\top  \left(\frac{\partial^2 \|\bm x\|_*}{\partial \bm x^2} \right) \bm y\Pr(d\rzeta). 
\end{align*}
For $\|\bm x\|_*=\|\bm x\|_p$ with $p\in (1,\infty)$, for each $i\in[n]$, we have
\begin{align*}
\frac{\partial \|\bm x\|_p}{\partial x_i} = \sign(x_i) |x_i|^{p-1} \left[\sum_{j\in[n]} |x_j|^{p}\right]^{\frac{1}{p}-1}, \forall i\in[n],
\end{align*}
and
\begin{align*}
& \frac{\partial^2 \|\bm x\|_p}{\partial x_i^2} =   |x_i|^{p-2}\left[\sum_{j\in[n]\setminus \{i\}} |x_j|^{p}\right] (p-1)\left[\sum_{j\in[n]} |x_j|^{p}\right]^{\frac{1}{p}-2} , \forall i\in[n],\\
   & \frac{\partial^2 \|\bm x\|_p}{\partial x_i\partial x_k} =   \sign(x_i)|x_i|^{p-1}\sign(x_k)|x_k|^{p-1} (1-p)\left[\sum_{j\in[n]} |x_j|^{p}\right]^{\frac{1}{p}-2}, \forall i\in[n], k\in[n]\setminus\{i\}.
\end{align*}
That is,
\begin{align*}
    \frac{\partial^2 \|\bm x\|_*}{\partial \bm x^2}=\frac{\partial^2 \|\bm x\|_p}{\partial \bm x^2} =  (p-1)\left(\sum_{j\in[n]} |x_j|^{p}\right)^{\frac{1}{p}-2} &\left[\left(\sum_{j\in[n]} |x_j|^{p}\right)\Diag\begin{pmatrix}|x_1|^{p-2}\\\vdots \\ |x_n|^{p-2}\end{pmatrix}\right.\\
   &\left.-\begin{pmatrix}\sign(x_1)|x_1|^{p-1}\\\vdots \\ \sign(x_n)|x_n|^{p-1}\end{pmatrix}\begin{pmatrix}\sign(x_1)|x_1|^{p-1}\\\vdots \\ \sign(x_n)|x_n|^{p-1}\end{pmatrix}^\top
   \right].
\end{align*}
Here, $\partial^2 \|\bm x\|_p/\partial \bm x^2$ is a positive semidefinite (PSD) matrix of rank $n-1$. This implies that (i) the value $\bm y^\top  ({\partial^2 \|\bm x\|_*}/{\partial \bm x^2} ) \bm y=0$ if and only if $\bm y=\tau\bm x$ with $\tau\neq 0$ (recall that we have $\bm y\neq \bm 0$); and (ii) if $\bm z$ is an alternative optimal solution, then we must have $\bm z=\ell \bm x^*$. Hence, if ${\bm y}\not\propto \bm{x}^*$, then we must have $\bm y^\top  ({\partial^2 \|\bm x^*\|_*}/{\partial \bm x^2} ) \bm y>0$, which implies that the Hessian $H_{\hat F}(\bm x^*)$ is PD, which confirms the uniqueness of the optimal solution $\bm x^*$. It remains to show that if the solution $\bm x^*$ is not unique, all optimal solutions should be proportional to $\bm x^*$. We split the proof into two steps by the sign of $b_1$.

\noindent{\textbf{Step I.}} When $b_1\not=0$, we show that the Hessian $H_{\hat F}(\bm x^*)$ is PD.
That is, it suffices to show that if $\bm y=\ell \bm{x}^*$ with $\ell\not=0$,  we must have 
\begin{align*}
\int_{\partial \mu(\bm x^*)}  \left( \bm y^\top\frac{\partial \theta\left\|\bm x^*\right\|_*}{\partial \bm x}+\bm y^\top\rzeta\right)^2\Pr(d\rzeta)>0.
\end{align*}
Indeed, we notice that
\begin{align*}
    \bm y^\top\frac{\partial \theta\left\|\bm x^*\right\|_*}{\partial \bm x}+\bm y^\top\rzeta = \left(\ell \bm{x}^* \right)^\top  \frac{\partial \theta\left\|\bm x^*\right\|_*}{\partial \bm x} +\left(\ell \bm{x}^* \right)^\top\rzeta =\ell \left[\theta\left\|\bm x^*\right\|_* +(\bm{x}^* )^\top\rzeta \right]
    =  \ell b_1 \neq 0.
\end{align*}
 Hence, 
 \begin{align*}
    \int_{\partial \mu(\bm x^*)}  \left( \bm y^\top\frac{\partial \theta\left\|\bm x^*\right\|_*}{\partial \bm x}+\bm y^\top\rzeta\right)^2\Pr(d\rzeta)= \int_{\partial\mu(\bm{x}^*)}  \ell^2 b_1^2\Pr(d\rzeta) >0,
 \end{align*}
which implies that $ \bm y^\top  H_{\hat F}(\bm x^*) \bm y >0$.

\par

\noindent{\textbf{Step II.}} When $b_1=0$, we show that $\bm{x}^*$ must be the unique solution. Suppose that there exists another optimal solution $\bm z$. Then we must have $\bm z=\ell \bm x^*$ such that $\ell\neq 1$ according to the statement at the beginning of the proof. Three subcases remain to be discussed:\par

\noindent\textit{Subcase (i)} When $0<\ell<1$, according to the homogeneity of the objective function in the lower-level $\alsox$ \eqref{drccp_alsox}, i.e.,
\[\E_{\Pr_{\trzeta}}\left[\theta\left\|\bm z \right\|_*+\bm z^\top{\trzeta}\right]_+ =  \E_{\Pr_{\trzeta}}\left[\theta\left\|\ell \bm x^* \right\|_*+\left(\ell \bm x^*\right)^\top{\trzeta}\right]_+ = \ell v^A(t). \]
Hence, $\bm z=\ell \bm x^*$ yields a strict better objective value than that of $\bm x^*$ as $v^A(t)>0$, a contradiction; \par

\noindent\textit{Subcase (ii)} Similarly, when $\ell>1$, the current optimal solution $\bm x^*$ is strict better  than $\bm z$, a contradiction;  \par

\noindent\textit{Subcase (iii)} When $\ell<0$, since the objective function in  the lower-level $\alsox$ \eqref{drccp_alsox} and the feasible region is convex, any convex combination of $\bm x^*$ and $\bm z=\ell\bm x^*$ is also an optimal solution. That is, if we choose
\begin{align*}
\frac{-\ell}{|\ell|+1} \bm x^* + \frac{1}{|\ell|+1} (\ell\bm x^*) =\bm 0,
\end{align*}
then $\bm 0$ is one optimal solution with objective value $0$, which is strictly less than that of $\bm{x}^*$, a contradiction. Thus, we have that $\bm{x}^*$ with $|\supp(\bm x^*)|=n$ is the unique optimal solution since we have $v^A(t)>0$.

Hence, the objective function of the lower-level $\alsox$\eqref{drccp_alsox} admits a unique solution.

\item Suppose that there does not exist an optimal solution $\bm{x}^*$ to the lower-level $\alsox$ \eqref{drccp_alsox} such that $|\supp(\bm{x}^*)|=n$. Let $\bm{x}^*$ be an optimal solution that has the largest support. If $\bm z$ is another optimal solution to the lower-level $\alsox$ \eqref{drccp_alsox} and $\supp(\bm{x}^*)=\supp(\bm z)$, then following the proof in Part (i), we must have $\bm{x}^*=\bm z$. Thus, $\supp(\bm{x}^*)\neq \supp(\bm z)$. Now let us define 
\begin{align*}
    \gamma=\frac{1}{2}\frac{\min\{|x_i^*|:i\in \supp(\bm{x}^*) \}}{\min\{|x_i^*|:i\in \supp(\bm{x}^*) \}+\max\{|z_i|:i\in \supp(\bm{z}) \}}.
\end{align*}
Then $\bar{\bm x}:= (1-\gamma)\bm{x}^*+\gamma \bm z$ is another optimal solution, according to the convexity of the feasible set of the lower-level $\alsox$. However, $|\supp(\bar{\bm{x}})|\geq |\supp(\bm{x}^*)|+1$, a contradiction that $\bm{x}^*$ has the largest support. This completes the proof.    \QEDA
\end{enumerate}
\end{proof}

\subsection{Proof of Theorem~\ref{theorem_drccp_unique_continuous} with $q\in \{1,\infty\}$}
\label{proof_corollary_drccp_unique_continuous}
\begin{restatable}{lemma}{corollarydrccpuniquecontinuous}\label{corollary_drccp_unique_continuous} 
	Suppose that in a single DRCCP \eqref{eq_drccp_infty}, the deterministic set $\X$ is convex, the reference distribution ${\Pr_{\trzeta}}$ is continuous with support $\Re^n$, $\|\cdot\|_*=\|\cdot\|_p$ with $p\in \{1,\infty\}$ and affine mappings $\bm{a}_1(\bm{x})=\bm x$ and $b_1(\bm x)=b_1$. Then the lower-level $\alsox$ \eqref{drccp_alsox} admits a unique optimal solution when $t\geq \min_{\bm x\in \X}\bm c^\top \bm x$ and $v^A(t)>0$.
\end{restatable}

\begin{proof}
We split the proof into two cases based on the value of $p$.

\noindent{\textbf{Case I. $p=1$.}}	
We split the proof into two parts by discussing whether there exists an optimal solution $\bm{x}^*$ to the lower-level $\alsox$ \eqref{drccp_alsox} such that $\supp(\bm{x}^*)=n$.

\begin{enumerate}[label=(\roman*)]
    \item Suppose that there exists an optimal solution $\bm{x}^*$ to the lower-level $\alsox$ such that $\supp(\bm{x}^*)=n$. Recall that $\hat {F}(\bm x)$ denotes the objective function in the lower-level $\alsox$ \eqref{drccp_alsox} with the Hessian 
    \begin{align*}
          H_{\hat F}(\bm x) =  \frac{1}{\left\|\bm x\right\|_2} \int_{\partial \mu(\bm x)}    \left( \frac{\partial \theta\left\|\bm x\right\|_1}{\partial \bm x}+\rzeta\right)  \left( \frac{\partial \theta\left\|\bm x\right\|_1}{\partial \bm x}+\rzeta\right)^\top \Pr(d\rzeta)+\theta \int_{ \mu(\bm x)}\frac{\partial^2 \|\bm x\|_1}{\partial \bm x^2}\Pr(d\rzeta). 
    \end{align*}   
For function $\|\bm x\|_1$ with $\supp(\bm{x})=n$, we have
	\begin{align*}
		 & \frac{\partial \|\bm x\|_1}{\partial x_i} = \sign(x_i), \forall i\in[n];  \frac{\partial^2 \|\bm x\|_1}{\partial x_i\partial x_j} =  0, \forall i\in[n], j\in[n].
		 \end{align*}
Here, with the assumption that $\supp(\bm{x}^*)=n$, we have $\partial^2 \|\bm x^*\|_1/\partial \bm x^2=0$, which implies that $\bm y^\top  ({\partial^2 \|\bm x\|_1}/{\partial \bm x^2} ) \bm y=0$. Following the similar proof of Theorem~\ref{theorem_drccp_unique_continuous}, we have 
	 \begin{align*}
	 	\int_{\partial \mu(\bm x^*)}  \left( \bm y^\top\frac{\partial \theta\left\|\bm x^*\right\|_1}{\partial \bm x}+\bm y^\top\rzeta\right)^2\Pr(d\rzeta)>0.
	 \end{align*}

\item The case when $|\supp(\bm{x}^*)|<n$ is identical to part (ii) in the proof of Theorem \ref{theorem_drccp_unique_continuous}  and is thus omitted. 
\end{enumerate} 


  \noindent{\textbf{Case II. $p=\infty$. }}  
  We split the proof into two parts by discussing whether there exists an optimal solution $\bm{x}^*$ to the lower-level $\alsox$ \eqref{drccp_alsox} such that $\supp(\bm{x}^*)=n$.

\begin{enumerate}[label=(\roman*)]
    \item  Suppose that there exists an optimal solution $\bm{x}^*$ to the lower-level $\alsox$ such that $\supp(\bm{x}^*)=n$. We split the following proof into two steps.

\noindent{\textbf{Step I.}} We denote $\S$ to be an index set corresponding to the largest absolute-value component, that is, for any subset $\S\subseteq [n]$, we assume $|x^*_i|=\|\bm x^*\|_\infty$ for each $i\in\S$. For the fixed subset $\S$ and function $\|\bm x\|_\infty$ with $\supp(\bm{x})=n$, we have
   \begin{align*}
	& \frac{\partial \|\bm x\|_\infty}{\partial x_i} = \left\{\begin{aligned}
  	 \sign(x_i), \quad  & i\in \S, \\
 		0, \quad &  i\notin \S
 	\end{aligned}
 	\right., \forall i\in[n],
\end{align*}
and
\begin{align*}
    \frac{\partial^2 \|\bm x\|_\infty}{\partial x_i\partial x_j} = 0, \forall i\in[n], k\in[n].
\end{align*}
Here, with the presumptions, we have $\partial^2 \|\bm x^*\|_\infty/\partial \bm x^2=0$, which implies that $\bm y^\top  ({\partial^2 \|\bm x\|_\infty}/{\partial \bm x^2} ) \bm y=0$. Following the similar proof of Theorem~\ref{theorem_drccp_unique_continuous}, we have 
	 \begin{align*}
	 	\int_{\partial \mu(\bm x^*)}  \left( \bm y^\top\frac{\partial \theta\left\|\bm x^*\right\|_\infty}{\partial \bm x}+\bm y^\top\rzeta\right)^2\Pr(d\rzeta)>0.
	 \end{align*}
  Thus, given a subset $\S$, the optimal solution is unique.

\noindent{\textbf{Step II.}} It remains to prove that for any subset $\S$, the optimal solution is unique, where $\S$ is an index set corresponding to the largest absolute-value components of $\bm x^*$. Suppose that $\bm x^*$ has the smallest size of the largest absolute-value components. Assume $\bm z \neq \bm x^*$ is another optimal solution to the lower-level $\alsox$ \eqref{drccp_alsox} with $\S_1$ being its corresponding index set with the largest absolute-value components. According to our assumption, we must have $|\S_1|\geq |\S|$.  Now let us define 
 \begin{align*}
 	\gamma=\frac{1}{2}\frac{\min\{|x_i^*|:i\in \supp(\bm{x}^*) \}}{\min\{|x_i^*|:i\in \supp(\bm{x}^*) \}+\max\{|z^*_i|:i\in \supp(\bm{z}^*) \}}.
 \end{align*}
Then the convexity of the feasible set of the lower-level $\alsox$ implies that $\bar{\bm x}:= (1-\gamma)\bm{x}^*+\gamma \bm z^*$ is another optimal solution.  However,  the new solution $\bar{\bm x}$ either  $\bm{x}^*$ is the unique solution of the lower-level $\alsox$ or contradicts  that $\bm x^*$ has the smallest size of the largest absolute-value components.




 \item The case when $|\supp(\bm{x}^*)|<n$ is identical to part (ii) in the proof of Theorem \ref{theorem_drccp_unique_continuous}  and is thus omitted. \QEDA
 \end{enumerate}   
\end{proof}

\section{Equivalent Reformulations for DRCCPs under Type $q-$Wasserstein Ambiguity Set}\label{sec_extension_eq_refor}

Under type $q-$Wasserstein ambiguity set with $q\in[1,\infty)$, the DRCCP \eqref{eq_drccp} can be written as
\begin{align}
	v_q^*=\min_{\bm x\in\X}\left\{	\bm c^\top \bm x\colon  \inf_{\Pr\in \P_q} \Pr\left\{\trxi\colon  \bm a_i(\bm x)^\top \trxi \leq  b_i(\bm x), \forall i\in[I] \right\}\geq 1-\varepsilon \right\}. \label{eq_drccp_q_general}
\end{align}
For the notational convenience, we denote the decision space induced by the worst-case chance constraint in DRCCP \eqref{eq_drccp_q_general} as the following Distributionally Robust Chance Constrained (DRCC) set
\begin{align}
	Z_q\colon= \left\{\bm{x}\in \Re^n\colon  \inf_{\Pr\in \P_q} \Pr\left\{\trxi\colon  \bm a_i(\bm x)^\top \trxi \leq  b_i(\bm x), \forall i\in[I] \right\}\geq 1-\varepsilon \right\}.	\label{drccp_q_general}
\end{align}

\subsection{Equivalent Reformulations of DRCCPs}
We can generalize the existing work in \cite{xie2021distributionally} on the equivalent formulation of DRCC set $Z_q$ \eqref{drccp_q_general} for any $q\in [1,\infty)$ and any reference distribution $\Pr_{\trzeta}$. 

\begin{restatable}{proposition}{thmdrccset}\label{thm_drcc_set} 
	(A generalization of corollary 1 in \cite{xie2021distributionally}) Under type $q-$Wasserstein ambiguity set, DRCC set $Z_q$ \eqref{drccp_q_general} is equivalent to
	\begin{equation}
		\label{q_cvar_interpretation_cvar}
		Z_q=\left\{ \bm{x}\in \Re^n\colon 
		\begin{aligned}
			&\theta^q\varepsilon^{-1}+\CVaR_{1-\varepsilon}\left[-f(\bm{x},\tilde\rzeta)^q\right]\leq 0,\\ &\theta^q\varepsilon^{-1}+\VaR_{1-\varepsilon}\left[-f(\bm{x},\tilde\rzeta)^q\right]\leq 0
		\end{aligned}
		\right\}.
	\end{equation}
	where
	\begin{equation*}
		f(\bm{x},\bm{\zeta})= \min \left\{ \min_{i\in[I]\setminus\mathcal{I}(\bm{x})}\frac{(b_i(\bm{x})-\bm a_i(\bm x)^\top \rzeta)_+}{\left\|\bm{a}_i(\bm{x})\right\|_*},\min_{i\in\mathcal{I}(\bm{x})}\chi_{\{\bm{x}\colon{b}_i(\bm{x})<0\}}(\bm{x})\right\},
	\end{equation*}
	and $\mathcal{I}(\bm{x}) = \emptyset$ if $\bm{a}_i(\bm{x})\not=\bm 0$ and $\mathcal{I}(\bm{x})=[I]$. 
\end{restatable}

 \begin{proof}
	We split our proof into two cases by discussing whether $\theta=0$ or not. 
	\begin{enumerate}[label={Case} \arabic*]
		\item	When $\theta=0$, DRCC set $Z_q$ reduces to the regular CCP under reference distribution $\Pr_{\tilde\rzeta}$, i.e.,
		\begin{align*}
			Z_q=\left\{ \bm{x}\in \Re^n\colon 
			\Pr\left\{\trzeta\colon\bm a_i(\bm x)^\top \trzeta-b_i(\bm{x})\leq 0,\forall i\in[I] \right\}\geq 1-\varepsilon \right\}.
		\end{align*}
		On the other hand, in \eqref{q_cvar_interpretation_cvar},
		the first constraint is redundant when $\theta=0$. Hence, the statement follows.
		\item	When $\theta>0$, the fact that the decision space induced by the first constraint in \eqref{q_cvar_interpretation_cvar} is equivalent to DRCC set $Z_q$ follows directly from the proof of corollary 1 in \cite{xie2021distributionally}. Thus, in this case, the second constraint in \eqref{q_cvar_interpretation_cvar} is redundant since for any random variable $\tilde{\bm{X}}$, we have $\CVaR_{1-\varepsilon}[\tilde{\bm{X}}]\geq \VaR_{1-\varepsilon}[\tilde{\bm{X}}]$.
	\end{enumerate}
	\QEDA
\end{proof}

The reformulation in Proposition~\ref{thm_drcc_set} can be simplified if $\|\bm{a}_i(\bm{x})\|_*=\|\bm{a}_1(\bm{x})\|_*$ for all $i\in[I]$, in which the condition can be viewed  as a generalization of $\bm a_1(\bm x)=\bm a_i(\bm x)$ for all $i\in[I]$, which has been discussed in the DRCCP literature (see, e.g., \cite{xie2021distributionally,chen2022data}). Notice that this condition always holds for a single DRCCP, where $I=1$.
\begin{restatable}{proposition}{qcvarinterpretationcvarsinglecoro}\label{q_cvar_interpretation_cvar_single_coro} 
	Suppose that $\|\bm{a}_i(\bm{x})\|_*=\|\bm{a}_1(\bm{x})\|_*$ for all $i\in[I]$. Then DRCC set $Z_q$ \eqref{drccp_q_general} can be simplified to 
\begin{equation}
\label{q_cvar_interpretation_cvar_single}
		Z_q=\left\{ \bm{x}\in \Re^n\colon
		\begin{aligned}
&\theta^q\varepsilon^{-1}\left\|\bm{a}_1(\bm{x})\right\|_*^q+\CVaR_{1-\varepsilon}\left[-\min_{i\in [I]}\left({b}_i(\bm{x})-\bm a_i(\bm x)^\top \trzeta\right)^q_+\right]\leq 0, \\
			&\Pr\left\{\trzeta\colon \bm a_i(\bm x)^\top \trzeta-b_i(\bm{x})\leq 0, \forall i\in[I] \right\}\geq 1-\varepsilon
		\end{aligned} \right\}.
\end{equation}
\end{restatable}

	\begin{proof}
		We split our proof into two cases by discussing whether $\|\bm{a}_1(\bm{x})\|_*=0$ or not. 
		\begin{enumerate}[label={Case} \arabic*]
			\item	When $\|\bm{a}_1(\bm{x})\|_*=0$ (i.e., $\bm{a}_1(\bm{x})=\bm{0}$), according to \eqref{q_cvar_interpretation_cvar}, set $Z_q\cap\{\bm{x}\in \Re^n:\|\bm{a}_1(\bm{x})\|_*=0 \}$ is equivalent to the set
			\begin{align*}
				\left\{ \bm{x}\in \Re^n\colon \left\|\bm{a}_1(\bm{x})\right\|_*=0
				,b_i(\bm{x})\geq 0 ,\forall i\in[I]\right\},
			\end{align*}
			which is equivalent to the right-hand side of \eqref{q_cvar_interpretation_cvar_single} intersecting with set $\{\bm{x}\in \Re^n:\|\bm{a}_1(\bm{x})\|_*=0 \}$ since its first constraint is redundant.

			\item 
			When $\|\bm{a}_1(\bm{x})\|_*>0$, according to Proposition~\ref{thm_drcc_set}, the function $f(\bm{x},\rzeta)$ becomes
			\begin{equation*}
				f(\bm{x},\rzeta)= \min_{i\in[I]}\frac{\left(b_i(\bm{x})-\bm{a}_i(\bm{x})^\top \rzeta\right)_+}{\left\|\bm{a}_1(\bm{x})\right\|_*}.
			\end{equation*}
			According to the positive homogeneity of the coherent risk measure $\CVaR$, set $Z_q\cap\{\bm{x}\in \Re^n:\|\bm{a}_1(\bm{x})\|_*>0 \}$ is equivalent to the set
			\begin{equation*}
				\left\{ \bm{x}\in \Re^n\colon 
				\begin{aligned}
					&\left\|\bm{a}_1(\bm{x})\right\|_*>0, \\
					&\theta^q\varepsilon^{-1}\left\|\bm{a}_1(\bm{x})\right\|_*^q+\CVaR_{1-\varepsilon}\left[-\min_{i\in[I]}\left(b_i(\bm{x})-\bm{a}_i(\bm{x})^\top\trzeta\right)_+^q\right]\leq 0,\\ &\Pr\left\{\trzeta\colon\bm{a}_i(\bm{x})^\top\trzeta-b_i(\bm{x})\leq 0,\forall i\in[I] \right\}\geq 1-\varepsilon 
				\end{aligned}
				\right\}.
			\end{equation*}
			This completes the proof.\QEDA
		\end{enumerate}
		
	\end{proof}
As a direct corollary of Proposition~\ref{q_cvar_interpretation_cvar_single_coro}, we remark that for the single DRCCP (i.e., $I=1$) with the elliptical reference
distribution $\Pr_{\trzeta}$ (see the discussions in Section~\ref{sec_exactness_elliptical}), DRCC set $Z_q$ admits a simple representation.

\begin{restatable}{corollary}{refeldrccp}\label{ref_el_drccp} 
	\begin{subequations}
	For the single DRCCP \eqref{eq_drccp_q_general}, when the affine mappings 
 are $\bm{a}_1(\bm{x})=\bm x$, $b_1(\bm x)=b_1$, the random parameters $\trzeta$ follow a joint elliptical distribution with $\trzeta\thicksim\Pr_{\mathrm{E}}(\bm{\mu},\rsigma,\hat g)$, and the norm defining the Wasserstein distance is the generalized Mahalanobis norm associated with the matrix $\bm{\mathrm{{\Sigma}}}$, DRCC set $Z_q$ \eqref{drccp_q_general} becomes
		\begin{align}
			Z_q=\left\{ \bm{x}\in \Re^n\colon b_1(\bm x)-\bm{\mu}^\top\bm{a}_1(\bm{x})\geq \eta_q^* \sqrt{\bm{a}_1(\bm{x})^\top\bm{\mathrm{{\Sigma}}}\bm{a}_1(\bm{x})} \right\},\label{drccp_el_formulation}
		\end{align}
		and $	\eta_q ^*$ is the unique minimizer of
		\begin{align}
			\eta_q^*=\min_\eta \left\{ \eta\colon \int_{\mathrm{\Phi}^{-1}(1-\varepsilon)}^{\eta}(\eta-t)^q\bar{k}\hat g(t^2/2) dt \geq\theta^q, \eta\geq \mathrm{\Phi}^{-1}(1-\varepsilon) \right\}.\label{drccp_el_formulation_eta}
		\end{align}
	\end{subequations}
\end{restatable}

	\begin{proof}
		We note that if $\bm{a}_1(\bm{x})=\bm{0}$, according to Proposition~\ref{q_cvar_interpretation_cvar_single_coro},  set $Z_q\cap\{\bm{x}\in \Re^n\colon\bm{a}_1(\bm{x})=\bm{0} \}$ becomes
		\begin{align*}
			\left\{ \bm{x}\in \Re^n\colon \bm{a}_1(\bm{x})=\bm{0}, b_1(\bm{x})\geq 0 \right\},
		\end{align*}
		which is equivalent to the right-hand side of \eqref{drccp_el_formulation} intersecting with set $\{\bm{x}\in \Re^n\colon\bm{a}_1(\bm{x})=\bm{0} \}$. Thus, without loss of generality, we assume that $\bm{a}_1(\bm{x})\neq\bm{0}$.
		
		Note that the linear function $\bm{a}_1(\bm{x})^\top \rzeta$ is still elliptically distributed (see, e.g., \cite{gupta2013elliptically}). For ease of notation, we denote the distribution of the linear function $\bm{a}_1(\bm{x})^\top\rzeta$ as $\Pr_{\mathrm{E}}(\mu_{\bm x},\sigma_{\bm x},\hat g)$ and denote its probability density function as 
		\begin{equation*}
			h(y) =\frac{\bar{k}}{\sigma_{\bm x}} \hat g\left(\frac{(y-\mu_{\bm x})^2}{2\sigma_{\bm x}^2}\right),
		\end{equation*}
		where $\mu_{\bm x}=\bm{\mu}^\top\bm{a}_1(\bm{x})$ and $\sigma_{\bm x}= \sqrt{\bm{a}_1(\bm{x})^\top\rsigma\bm{a}_1(\bm{x})}$.
		
		According to Proposition~\ref{q_cvar_interpretation_cvar_single_coro}, DRCC set $Z_q$ is
		\begin{equation*}
			Z_q=\left\{ \bm{x}\in \Re^n\colon
			\begin{aligned}
		&\theta^q\varepsilon^{-1}\left\|\bm{a}_1(\bm{x})\right\|_*^q+\CVaR_{1-\varepsilon}\left[-\left({b}_1(\bm{x})-\bm{a}_1(\bm{x})^\top\trzeta\right)^q_+\right]\leq 0, \\
				&\Pr\left\{\trzeta\colon\bm{a}_1(\bm{x})^\top\trzeta-b_1(\bm{x})\leq 0 \right\}\geq 1-\varepsilon
			\end{aligned} \right\}.
		\end{equation*}
		Following the similar derivation in theorem 7 \cite{chen2021sharing} and according to the definition of $\CVaR$ (see, e.g., \cite{rockafellar2000optimization}), set $Z_q$ is equivalent to
		\begin{equation*}
			Z_q=\left\{ \bm{x}\in \Re^n\colon
			\begin{aligned}
				\frac{1}{\varepsilon} \int_{\VaR_{1-\varepsilon}[\bm{a}_1(\bm{x})^\top \trzeta]}^{b_1(\bm x)}(b_1(\bm x)-y)^qh(y)dy \geq \theta^q\varepsilon^{-1}\left\|\bm{a}_1(\bm{x})\right\|^q_*,b_1(\bm x) \geq \VaR_{1-\varepsilon}\left[\bm{a}_1(\bm{x})^\top \trzeta\right]
			\end{aligned} \right\}.
		\end{equation*}
		According to theorem 3 in \cite{prekopa1974programming}, we have
		\begin{align*}
			\VaR_{1-\varepsilon}\left[\bm{a}_1(\bm{x})^\top \trzeta\right]=\mu_{\bm x}+ \mathrm{\Phi}^{-1}(1-\varepsilon) \sigma_{\bm x}.
		\end{align*}
		Now let $t=(y-\mu_{\bm x})/\sigma_{\bm x}, y=t\sigma_{\bm x}+\mu_{\bm x}$ and $\eta=(b_1(\bm x)-\mu_{\bm x})/\sigma_{\bm x}$, and then DRCC set $Z_q$ is further equal to
		\begin{equation*}
			Z_q=\left\{ \bm{x}\in \Re^n\colon
			\begin{aligned}b_1(\bm x)-\mu_{\bm x}\geq \eta \sigma_{\bm x},
				\int_{\mathrm{\Phi}^{-1}(1-\varepsilon)}^{\eta}(\eta\sigma_{\bm x}-t\sigma_{\bm x})^qh(t\sigma_{\bm x}+\mu_{\bm x})dt \geq\theta^q \sigma_{\bm x}^{q-1}, \eta\geq \mathrm{\Phi}^{-1}(1-\varepsilon)
			\end{aligned} \right\}.
		\end{equation*}
		%
		where we use the fact that $\|\bm{a}_1(\bm{x})\|_*=\sigma_{\bm x}$. We see that set $Z_q$ expands as $\eta$ decreases, and thus, we can replace it by the minimal $\eta_q^*$ defined \eqref{drccp_el_formulation_eta}.
		Substituting the generating function $\hat g(\cdot)$, we arrive at \eqref{drccp_el_formulation}. Finally, we note that $\int_{\mathrm{\Phi}^{-1}(1-\varepsilon)}^{\eta}(\eta-t)^q\bar{k}\hat g(t^2/2) dt$ is monotone increasing in $\eta$, which demonstrates the uniqueness of $\eta_q^*$.\QEDA
	\end{proof}

 \subsection{Equivalent Reformulations of $\alsox$ }
Similar to $\alsox$ \eqref{drccp_alsox_formulation}, we extend the $\alsox$ under type $q-$Wasserstein ambiguity set. For any $q\in [1,\infty)$, $\alsox$ is formally defined as
\begin{subequations}
	\label{drccp_alsox_q_general}
	\begin{align}
		v_q^A =\min _{ {t}}\quad & t, \label{drccp_alsox_q_general_a}\\
		\text{s.t.}\quad&\bm x^*\in\argmin _{\bm {x}\in \mathcal{X}}\, \sup _{\Pr\in\P_q} \left\{ \E_\Pr\left[\max_{i\in [I]}\left(\bm a_i(\bm x)^\top \trxi-b_i(\bm x)\right)_+\right]\colon \bm{c}^\top \bm{x} \leq t \right\},\label{drccp_alsox_q_general_b}\\
		&\bm x^*\in Z_q.\label{drccp_alsox_q_general_c}
	\end{align}
\end{subequations} 
We then derive an equivalent reformulation of the hinge-loss approximation \eqref{drccp_alsox_q_general_b} under type $q-$Wasserstein ambiguity set. 
\begin{restatable}{proposition}{qalsoxtheorem}\label{q_alsox_theorem} 
	Under type $q-$Wasserstein ambiguity set with $q\in[1,\infty)$, hinge-loss approximation \eqref{drccp_alsox_q_general_b} is equivalent to
		\begin{align}
			v_q^A(t)=\min _{ \begin{subarray}{c}\bm {x}\in \X,\\ \lambda\geq 0 \end{subarray}} \left\{ \lambda\theta^q+\E_{\Pr_{\trzeta}}\left[  \max_{ i\in[I]}\left(\bm a_i(\bm x)^\top \trzeta-b_i(\bm{x}) +P_{q,i}(\bm x,\lambda)\right)_+\right]\colon 
			\bm{c}^\top \bm{x} \leq t\right\},	\label{wasser_q}
		\end{align}
		where for each $i\in [I]$, $P_{q,i}(\bm x,\lambda)=(\|\bm{a}_i(\bm x)\|_*)^{\frac{q}{q-1}}\lambda^{-\frac{1}{q-1}}q^{-\frac{q}{q-1}} (q-1)$ with its limit being
		\begin{align*}
			\lim_{q\rightarrow 1_+} P_{q,i}(\bm x,\lambda)= \lim_{q\rightarrow 1_+}(\left\|\bm{a}_i(\bm{x})\right\|_*)^{\frac{q}{q-1}}\lambda^{-\frac{1}{q-1}}q^{-\frac{q}{q-1}} (q-1)
			= \chi_{\{\Re^n:\left\|\bm{a}_i(\bm{x})\right\|_* \leq \lambda \}} (\bm x). 
		\end{align*}
	\end{restatable}

	\begin{subequations}
		\begin{proof}
			According to theorem 1 in \cite{gao2022distributionally} or theorem 1 in \cite{blanchet2019quantifying}, the inner supremum $\sup _{\Pr\in\P_q} \E_\Pr[\max_{i\in [I]}(\bm a_i(\bm x)^\top \trxi-b_i(\bm x))_+]$ in \eqref{drccp_alsox_q_general_b} can be reformulated as
			\begin{equation*}
				\min_{\bm{x}\in \mathcal{X},\lambda\geq 0} \lambda\theta^q-\E_{\Pr_{\trzeta}}\left[\inf_{{\bm{\xi}}}\left\{\lambda\left\|{\bm{\xi}}-\tilde{\rzeta}\right\|^q-\max_{i\in [I]}(\bm a_i(\bm x)^\top \rxi-b_i(\bm x))_+ \right\} \right].
			\end{equation*}
			Next, we split the proof into two steps.\par
			\noindent\textbf{Step 1.} We first reformulate the term $ \inf_{\rxi}\{\lambda\|\rxi-\trzeta\|^q-\max_{i\in [I]}(\bm a_i(\bm x)^\top \rxi-b_i(\bm x))_+ \}$.
			Moving the minus sign into the maximum operators, we have 
			\begin{align*}
				\inf_{{\bm{\xi}}}\left\{ \min_{ i\in[I]}\min\left\{ \lambda\left\|\rxi-\trzeta\right\|^q-\bm a_i(\bm x)^\top \rxi+b_i(\bm{x}) ,\lambda\left\|{\bm{\xi}}-\tilde{\rzeta}\right\|^q\right\} \right\}.
			\end{align*}
			Then interchange the minimum and infimum, we obtain
			\begin{align*}
				\min_{ i\in[I]}\min\left\{ \inf_{\rxi}\left\{ \lambda\left\|\rxi-\trzeta\right\|^q-\bm a_i(\bm x)^\top \rxi+b_i(\bm{x})\right\} ,\inf_{\rxi} \lambda\left\|\rxi-\trzeta\right\|^q \right\}.
			\end{align*}
			Note that $\inf_{\rxi} \lambda\|\rxi-\trzeta\|^q=0$ and it remains to simplify $\inf_{\rxi}\{ \lambda\|\rxi-\trzeta\|^q-\bm a_i(\bm x)^\top \rxi+b_i(\bm{x}) \}$ for each $i\in [I]$. Letting $\hat\rzeta=\rxi-\trzeta$, we have
			\begin{equation*}
				\inf_{{\bm{\xi}}}\left\{ \lambda\left\|\rxi-\trzeta\right\|^q-\bm a_i(\bm x)^\top \rxi+b_i(\bm{x}) \right\}
				=\inf_{\hat\rzeta}\left\{ \lambda\left\|\hat\rzeta\right\|^q-\bm a_i(\bm x)^\top \hat\rzeta \right\}-\bm a_i(\bm x)^\top \trzeta+b_i(\bm{x}).
			\end{equation*}
			According to H\"older's inequality and the fact that infimum is attainable, we have 
			\begin{align*}
				\inf_{\hat\rzeta}\left\{ \lambda\left\|\hat\rzeta\right\|^q-\bm a_i(\bm x)^\top \hat\rzeta \right\}
				&	= \inf_{\hat\rzeta}\left\{ \lambda\left\|\hat\rzeta\right\|^q-\left\|\bm{a}_i(\bm{x})\right\|_*\left\|\hat\rzeta\right\| \right\}\nonumber
				\\
				&= \left(\left\|\bm{a}_i(\bm{x})\right\|_*\right)^{\frac{q}{q-1}}\lambda^{-\frac{1}{q-1}}q^{-\frac{q}{q-1}} (1-q).
		\end{align*}
		Note that when $q\rightarrow 1_+$,  the hinge-loss approximation reduces to \begin{align*}
			\lim\limits_{q\rightarrow 1_+}(\left\|\bm{a}_i(\bm{x})\right\|_*)^{\frac{q}{q-1}}\lambda^{-\frac{1}{q-1}}q^{-\frac{q}{q-1}} (q-1)
			= \chi_{\{\Re^n:\left\|\bm{a}_i(\bm{x})\right\|_* \leq \lambda \}} (\bm x).
		\end{align*}
		
		\noindent\textbf{Step 2.} According to Step 1, the hinge-loss approximation \eqref{drccp_alsox_q_general_b} becomes
		\begin{align*}
			v_q^A(t)=\min _{\bm {x}\in \X,\lambda\geq 0} &\left\{ \lambda\theta^q-\E_{\Pr_{\trzeta}}\left[\min_{i\in[I]}\min\left(-\bm a_i(\bm x)^\top\trzeta+b_i(\bm{x})-(\left\|\bm{a}_i(\bm{x})\right\|_*)^{\frac{q}{q-1}}\lambda^{-\frac{1}{q-1}}q^{-\frac{q}{q-1}} (q-1),0\right)\right]\colon \right.\\
			&\left.
			\bm{c}^\top \bm{x} \leq t\right\}.
		\end{align*}
		Moving the minus sign inside the expectation, we arrive at the conclusion.
		\QEDA
	\end{proof}
\end{subequations}
 
  \subsection{Equivalent Reformulations of $\CVaR$ Approximation }
 For DRCCP \eqref{eq_drccp_q_general}, its $\CVaR$ approximation is defined as 
\begin{equation*}
v_q^{\CVaR}=\min_{\bm x\in\X}\left\{\bm c^\top \bm x \colon \sup _{{\Pr}\in\P_q} \inf_{\beta\leq 0}\left[ \beta+ \frac{1}{\varepsilon}\E_{\Pr}\left[ \left( \max_{i\in[I]} \left(\bm a_i(\bm x)^\top \trxi-b_i(\bm{x}) \right)-\beta \right)_+ \right] \right] \leq 0 \right\}.
\end{equation*}
Since the ambiguity set $P_q$ is weakly compact according to Assumption \ref{A_1} and theorem 1 in \cite{yue2022linear}, we can interchange the infimum with the supremum and multiply both sides by $\varepsilon$. Then  for any $q\in[1,\infty)$, $\CVaR$ approximation can be formulated as
\begin{align*}
	v_q^{\CVaR}=\min_{\bm x\in\X}\left\{ \bm c^\top \bm x\colon \inf_{\beta\leq 0}\sup _{{\Pr}\in\P_q} \left[ \varepsilon\beta+ \E_{\Pr}\left[ \left( \max_{i\in[I]} \left(\bm a_i(\bm x)^\top \trxi-b_i(\bm{x}) \right)-\beta \right)_+ \right] \right] \leq 0 \right\}.
\end{align*}
Following similar derivations as those of Proposition~\ref{thm_drcc_set} and Proposition~\ref{q_alsox_theorem}, we obtain the equivalent reformulation of $\CVaR$ approximation for DRCCP \eqref{eq_drccp_q_general}.

\begin{restatable}{proposition}{qcvartheorem}\label{q_cvar_theorem} 
Under type $q-$Wasserstein ambiguity set, $\CVaR$ approximation of DRCCP \eqref{eq_drccp_q_general} is equivalent to
	\begin{align}
		v_q^{\CVaR}=\min_{\bm x\in\X,\lambda\geq 0, \beta\leq 0 }\left\{ \bm c^\top \bm x\colon \varepsilon\beta+\lambda\theta^q+\E_{\Pr_{\trzeta}}\left[  \max_{ i\in[I]}\left(\bm a_i(\bm x)^\top \trzeta-b_i(\bm{x}) +P_{q,i}(\bm x,\lambda)-\beta\right)_+\right]\leq 0\right\}, \label{worse_case_cvar_q_formulation}
	\end{align}
	where for each $i\in [I]$, $P_{q,i}(\bm x,\lambda)$ is defined in Proposition~\ref{q_alsox_theorem}.
\end{restatable}

\begin{proof}
	According to the similar derivations in Proposition~\ref{thm_drcc_set} and Proposition~\ref{q_alsox_theorem}, we have
	\begin{align*}
		v_q^{\CVaR}=\min_{\bm x\in\X}\left\{ \bm{c}^\top\bm x\colon \inf_{\beta\leq 0}\left\{\varepsilon\beta+\lambda\theta^q+\E_{\Pr_{\trzeta}}\left[  \max_{ i\in[I]}\left(\bm a_i(\bm x)^\top \trzeta-b_i(\bm{x}) +P_{q,i}(\bm x,\lambda)-\beta\right)_+\right]\right\}\leq 0,  \lambda\geq 0 \right\}.
	\end{align*}
	Note that the infimum is achievable since the left-hand function is continuous and convex in $\beta$ and when $\beta\rightarrow -\infty$, the left-hand function goes to positive infinity. This completes the proof.\QEDA
\end{proof}
Equivalently, we also recast the CVaR approximation \eqref{worse_case_cvar_q_formulation} as a bilevel program. That is,
\begin{subequations}
	\label{eq_cvar_q}
	\begin{align}
		&v_q^{\CVaR} =\min _{ {t}}\quad  t, \label{eq_cvar_q_a} \\
		&\textup{s.t.}\quad(\bm x^*,\lambda^*,\beta^*)\in\argmin _{\begin{subarray}{c}\bm {x}\in \X, \bm{c}^\top \bm{x} \leq t,\\ \lambda\geq 0, \beta\leq 0 \end{subarray}}\,\left\{ \varepsilon\beta+ \lambda\theta^q+\E_{\Pr_{\trzeta}}\left[  \max_{ i\in[I]}\left(\bm a_i(\bm x)^\top \trzeta-b_i(\bm{x}) +P_{q,i}(\bm x,\lambda)-\beta\right)_+\right]\right\},\label{eq_cvar_q_b}\\
		& \quad\quad \varepsilon\beta^*+ \lambda^*\theta^q+\E_{\Pr_{\trzeta}}\left[  \max_{ i\in[I]}\left(\bm a_i(\bm x^*)^\top \trzeta-b_i(\bm{x^*}) +P_{q,i}(\bm x^*,\lambda^*)-\beta\right)_+\right]\leq 0.\label{eq_cvar_q_c}
	\end{align}
\end{subequations}  

In the following result, we observe that under the same premise as 
Proposition~\ref{q_cvar_interpretation_cvar_single_coro}, the $\CVaR$ approximation \eqref{worse_case_cvar_q_formulation} can be simplified.

\begin{restatable}{proposition}{cvarqformulationcorollary}\label{cvar_q_formulation_corollary} 
	Suppose that $\|\bm{a}_i(\bm{x})\|_*=\|\bm{a}_1(\bm{x})\|_*$ for all $i\in[I]$. Then  $\CVaR$ approximation \eqref{worse_case_cvar_q_formulation} is
	\begin{equation}
	v_q^{\CVaR}=\min_{\bm x\in\X}\left\{ \bm{c}^\top\bm x\colon\theta\varepsilon^{-\frac{1}{q}}\left\|\bm{a}_1(\bm{x})\right\|_* + \CVaR_{1-\varepsilon}\left[ \max_{i\in[I]} \left\{\bm{a}_i(\bm{x})^\top\trzeta-b_i(\bm{x})\right\}\right]\leq 0 \right\}.	\label{worse_case_cvar_q_formulation_a}
	\end{equation}
\end{restatable}

	\begin{proof}
	Since $\|\bm{a}_i(\bm{x})\|_*=\|\bm{a}_1(\bm{x})\|_*$ for all $i\in[I]$, we must have $P_{q,i}(\bm{x},\lambda)=P_{q,1}(\bm{x},\lambda)$, for all $ i\in[I]$. Then $\CVaR$ approximation  is equivalent to
		\begin{equation*}
		v_q^{\CVaR}=\min_{\bm x\in\X, \lambda,\beta}\left\{ \bm{c}^\top\bm x\colon
			\begin{aligned}
&\lambda\theta^q+\E_{\Pr_{\tilde\rzeta}}\left[ \max\left\{\max_{i\in[I]}\left\{ \bm{a}_i(\bm{x})^\top\trzeta-b_i(\bm{x})+P_{q,1}(\bm{x},\lambda) \right\},\beta\right\}\right]-(1-\varepsilon)\beta \leq 0,\\
				&\lambda \geq 0
			\end{aligned}\right\}.
		\end{equation*}
		Subtracting the $\beta$ in the inner maximum operator and redefining $\beta:=\beta-P_{q,1}(\bm{x},\lambda)$, we have
		\begin{align*}
		v_q^{\CVaR}=\min_{\bm x\in\X}\left\{ \bm{c}^\top\bm x\colon
			\begin{aligned}
				& \lambda\theta^q+ \varepsilon P_{q,1}(\bm{x},\lambda) + \varepsilon\beta +\E_{\Pr_{\tilde\rzeta}}\left[ \max_{i\in[I]}\left\{ \bm{a}_i(\bm{x})^\top\trzeta-b_i(\bm{x})-\beta \right\}_+\right] \leq 0,\\
				&\lambda \geq 0
			\end{aligned}\right\}.
		\end{align*}
		Replacing the existence of $\beta$ and $\lambda\geq 0$ by the minimum operator over $\beta$ and $\lambda\geq 0$ in the left-hand side of the first constraint, we arrive at
		\begin{align*}
			v_q^{\CVaR}=\min_{\bm x\in\X}\left\{ \bm{c}^\top\bm x\colon
			\min_{\lambda\geq 0}\left\{\lambda\theta^q+\varepsilon P_{q,1}(\bm{x},\lambda)\right\}+\varepsilon\CVaR_{1-\varepsilon}\left[ \max_{i\in[I]} \left\{\bm{a}_i(\bm{x})^\top\trzeta-b_i(\bm{x})\right\}\right]\leq 0 \right\}.
		\end{align*}
		Note that for any given $\bm{x}$, the function $\lambda\theta^q+\varepsilon P_{q,1}(\bm{x},\lambda)$ is convex in $\lambda$ over the domain $\lambda\in [0,\infty)$.
		Let us take its first-order derivative with respect to $\lambda$, and set it to be $0$, which has a nonnegative root 
		\begin{align*}
			\lambda^* = \varepsilon^{\frac{q-1}{q}}\left\|\bm{a}_1(\bm{x})\right\|_*q^{-1}\theta^{\frac{1}{q-1}}\geq 0.    
		\end{align*}	
		Thus, $\lambda^*$ solves $ \min_{\lambda\geq 0}\left\{\lambda\theta^q+\varepsilon P_{q,1}(\bm{x},\lambda)\right\}$.
		Substituting $\lambda^*$ into $\CVaR$ approximation, we arrive at the equivalent representation \eqref{worse_case_cvar_q_formulation}.
		\QEDA
	\end{proof}
 

 \subsection{Equivalent Reformulations of $\alsoxs$ and $\alsoxus$}
 
According to the reformulations above, under type $q-$Wasserstein ambiguity set with $q\in[1,\infty)$, $\alsoxs$ admits the form of
\begin{subequations}
	\label{eq_alsoxs_q}
	\begin{align}
		&v_q^{\as} =\min _{ {t}}\quad  t, \label{eq_alsoxs_q_a} \\
		&\textup{s.t.}\quad(\bm x^*,\lambda^*,\beta^*)\in\argmin _{\begin{subarray}{c}\bm {x}\in \X, \bm{c}^\top \bm{x} \leq t,\\ \lambda\geq 0, \beta\leq 0 \end{subarray}}\,\left\{ \varepsilon\beta+ \lambda\theta^q+\E_{\Pr_{\trzeta}}\left[  \max_{ i\in[I]}\left(\bm a_i(\bm x)^\top \trzeta-b_i(\bm{x}) +P_{q,i}(\bm x,\lambda)-\beta\right)_+\right]\right\},\label{eq_alsoxs_q_b}\\
		& \quad\quad \bm x^*\in Z_q.\label{eq_alsoxs_q_c}
	\end{align}
\end{subequations}  
Under type $q-$Wasserstein ambiguity set with $q\in[1,\infty)$, we show that $\alsoxs$ is better than $\CVaR$ approximation.

\alsosharpbettercvarqwass*
\begin{proof}
	For a given objective upper bound $t$, if the solution of lower-level $\CVaR$ approximation \eqref{eq_cvar_q_b} satisfies \eqref{eq_cvar_q_c}, it is feasible to the DRCCP. Since the lower-level $\alsoxs$ \eqref{eq_alsoxs_q_b} and the lower-level $\CVaR$ approximation \eqref{eq_cvar_q_b} coincide, then $\alsoxs$ finds a feasible solution to the DRCCP if the $\CVaR$ approximation is able to find one. This completes the proof.
    \QEDA
\end{proof}

Similarly, we introduce the $\alsoxus$ by dropping the constraint $\beta\leq 0$ in the lower-level $\alsoxs$ \eqref{eq_alsoxs_q_b}, which has the following formulation:
\begin{subequations}
\label{eq_alsoxus_q}
\begin{align}
v_q^{\aus} &=\min _{ {t}}\, t, \label{eq_alsoxus_q_a} \\
\textup{s.t.} &\quad(\bm x^*,\lambda^*,\beta^*)\in\argmin _{\begin{subarray}{c}\bm {x}\in \X, \bm{c}^\top \bm{x} \leq t,\\ \lambda\geq 0, \beta \end{subarray}}\,\left\{ \varepsilon\beta+ \lambda\theta^q+\E_{\Pr_{\trzeta}}\left[  \max_{ i\in[I]}\left(\bm a_i(\bm x)^\top \trzeta-b_i(\bm{x}) +P_{q,i}(\bm x,\lambda)-\beta\right)_+\right]\right\},\label{eq_alsoxus_q_b}\\
&  \quad\quad \bm x^*\in Z_q.\label{eq_alsoxus_q_c}
\end{align}
\end{subequations} 

Following the similar proof in Theorem~\ref{also_sharp_better_weak_sharp}, we can prove that $\alsoxs$ is better than  $\alsoxus$ under type $q-$Wasserstein ambiguity set with $q\in[1,\infty)$.

 \alsosharpbetterweaksharpqwass*
 \begin{proof}
 It is sufficient to show that for a given objective upper bound $t$, if the lower-level $\alsoxus$ \eqref{eq_alsoxus_q_b} yields a feasible solution to DRCCP, i.e., satisfying \eqref{eq_alsoxus_q_c}, then the lower-level $\alsoxs$ \eqref{eq_alsoxs_q_b} will also provide a feasible solution. Let $(\hat{\bm x}, \hat{\beta})$ denote an optimal solution from the lower-level  $\alsoxs$ \eqref{eq_alsoxs_q_b} and let $(\bar{\bm x}, \bar{\beta})$ denote an optimal solution from the lower-level $\alsoxus$ \eqref{eq_alsoxus_q_b}. Suppose that $\bar{\bm x}$ is feasible to DRCCP \eqref{drccp_q_general}, i.e., $\bar{\bm x}\in Z_q$  \eqref{q_cvar_interpretation_cvar}. Now let $\bar{\Pr}^*$ denote the worst-case distribution of $\trxi$ in the lower-level $\alsoxus$ \eqref{eq_alsoxus_q_b} and let 
\begin{align*}
\bar{\beta}^*: = \bar{\Pr}^*\text{-}\VaR_{1-\varepsilon} \left\{ \max_{i\in [I]} \bm a_i(\bar{\bm x})^\top \trxi- b_i(\bar{\bm x}) \right\}\leq  \sup_{\Pr\in \P}\Pr\text{-}\VaR_{1-\varepsilon} \left\{  \max_{i\in [I]} \bm a_i(\bar{\bm x})^\top \trxi- b_i(\bar{\bm x}) \right\}\leq 0.
\end{align*}
Then according to theorem 1 in \cite{rockafellar2000optimization} (see, e.g., equation (7) in \cite{rockafellar2000optimization}) and the discussions in Theorem~\ref{also_sharp_better_weak_sharp}, we have that $(\hat{\bm x}, \bar{\beta}^*)$ is another optimal solution to the lower-level $\alsoxus$ \eqref{eq_alsoxus_q_b}. Since the only difference between the lower-level $\alsoxs$ \eqref{eq_alsoxs_q_b} and the lower-level $\alsoxus$ \eqref{eq_alsoxus_q_b} is the constraint $\beta\leq 0$, with the assumption that the lower-level $\alsoxus$ \eqref{eq_alsoxus_q_b} admits a unique optimal solution of $\bm x$, we must have $\bar{\bm x}=\hat{\bm x}$. That is, for a given objective upper bound $t$, both lower-level problems have the same optimal value and optimal $\bm x$-solution.  
This implies that the lower-level $\alsoxs$ \eqref{eq_alsoxs_q_b} yields a feasible solution to DRCCP. 
      \QEDA
 \end{proof}

We remark that the uniqueness of the lower-level $\alsoxus$ can be achieved in many DRCCPs. For example, one condition is that  the affine mappings 
 are $\bm{a}_1(\bm{x})=\bm x$, $b_1(\bm x)=b_1$, the random parameters $\trzeta$ follow a joint elliptical distribution with $\trzeta\thicksim\Pr_{\mathrm{E}}(\bm{\mu},\rsigma,\hat g)$, and the norm defining the Wasserstein distance is the generalized Mahalanobis norm associated with the positive definite matrix $\rsigma$. Under this setting, the lower-level $\alsoxus$ can be written as
\begin{align*}
 \bm x^*\in\argmin_{\bm{x}\in \X,\bm{c}^\top \bm{x} \leq t} \left\{   \overline{F}(\bm x)\colon=\bm \mu^\top \bm x +
\left[\overline{G}\left(\left(\mathrm{\Phi}^{-1}(1-\varepsilon)\right)^2/2\right)/\varepsilon+\theta\varepsilon^{-\frac{1}{q}}\right]\sqrt{\bm x^\top\rsigma\bm x} -b_1 \right\}, 
\end{align*}
and its first-order and second-order derivatives are
\begin{align*}
   &  \frac{\partial \overline{F}(\bm x)}{\partial \bm x} = \bm \mu^\top + \left[\overline{G}\left(\left(\mathrm{\Phi}^{-1}(1-\varepsilon)\right)^2/2\right)/\varepsilon+\theta\varepsilon^{-\frac{1}{q}}\right] \rsigma\bm x^\top,\\
    & \frac{\partial^2 \overline{F}(\bm x)}{\partial \bm x^2} = \left[\overline{G}\left(\left(\mathrm{\Phi}^{-1}(1-\varepsilon)\right)^2/2\right)/\varepsilon+\theta\varepsilon^{-\frac{1}{q}}\right] \rsigma\succ \bm 0.
\end{align*}
Hence, the lower-level $\alsoxus$ admits a unique solution whenever set $\X$ is convex.

We conclude this section by providing theoretical comparisons among the output objective values of  $\alsoxs$, $\alsoxus$, $\alsox$, and $\CVaR$ approximation under type $q-$Wasserstein ambiguity set with $q\in[1,\infty)$, which are shown in Figure~\ref{summary_figure_general}.
\vspace{-2em}
\begin{figure}[htbp]\centering
\small
	\begin{tikzpicture}[framed,
		every path/.style={>=latex},
		every node/.style={draw}
		]
		\node[rectangle, align=left] (sharp) { $\alsoxs$ };
		\node[rectangle, left = of sharp, align=left] (drccp) { DRCCP };
		\draw [dashed,-,circle] (drccp) -- (sharp); 
		\draw [dashed,-,circle,draw opacity=0] (drccp) -- (sharp) node[above,midway,sloped]{$\leq$}; 
		
		\node[rectangle, below right= of sharp, align=left] (wsharp) {$\alsoxus$ };
		\node[rectangle, above right= of sharp, align=left] (alsox) { $\alsox$};
		
		\node[rectangle, above right= of wsharp, align=left] (cvar) {$\CVaR$ Approximation};  
		
		\draw [dashed,-,circle] (sharp) -- (cvar); 
		\draw [dashed,-,circle,draw opacity=0] (sharp) -- (cvar) node[above,midway,sloped]{$\leq$}; 
		\draw [dashed,-,rectangle,draw opacity=0] (sharp) -- (cvar) 
		node[below,sloped,pos=0.5]{ Proposition~\ref{also_sharp_better_cvar_q_wass}};
		
				\draw [dashed,-,circle] (sharp) -- (wsharp); 
		\draw [dashed,-,circle,draw opacity=0] (sharp) -- (wsharp) 	
node[above,midway,sloped]{$\leq$}; 
\draw [dashed,-,rectangle,draw opacity=0] (sharp) -- (wsharp) 
node[below,sloped,pos=0.36]{ Proposition~\ref{also_sharp_better_weak_sharp_q_wass}};
		
		\draw [dashed,-,circle] (sharp) -- (alsox); 
		\draw [dashed,-,circle,draw opacity=0] (sharp) -- (alsox)
		node[left,pos=0.7,sloped]{\LARGE x};  
%
		\draw [dashed,-,circle] (alsox) -- (cvar); 
		\draw [dashed,-,circle,draw opacity=0] (alsox) -- (cvar)
		node[left,pos=0.6,sloped]{\LARGE x}; 
		
		\draw [dashed,-,circle] (wsharp) -- (cvar); 
	\draw [dashed,-,circle,draw opacity=0] (wsharp) -- (cvar)
	 node[above,midway,sloped]{$\leq$}; 
		\draw [dashed,-,rectangle,draw opacity=0] (wsharp) -- (cvar)
		node[below,sloped,pos=0.65]{Theorem~\ref{also_sharp_better_cvar} };
		
		\draw [dashed,-,circle] (wsharp) -- (alsox); 
		\draw [dashed,-,circle,draw opacity=0] (wsharp) -- (alsox) node[left,pos=0.65,sloped]{\LARGE x}; 

   	\matrix [densely dotted,draw opacity=0.5,below] at (6,-1) {
			\node[font=\fontsize{1pt}{2pt},draw opacity=0] { X: non-comparable}; \\
		};

		
	\end{tikzpicture}
	\vspace{-0.5em}
	\caption{Summary of Comparisons under Type $q-$Wasserstein ambiguity set with $q\in[1,\infty)$}
	\label{summary_figure_general}
	\vspace{-2em}
\end{figure}
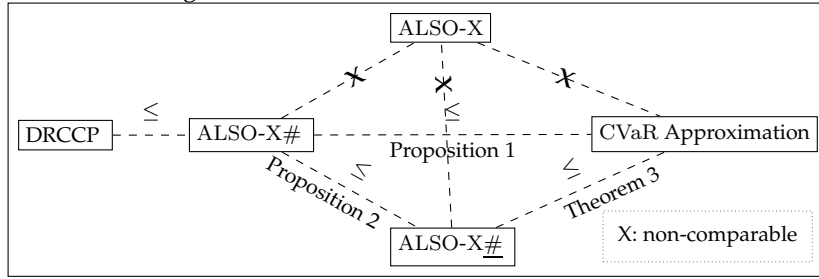 
	\vspace{-1.1em}

\end{appendices}

\end{document}